\documentclass[11pt,a4paper,reqno,pdflatex]{amsart}
\usepackage[margin=1in]{geometry} 

\usepackage{amsfonts,amsmath,graphics,epsfig,latexsym,times,mathabx}
\usepackage{ae,aecompl}
\newif\ifpix \pixtrue
\usepackage{setspace,fancyhdr,amssymb,amsthm,graphicx,ifthen,float,calrsfs,framed} 
\usepackage{hyperref,url}
\hypersetup{
    colorlinks,
    linkcolor={red!70!black},
    citecolor={blue!70!black},
    urlcolor={blue!80!black}
}
\usepackage{tikz-cd}
\usepackage{hyperref}
\usepackage{tikz,pgf}
\numberwithin{equation}{section} 
\newtheorem{thm}{Theorem}[section]  
\newtheorem{cor}[thm]{Corollary}    
\newtheorem{lem}[thm]{Lemma}        
\newtheorem{prop}[thm]{Proposition}  

\theoremstyle{definition} \newtheorem{dfn}[thm]{Definition}
\newtheorem{notn}[thm]{Notation}\newtheorem{rmk}[thm]{Remark}

\newtheorem*{claim*}{Claim} 
 
\newtheorem{ex}[thm]{Example}

\allowdisplaybreaks[1]

\newcommand{\la}{\langle}
\newcommand{\ra}{\rangle}
\newcommand{\cB}{\mathcal{B}}
\newcommand{\cH}{\mathcal{H}}

\newcommand{\ol}{\overline}

\newcommand{\cR}{\mathcal{R}}
\newcommand{\cD}{\mathcal{D}}

\newcommand{\bom}{\boldsymbol{\omega}}

\newcommand{\ba}{\boldsymbol{a}}
\newcommand{\bbb}{\boldsymbol{b}}

\newcommand{\bc}{\boldsymbol{c}}

\newcommand{\bze}{\boldsymbol{\zeta}}
\newcommand{\bu}{\boldsymbol{u}}
\newcommand{\bg}{{\mathbf g}}
\newcommand{\bh}{{\mathbf h}}
\newcommand{\bG}{{\mathbf G}}
\DeclareMathOperator{\ei}{ei}
\DeclareMathOperator{\eb}{eb}
\DeclareMathOperator{\ff}{ff}

\DeclareMathOperator{\SO}{SO} 

\DeclareMathOperator{\ad}{ad}
\DeclareMathOperator{\AH}{AH}
 
\DeclareMathOperator{\TN}{TN}

\DeclareMathOperator{\bo}{b}

\newcommand{\rd}{{\mathrm d}}

\def\so{{\mathfrak{s}\mathfrak{o}}}
\newcommand{\p}{\partial}

\DeclareMathOperator{\supp}{supp}

 \DeclareMathOperator{\Diff}{Diff}

\DeclareMathOperator{\SL}{SL} 

\newcommand{\CC}{\mathbb{C}}
\DeclareMathOperator{\tr}{tr} 
\newcommand{\RR}{\mathbb{R}} \newcommand{\ZZ}{\mathbb{Z}}

\newcommand{\TT}{\mathbb{T}}

\newcommand{\cL}{\mathcal{L}}
\newcommand{\cF}{\mathcal{F}}
\newcommand{\cU}{\mathcal{U}}

\newcommand{\cV}{\mathcal{V}}

\newcommand{\cW}{\mathcal{W}}

\newcommand{\wt}{\widetilde}

\newcommand{\wh}{\widehat}

\newcommand{\ve}{\varepsilon}
\newcommand{\vc}[1]{\boldsymbol{#1}}

\renewcommand{\geq}{\geqslant}
\renewcommand{\leq}{\leqslant}

\renewcommand{\epsilon}{\varepsilon}

\newcommand{\hlf}{\frac12}

\def\bs{\mathbf{s}}

\def\tx{\tilde{x}}

\newcommand{\CP}{\mathbb{CP}}
\newcommand{\RP}{\mathbb{RP}}
\def\CPD{\CP_1^{\text{\tiny diag}}}
\def\CPAD{\CP_1^{\text{\tiny adiag}}}
\DeclareMathOperator{\SEN}{Se}
\DeclareMathOperator{\AHr}{HA}
\def\AHu{\widehat{\AH}}
\def\AHr{\text{HA}}
\def\AHd{\mathcal{M}^0_2}
\def\bee{\begin{equation}}
\def\eee{\end{equation}}

\author{B.J. Schroers}
\address{Department of Mathematics,   Heriot-Watt University and  Maxwell Institute for Mathematical Sciences}
\email{b.j.schroers@hw.ac.uk}
\author{M.A. Singer}
\address{Department of Mathematics, University College London}
\email{michael.singer@ucl.ac.uk}
\date{6 April 2020}
\title[$D_k$ Gravitational Instantons]{$D_k$ Gravitational Instantons as superpositions of    Atiyah--Hitchin  and Taub--NUT geometries}
\begin{document}
\begin{abstract}
We obtain $D_k$ ALF gravitational instantons by a gluing construction which captures, in a precise and explicit fashion,  their interpretation as non-linear superpositions of the  moduli space of centred  $SU(2)$ monopoles, equipped with the Atiyah--Hitchin metric,  and $k$ copies of the Taub--NUT manifold. The construction proceeds from a finite set of points in euclidean space, reflection symmetric about the origin, and depends on an adiabatic parameter  which is incorporated into the geometry as a  fifth dimension.  Using a formulation in terms of hyperK\"ahler triples on manifolds with boundaries, we show that the constituent Atiyah--Hitchin and Taub--NUT geometries arise as boundary components of the 5-dimensional geometry as the adiabatic parameter is taken to zero.
\end{abstract}
\maketitle

\begin{center}
{ \small To appear in a special issue of   The Quarterly Journal of Mathematics  \\
dedicated to Sir Michael Atiyah}
\end{center}

\section{Introduction and conclusion}

\subsection{First statement of the main result}
A 4-dimensional gravitational instanton is a complete hyperK\"ahler
4-manifold $(M,g)$, possibly with a decay condition on the curvature
at infinity.  Michael Atiyah was fascinated by gravitational
instantons from the early 1980s onwards, and much progress was made by
the Oxford group, led by Sir Michael, until he left for the mastership
of Trinity College Cambridge in 1990.  In particular, his student
Peter Kronheimer, building on the work of Nigel Hitchin (e.g.\
\cite{poly_grav}) and others, gave a complete classification of the
asymptotically locally euclidean (ALE) gravitational instantons,
using the hyperK\"ahler quotient construction
\cite{Kronheimer86,Kronheimer89}.  At about the same time, Atiyah and
Hitchin computed the metric on the moduli space $\AHd$ of centred
$SU(2)$  monopoles, which  is an example of an asymptotically locally
flat (ALF) gravitational instanton \cite{AtiyahHitchin85,AHbook}. 

The classification of  ALF gravitational instantons has proved to be
more difficult, but, following substantial progress
\cite{CherkisKapustin98,CherkisKapustin99,CherkisHitchin05,Minerbe10,Minerbe11,Auvray18,
  ChenChen19} which we review below,  is now quite well understood.
In particular, there are two infinite families, the $A_k$ and $D_k$
ALF gravitational instantons, labelled by a non-negative integer $k$
and distinguished by the fundamental group of the asymptotic region of
$M$.   The $A_k$ ALF gravitational instantons can all be constructed by
the Gibbons--Hawking Ansatz \cite{GibbonsHawking78,Minerbe11}. In
particular, the $A_0$ gravitational instanton is the euclidean (positive mass)
Taub-NUT space which we denote by $\TN$ in the following.

Constructions
of $D_k$ gravitational instantons are not so explicit.  The early
paper \cite{DaDi} gave a construction using Nahm's equations of hyperK\"ahler metrics on
$4$-manifolds with the correct $D_k$ asymptotic topology, but did not
prove the ALF property.  Other constructions 
either use twistor theory
\cite{CherkisKapustin98,CherkisKapustin99,CherkisHitchin05}  or  rely
on gluing or desingularization constructions
\cite{BiquardMinerbe11, Auvray18}.   In this paper we shall present a
construction in which $D_k$  ALF gravitational instantons appear
as (nonlinear) superpositions of $\AHd$ and $k$ copies of
$\TN$.

The idea of this construction can be found in a paper
  of Ashoke Sen \cite{Sen97}:  the Gibbons--Hawking Ansatz, applied to
  a harmonic function of the form
  \bee\label{e31.21.11.20}
  V = 1 - \frac{2}{|x|} + \sum_{\nu=1}^k\left(\frac{1}{2|x- q_\nu|}+
    \frac{1}{2|x+ q_\nu|}\right),
  \eee
  where the $q_\nu$ are points of $\RR^3$ widely separated from each
  other and from $0$, yields a hyperK\"ahler manifold $(M_q,g_q)$, with ALF
  asymptotics, but which is incomplete near $0$ due to the negative
  coefficient of $1/|x|$. The function $V$ is symmetric under
  $x\mapsto -x$ and this is covered by an orientation-preserving
  isometry $\iota$ of $M_q$.  The geometry of $M_q/\iota$ for $|x| \sim R$
  is well approximated by the {\em asymptotic geometry} of $\AHd$ (see
  below),  provided that $R$ is large but is much less than the smallest of the
  $|q_{\nu}|$.   Denote by $\SEN_k$ the $4$-manifold obtained by
  gluing $\AHd$ into the `hole' near $x=0$ in $M_q/\iota$.  (The
  construction will be described carefully later,  cf.\
  \S\ref{improved} and \S\ref{s1.28.11.20}).    This is a smooth
  $4$-manifold with a metric $g^\chi$ obtained by gluing $g_q$ to
  $g_{\AH}$.  In particular $g^\chi$ is approximately hyperK\"ahler
  on $\SEN_k$.   A first statement of the Theorem to be proved is as follows:  
  \begin{thm}
Let $p_{1},\ldots,p_k \in \RR^3$ be such that $\{0,\pm p_1,\ldots,\pm
p_k\}$ is a set of $2k+1$ distinct points.   Define $q_{\nu} =
p_{\nu}/\ve$. Then there exists $\ve_0>0$, so that for $\ve \in
(0,\ve_0)$, 
      there is a small perturbation $g_{\SEN,\ve}$ of $g^{\chi}$ such that
      $(\SEN_k,g_{\SEN,\ve})$ is a $D_k$ ALF gravitational instanton.
\label{version0}    
\end{thm}
Since the geometry of $g_q$ is approximately that of the Taub--NUT
metric for $|x \pm q_\nu| \sim R$, for $R$ large but much less than
$1/\ve$, this result already justifies the title of the paper:
$(\SEN_k,g_{\SEN,\ve})$ is a $D_k$ ALF gravitational instanton appearing
as a (nonlinear) superposition of the Atiyah--Hitchin and Taub--NUT
geometries.   We shall, however, prove a much more precise version of
this result (Theorem~\ref{mainthm} below), which shows that
$g_{\SEN,\ve}$ is, in a suitable sense, {\em smooth in $\ve$}
uniformly down to $\ve=0$.

In the remainder of this extended introduction we explain  the relation of this result to  Michael Atiyah's  interest in geometrical  models of matter,  provide  some  technical background and use it to state a   more detailed  version of the theorem.

\subsection{Motivation}
 \label{motisec}
Our interest in  Theorem~\ref{version0}  has its origin in a  speculative proposal  for  purely geometric models of  physical  particles made  in \cite{AMS}  by Michael Atiyah, Nick Manton and the first  named   author of the current paper. While our main concern here is  geometry,   we briefly recall the  physical motivation.

 The idea developed in \cite{AMS} is to use non-compact hyperK\"ahler
 4-manifolds  to model electrically charged particles like the
 electron or the proton.  Outside a  compact core region, or  at least
 asymptotically, the 4-manifolds are required to be circle fibrations
 over physical  3-dimensional space. In this asymptotic region, the
 model is interpreted as a dual Kaluza-Klein picture:  the Chern class
 of the asymptotic circle bundle, which would be the magnetic charge
 in Kaluza-Klein theory,  is taken to represent the negative of the
 electric charge.  The further requirement that the 4-manifold has
 cubic volume growth means that the  allowed geometric models are in
 effect ALF gravitational instantons \cite{ChenChen19}. 
 
 In \cite{AMS}, these ideas were illustrated with two main examples, namely   the Taub--NUT   and   Atiyah--Hitchin manifolds  as potential models of, respectively,  the electron and the proton.  Following the convention of \cite{AMS} we  write  $\AH$ for the Atiyah--Hitchin manifold by which we mean  the  simply-connected double cover of the moduli space $\AHd$ of centred 2-monopoles in critically coupled  $SU(2)$ Yang-Mills-Higgs theory \cite{AHbook}.

Geometries which are obtained by gluing together copies of $\TN$ and
of $\AH$ or $\AHd$  are potential geometric models for  electrons
interacting with each other and a proton, and therefore interesting
arenas for exploring  if and how  geometrical models can make contact
with physics beyond basic quantum numbers like electric charge and
baryon number. In  particular, the model for a single electron
interacting with the proton would need to  account for the formation
of the hydrogen atom and its excited states. 
   
   The gluing process is well-understood when dealing only with copies of $\TN$, where it leads to the multi-center  Taub--NUT spaces  which make up  the $A_k$  series of ALF gravitational instantons, with the positive integer $k+1$ counting the number of centres or `NUTs'\cite{GibbonsHawking78,Minerbe11}.
 
 However, the interpretation of  $D_k$  ALF gravitational instantons, even in some asymptotic region, as a composite of more elementary geometries is less clear. This is the issue addressed by Theorem~\ref{version0}, building on the procedure first outlined by Sen.
 While Sen's proposal was made in the context of  $M$-theory, it is similar in spirit to the motivation coming from geometric models of matter.  In both cases one aims to obtain  a  $D_k$ space   as a non-linear superposition  of  $\AHd$  and $k$ copies of $\TN$, thus interpreting it  as a  composite object or bound state. 
 
The construction has two main ingredients, the Gibbons--Hawking
gravitational instanton $(M_q,g_q)$ (see \eqref{e31.21.11.20})
with the additional symmetry $\iota$, and 
and a further manifold,  obtained as a $\ZZ_2$-quotient of a
 branched cover $\AHu$  of  the  Atiyah--Hitchin space $\AH$,  which
 we call $\AHr$.  We now discuss these  in turn, but
 should alert the reader that, while $\AH$ and $\AHd$ have smooth
 hyperK\"ahler metrics, the lifts of these metrics to the branched
 covers $\AHu$ and $\AHr$ are singular on the branching locus.

\subsection{The adiabatic Gibbons--Hawking Ansatz}
\label{AGH}
  The definition of ALF gravitational instantons 
allows for the complement of all sufficiently large compact subsets  to have a non-trivial fundamental group
$\Gamma$.  Apart from a few
exceptional cases,  $\Gamma$ must be a finite subgroup of
$SU(2)$, more specifically a cyclic group $\ZZ_\ell$ or the binary
dihedral group ${\mathcal D}_{\ell}$ of order $4\ell$, for a suitable
positive integer $\ell$.  The corresponding ALF gravitational instantons are called $A_{\ell-1}$  and $D_{\ell+2}$ ALF gravitational instantons. In fact, it is natural to extend this correspondence to $D_k$ instantons for non-negative integers $k$ as follows.

To fix notation, our  presentation of ${\mathcal D}_\ell$ as a subgroup of $SU(2)$ is as the group 
generated by
\begin{equation}\label{e1.30.9.19}
  R_\ell=\begin{pmatrix} e^{-i \frac{ \pi}{ \ell}} & 0 \\ 0 & e^{i\frac{\pi}{ 
      \ell}} \end{pmatrix},\;\;
S=  \begin{pmatrix} 0 & -1 \\ 1 & \phantom{-}0 
  \end{pmatrix},\quad  \ell \geq 1, 
\end{equation}
so that ${\mathcal D}_1\simeq \ZZ_4$ and ${\mathcal D}_2$ is the lift
of the Vierergruppe, viewed as the group of   rotations  by  $\pi$
around orthogonal axes in $\RR^3$,  to $SU(2)$. For our purposes it
is convenient to define also $\mathcal D_0$ as the  infinite
group with generators $R_0,S$ and relations $SR_0S^{-1}=R_0^{-1}$ and
$S^4=\text{id}$. With these conventions, the   fundamental  group of
the asymptotic region  of  $D_k$ ALF gravitational instantons  is
${\mathcal D}_{k^*}$, where 
\begin{equation}
\label{kstar}
k^*=|k-2|.
\end{equation}

We now give a more detailed account of the hyperK\"ahler manifolds
$(M_q,g_q)$ which appear Theorem~\ref{version0}.  It is convenient to
fix $p_{\nu}$ as in the statement of that theorem, set $q_{\nu} =
p_{\nu}/\ve$ and also replace $x$ by $x/\ve$, so that $V$ becomes
\bee \label{e1.9.1.20}
h_{\ve}(x) = 1  -\frac{2\ve}{|x|} + \sum_{\nu=1}^k 
\left(\frac{\ve}{2|x-p_{\nu}|} + \frac{\ve}{2|x+p_{\nu}|} \right) .
\eee
This function yields an {\em adiabatic family of Gibbons--Hawking
  metrics}
\begin{equation}\label{e1.22.8.18}
  g_{\ve} = h_\ve\frac{|\rd x|^2}{\ve^2}  + h_{\ve}^{-1}\alpha^2,\;\; 
  \rd \alpha = *_{\ve}\rd h_{\ve}, 
\end{equation}
on $4$-manifolds $M_{\ve}$ which carry 
a circle-action with quotient $\{|x| > 2\ve\} \subset \RR^3$.
Denoting the quotient map by $\phi$ and putting
\bee\label{e1.26.11.20}
P = \{\pm p_1,\ldots, \pm p_k\},
\eee
the action is free away from $\phi^{-1}(P)$, the fixed-point set of
the $S^1$-action on $M_{\ve}$.    Then $M_{\ve}\setminus \phi^{-1}(P)$
is the total space of a principal $S^1$-bundle over $Y_{\ve} =
\{|x|>2\ve\} \setminus P$, and 
the precise
interpretation of $\alpha$ in \eqref{e1.22.8.18} is as a
connection-form on this bundle.  Thus in local
coordinates $(x,\theta)$ on $\phi^{-1}(Y_{\ve})$, $\alpha =
\rd\theta +\mbox{($1$-form on base)}$, and $\rd \alpha$,
the curvature of $\alpha$, is a $2$-form on $Y_{\ve}$.  The
second  equation of
\eqref{e1.22.8.18} thus makes sense with 
$*_\ve$ denoting the Hodge star operator of the metric $|\rd
x|^2/\ve^2$ (and a fixed orientation) on $\RR^3$.

The involution $x\mapsto -x$ preserves $h_{\ve}$ and 
is covered on $M_{\ve}$ by an orientation-preserving involution
$\iota$.  
Then $\alpha$ can (and will) be chosen so that $\iota^*\alpha =
-\alpha$;  $\iota$ is then an isometry of $g_{\ve}$.  Since $\iota$
acts freely on $M_{\ve}$, the quotient
$(M_{\ve}/\iota, g_{\ve})$ is a smooth hyperK\"ahler manifold with
$D_k$ ALF asymptotics, but it is
incomplete near $0$ due to the restriction to $|x|>2\ve$.   To verify
the claim about the asymptotics, note that 
for $|x|\gg 1$, 
\bee\label{e1.22.10.20}
h_\ve(x) = 1 + \frac{\ve(2k-4)}{2|x|} + O(\ve|x|^{-3}) =
1  + \frac{2k^*\ve}{2|x|} + O(\ve|x|^{-3}).
\eee
(There is no $O(\ve|x|^{-2})$ term because $h_{\ve}(-x)=h_{\ve}(x)$.) 
This means that the asymptotic topology  of $M$ is $\RR^3 \times S^1$
if $k^*=0$,  and $\RR^4/\la R_{k^*}\ra$ otherwise ($R$ is defined in
\eqref{e1.30.9.19}). Our involution $\iota$ corresponds to the
generator $S$ above acting on 
$\RR^4/\la R_{k^*}\ra$ so that $(M_{\ve}/\iota, g_{\ve})$ does indeed
have the correct asymptotic geometry for a  $D_k$ ALF gravitational
instanton. 

As we have indicated, the choice of leading term $1- \frac{2\ve}{|x|}$ is made
so that the hole can be filled by gluing in $\AHd$ with the
Atiyah--Hitchin metric.
We therefore turn next
to various aspects of its geometry.  
  
\subsection{The Atiyah--Hitchin manifold and related spaces}
\label{AHreview}

Michael Atiyah revisited the geometry of   the Atiyah-Hitchin manifold on several occasions. Even though it arose in the specific physical context of magnetic  monopoles, he hoped  for an  application to real and fundamental physics,    and pursued this in   the Skyrme model of nuclear particles \cite{AtiyahManton} and  in geometric models of matter \cite{AMS}. In all these studies, he stressed and used the interpretation of  $\AH$ and its branched cover   as parameter spaces of oriented ellipses, up to scale, in euclidean space.  

We have also found this picture  helpful, and develop it further in this section and  Appendix \ref{cpapp} in order to  clarify the discrete symmetries and their action on the core and asymptotic regions.
We begin by noting that, as real manifolds, 
\bee
\label{coverdef}
\AHu= TS^2\simeq  \CP_1\times \CP_1\setminus  \CPAD,
\eee
where $\CPAD$ is the anti-diagonal  $\CP_1$ in $ \CP_1\times \CP_1$. This manifold 
  is a  branched  cover  of the Atiyah--Hitchin manifold $\AH$, which, as already explained,  is the double cover of the moduli space $\AHd$ of centred 2-monopoles.  We would like to make this explicit, and  to define the manifold $\AHr$ in terms of $\AHu$.

In Appendix \ref{cpapp}, we derive the  concrete realisation of $\AHu$  as 
\bee
\label{AHuexplicit}
\AHu = \{Y\in \CC^3| Y_1^2 + Y_2^2 +Y_3^2=1\},
\eee
where we  wrote $Y $ for the vector in $\CC^3$ with coordinates $Y_1,Y_2$ and $Y_3$.   The real and imaginary parts of 
$Y= y+i\eta$ are orthogonal, with magnitudes related via $|y|^2= 1
+|\eta|^2$. We can picture this description in terms of an oriented
ellipse, called the $Y$-ellipse in the following,   with major axis
$y$ and minor axis $\eta$.  When $|\eta|=0$ the $Y$-ellipse
degenerates to an oriented line. The set   of  these lines is a
two-sphere  to which $\AHu$ retracts and which we call the {\em core} in the
following.  It is the diagonal submanifold of $\CP_1\times \CP_1$,
and we denote it by $\CPD$. 

This description of $\AHu$ is useful  for understanding its symmetries and the structure near the core, but less useful when studying the asymptotic region away from the core, which for us means simply  $|\eta| \neq0$. In this region  it is convenient to switch to a dual description, derived in the appendix, in terms of a  complex vector  $X$ whose components also satisfy $X_1^2 + X_2^2 +X_3^2=1$,  but whose real and imaginary parts are 
\bee
X= \tx + i\eta, \qquad \tx = \frac{y\times  \eta}{|\eta|^2}, \quad \xi =   - \frac{ \eta}{ |\eta|^2}.
\eee
One checks that  $|\tx|^2=1+|\xi|^2$, and in Appendix \ref{cpapp} we explain that $\tx$ and $\xi$ are the major and minor axes of a  family of ellipses which we call $X$-ellipses and  which are dual to the $Y$-ellipses. 

The $X$-ellipses 
degenerate into oriented  lines in the direction of $\tx$  when $|\xi|=0$.  The directions of these lines make up  the sphere at spatial infinity in the asymptotic region of $\AHu$, which is $\CPAD$ in the description \eqref{coverdef}. 
The core  $\CPD$ of $\AHu$   is obtained in another degenerate limit of the $X$-ellipses, namely in  the limit $|\xi|\rightarrow \infty$, where they become circles of infinite radius.  

We are interested in the quotients of $\AHu$ by discrete symmetries which arise naturally from its description as  $\CP_1\times \CP_1\setminus  \CPAD$, namely the factor switching map $s$,  the antipodal maps on both factors $a$ and the composition $r=as$.  It follows from the description of these maps in Appendix \ref{cpapp}, that they act in the following way on the ellipse parameters, where the formulation in terms of $(\tx,\xi)$ assumes that $|\eta|\neq 0$:
\begin{align}
\label{sra}
s&: (y,\eta)\mapsto (y,-\eta), \qquad (\tx,\xi)\mapsto (-\tx, -\xi), \nonumber \\
r&: (y,\eta)\mapsto (-y,-\eta), \qquad (\tx,\xi)\mapsto (\tx, -\xi), \nonumber \\
a&: (y,\eta)\mapsto (-y,\eta), \qquad (\tx,\xi)\mapsto (-\tx, \xi).
\end{align}
In particular we see that $s$ fixes the core $\CPD$ but acts as the antipodal map on the sphere at spatial infinity $\CPAD$, while $r$  fixes the sphere at spatial infinity and acts as the antipodal map on the core. Writing $1$ for the identity map   and defining the Vierergruppe
\bee
\text{Vier} =\{1,s,r,a \},
\eee  
 we can characterise the Atiyah--Hitchin manifold  $\AH$ and  the moduli space $\AHd$  of centred 2-monopoles as  the quotients
 \bee
 \label{AHasquot}
 \AH =\AHu / s , \quad \AHd = \AH/r = \AHu/\text{Vier}.
\eee
It follows from  our discussion of the generators, that in  $\AH$ the core is still a two-sphere, but  the space of directions at spatial infinity   is now  $\CPAD/\ZZ_2 \simeq \RP_2$.  Finally, in $\AHd$  both the core and the  space of directions at spatial infinity   are isomorphic to $\RP_2$.

The manifold obtained by quotienting $\AHu$ by the free action of $r$ is, literally,  central to the construction of  the Sen spaces.  We therefore define
\bee
\AHr = \AHu/ r. 
\eee
This manifold still has a  two-sphere of directions at  spatial infinity, but its core is  isomorphic to  $\RP^2$. 
It allows us to write the moduli space $\AHd$  of centred 2-monopoles  also as the quotient 
\bee
\label{d0}
 \AHd = \AHr/s,
\eee
and this is precisely what we require for our construction. 
The situation is summed up in Fig.~\ref{AHdiag}.

\begin{figure}[h]
\begin{center}
\begin{tikzcd} 
   &\arrow[dl, "\sim_s"]  \AHu  \arrow [dr, "\sim_r"']&   \\
  \AH \arrow [dr,"\sim_r"'] &&\arrow[dl,"\sim_s"]  \AHr   \\
& \AHd  & 
\end{tikzcd}
\end{center}
\caption{Coverings and quotients of the Atiyah--Hitchin manifold}
\label{AHdiag}
\end{figure}
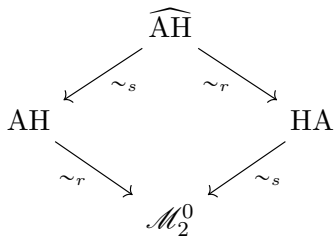

Having defined the manifolds, we turn to their symmetries and metric structure. 
The rotation group $SO(3)$ acts on all four manifolds in
Fig.~\ref{AHdiag} by the obvious action of  $G\in SO(3)$ on $X\in
\CC^3$. This action   commutes with the action of the Vierergruppe
\eqref{sra}, so that the generic $SO(3)$ orbit is $SO(3)$ for $\AHu$,
$SO(3)/\ZZ_2$ for both $\AH$ and $\AHr$ and   $SO(3)/\text{Vier}$ for $\AHd$, with the generators $s,r $ and $a$  realised as  rotations by $\pi$ around three orthogonal axes.
Away from the core $\CPD$, we have a $U(1)$ action which commutes with the $SO(3)$ action and which fixes the asymptotic  direction $m= \tx /|\tx|$: 
\bee
\label{Uact}
 (e^{i\theta}, (\tx,\xi)) \mapsto 
(\tx, R_m(\theta)\xi),
\eee
where $R_m(\theta)$ is the rotation about $m$  by an angle $\theta\in [0,2\pi)$.

The Atiyah--Hitchin metric is most easily expressed in terms of the $SO(3)$ matrix $G$ and one transversal coordinate $\tau$ (bijectively related to ellipse parameter $|\tx|$).  We define
left-invariant 1-forms on $SO(3)$  via $G^{-1} \rd G = \sigma_1 t_1 +
\sigma_2 t_2 + \sigma_3 t_3$ for generators $t_1,t_2,t_3\in \so(3)$ of
the rotations around  three orthogonal  axes,  satisfying
$[t_i,t_j]=\epsilon_{ijk} t_k$. Then   the   Atiyah--Hitchin metric is  
\bee
\label{AHmetric}
g_{\AH} = f^2 \rd \tau ^2 + a^2 \sigma_1^2 +b^2 \sigma_2^2 + c^2\sigma_3^2,
\eee
where  the choice of $f$ amounts to  fixing the transversal coordinate $\tau$, and the radial functions $a,b $ and $c$ obey coupled differential equations which follow from the hyperK\"ahler property of the metric \cite{AHbook}.   As explained  in \cite{GM},  the choice $f=-b/\tau$ results  in a radial coordinate in the range $\tau\in [\pi,\infty)$,  with $\tau=\pi$ corresponding to the core $\CPD$,  and coefficient functions  $a,b$ and $c$ with the asymptotic  form
\bee
\label{AHasy}
a\sim b\sim \tau  \sqrt{1-\frac 2 \tau }, \qquad c\sim -\frac{2}{\sqrt{1-\frac 2  \tau}},
 \eee
and exponentially small corrections. Substituting the  asymptotic form into \eqref{AHmetric} yields  the  negative-mass Taub-NUT  metric as  the leading term
 \bee
 \label{AHmetricasy}
 g_{\AH} = \left(1-\frac{2}{|x'|}\right) |\rd x'|^2 + \left( 1-\frac{2}{|x'|} \right)^{-1} (\alpha')^2 + O(e^{-|x'|}),
 \eee
where we made  the identifications
\bee
 \frac{x'}{|x'|} = G\begin{pmatrix}  0 \\0 \\ 1\end{pmatrix}, \quad  |x'|=\tau, \quad  \alpha' = 2 \sigma_3.
\eee
In the following we will refer to \eqref{AHmetric} as the Atiyah--Hitchin metric and to \eqref{AHmetricasy} as its asymptotic form regardless of whether the underlying manifold  is $\AHu, \AH, \AHr$ or $\AHd$, even though the metric is singular at  the core on  $\AHu$ and $\AHr$.

We can tie together the asymptotic Taub--NUT geometry with the description of  $\AHu$ (and its quotients)   in terms of the $X$-ellipses by noting that both are $U(1)$ bundles over $\RR^3$, with $x$ and the major axis  $\tx$ being coordinates on the base. The directions of both  $x$ and $\tx$  parametrise the two-sphere  at spatial infinity and can be identified.  The  magnitudes of $\tx$ and $x$ are  bijectively related, but not in any obvious way. 
 
To end this review of the Atiyah--Hitchin geometry we note  
that  the  moduli space  $\AHd$ equipped with the  hyperK\"ahler
metric \eqref{AHmetric} is the, up to scaling, unique $D_0$ ALF
gravitational instanton \cite{AHbook}, and, suitably interpreted,
fits into  the general  construction outlined in \S\ref{AGH} with
$k=0$.  In this case, no gluing is required since  
the  manifold $\AHr$ has the required asymptotic structure, both
topologically and metrically. The quotient \eqref{d0} realises the
division by an involution $\iota =s$ which is covered by the generator
$S$ in \eqref{e1.30.9.19}, see our Appendix \ref{cpapp} for details.  

\subsection{Improved statement of Theorem~\ref{version0}}
\label{improved}  Let the notation be as in \S\ref{AGH}.  In particular,
with $h_{\ve}$ defined as in \eqref{e1.9.1.20}, we have
\begin{equation}\label{e2.9.1.20}
{h}_\ve(x) = 1 + \mu\ve  - \frac{2\ve}{|x|} + O(\ve|x|^2)
\end{equation}
for $x$ in some small ball $B(0,\delta) \subset \RR^3$ and
\bee\label{e12.13.9.20}
\mu =  \sum_{p\in P} \frac{1}{2|p|} =\sum_{\nu=1}^k\frac{1}{|p_{\nu}|}.
\eee
The error term in \eqref{e2.9.1.20} is smooth in $B(0,\delta)$ and we
have $O(\ve|x|^2)$ rather than $O(\ve|x|)$ because ${h}_{\ve}(-x)
= {h}_{\ve}(x)$.  

Fix $\delta\in (0,1/2)$ and take $\ve \in (0,\delta^2)$. If we define
\bee
U_1 = \left\{\frac{\ve}{\delta} < |x| < \delta\right\},\;\; V_1=
  \phi^{-1}(U_1),
\eee
then $V_1$ is a non-empty open subset of $M_{\ve}$. 

We have seen that an asymptotic region of $\AHr$ is the total space of
a circle-bundle, $\psi$, say, over a subset of the form
$\{\delta|x'|>1\}$.   Define
\bee
U_0 = \left\{\frac{1}{\delta} < |x'| <
  \frac{\delta}{\ve}\right\},\;\; V_0 = \psi^{-1}(U_0).
\eee

\begin{dfn}\label{n1.21.11.20}
For $\mu>0$ denote by $g_{\AH,\ve}$ the AH metric which has
the asymptotic form 
\bee\label{e2.10.10.20}
(1+\mu\ve -2/|x'|)|\rd x'|^2 +
(1+\mu\ve -2/|x'|)^{-1}(\alpha')^2
\eee
on $\psi^{-1}\{\delta|x'| >1\}$, cf.\ \eqref{AHmetricasy}.
\end{dfn}

The rescaling $x' = x/\ve$ defines a diffeomorphism  $s:U_1 \to U_0$,
 and matches
up the harmonic function of $x'$ appearing
in \eqref{e2.10.10.20} with the leading terms of $h_{\ve}$ in
\eqref{e2.9.1.20}.    The two circle bundles $\phi$ and $\psi$ have 
the same degree and so the diffeomorphism $s$ can be
covered by a bundle map $\kappa: V_1 \to V_0$, say.  Such $\kappa$ is
far from unique, but it may be 
fixed up to a constant phase by insisting that it matches up the
leading terms of $g_{\ve}$ and $g_{\AH,\ve}$.

From \eqref{e2.9.1.20} it follows that 
\bee\label{e5.10.10.20}
\alpha = \alpha_0  + O(\ve|x|^2),
\eee
where $\alpha_0$ is the standard connection on $V_1$
which satisfies
\bee \label{e6.10.10.20}
\rd\alpha_0 = -*_\ve \rd\left(\frac{2\ve}{|x|}\right).
\eee
(In \eqref{e5.10.10.20}, the error term is a $1$-form on $B(0,\delta)$
whose length, as measured by the metric $|\rd x|^2/\ve^2$, is
$O(\ve|x|^2)$; see Proposition~\ref{primitive2}.)

On the other hand,
\bee\label{e7.10.10.20}
\rd\alpha' = -*\rd \left(\frac{2}{|x'|}\right)
\eee
and the right-hand sides of \eqref{e6.10.10.20} and
\eqref{e7.10.10.20} are matched by our rescaling $s$. It follows that
$\kappa$ can be chosen so that 
$\kappa^*(\alpha_0) = \alpha'$, and this 
condition fixes $\kappa$ up to a constant phase.

\begin{dfn}  The space $\wt{\SEN}_{k}$ (also denoted by
  $\wt{\SEN}_{k,\ve}$ when we need to keep track of the scale $\ve$)
  is obtained from the disjoint union
$$
(\AHr_{\ve} \setminus \psi^{-1}\{|x'| \geq \delta\ve^{-1}\}) \amalg
(M_{\ve} \setminus \phi^{-1}\{|x| \leq \delta^{-1}\ve\})
$$
by identifying $V_0$ with $V_1$ by $\kappa$.

The space $\SEN_k$ (or $\SEN_{k,\ve}$) is the quotient $\wt{\SEN}_{k}/\iota$.
\end{dfn}

\begin{rmk}
This discussion clarifies the gluing referred to in the introduction
to Theorem~\ref{version0}. It would be straightforward to 
perturb $g_{\ve}$ and $\kappa^*g_{\AH,\ve}$  over $V_1$ to yield a
metric $g^{\chi}$ which is smooth on $\SEN_{k,\ve}$, has $D_k$ ALF
asymptotics, and is hyperK\"ahler outside of $V_1$.  We shall not do
so here, since such a construction is subsumed in our work in
\S\ref{formalsol_sec}, see in particular Prop.~\ref{p1.17.1.20}.
\end{rmk}

To give the improved version of Theorem~\ref{version0}, we shall make
the construction of $\SEN_{k,\ve}$ uniform in $\ve$ for $\ve \to 0$.
We give a summary here, referring forward to \S\ref{s_gluing_space} for
the details.

We start by introducing compactifications $\ol{\AHr}$, $\ol\AHd$,
$\ol\TN$ by adjoining a boundary `at spatial infinity'.   Each of
these compactifications is a smooth $4$-manifold with fibred boundary,
cf.\ \S\ref{s1.26.11.20} below. Thus
$\p\ol{\AHr}$ is
the total space of the $S^1$-bundle over $S^2$ of degree $4$, $\p\ol{\AHd}$
is the quotient of this by $\iota$, and $\p\ol{\TN} = S^3$, viewed as
the total space of the Hopf fibration over $S^2$.   Moreover the ALF
metrics on these spaces extend smoothly to $\phi$-metrics on their
compactifications (see  Proposition~\ref{prop_alf_extension}). 
Similarly, for fixed $\ve>0$,   $(M_\ve,g_{\ve})$ has a
(partial) compactification $\ol{M}_\ve$ by 
adjoining $S^3/\la R_{2k^*} \ra$ (viewed as the total space of the
$S^1$-bundle of degree $2k^*$ over $S^2$, and interpreted as $S^2\times S^1$ if
$k^*=0$) as a boundary at spatial infinity. The involution $\iota$ extends smoothly to the
boundary and its quotient is of course $S^3/\cD_{k^*}$.  

Since
$\wt{\SEN}_k$ is obtained from $M_{\ve}$
by gluing in $\AHr$, near $0$, this compactification of $M_{\ve}$ also
gives a compactification of $\wt{\SEN}_k$ and the quotient by $\iota$
gives a compactification, $\ol{\SEN}_k$,  of $\SEN_k$ as a manifold
with fibred boundary.

These compactifications feature in the construction, in
\S\ref{s_gluing_space}, of spaces $\wt{\cW}$ and $\cW$ whose
properties we summarize in the following result:

\begin{prop}\label{list1}
  There exists a $5$-manifold $\wt{\cW}$ with corners up to
  codimension $2$ having the following properties: 
\begin{enumerate}
\item[(i)] $\wt{\cW}$ is equipped with a smooth proper map (technically, a
  $\bo$-fibration, cf.\ \cite{CCN}) $\wt{\pi} : \wt{\cW} \longrightarrow I =
  [0,\ve_0)$, $\ve_0$ being a small positive constant; and for $\ve>0$,
  $\pi^{-1}(\ve)$ is the above compactification of
  $\wt{\SEN}_{k,\ve}$, its boundary being $\pi^{-1}(\ve) \cap \wt{I}_\infty$.
\item[(ii)] The boundary of $\wt{\cW}$ is a union of hypersurfaces denoted
  $\wt{X}_{\ad}$, $\wt{X}_0$, $\wt{X}_p$ for $p\in P$ and
  $\wt{I}_\infty$.  Here 
\bee
\wt{X}_0 = \ol{\AHr},\; \wt{X}_p = \ol{\TN},\;\; (p \in P),
\eee
and 
\bee\label{e1.13.10.20}
\wt{\pi}^{-1}(0) = \wt{X}_{\ad} \cup \wt{X}_0 \cup \bigcup_{p \in P} \wt{X}_p.
\eee
\item[(iii)] The boundary hypersurface $\wt{I}_\infty$ is not compact but
  $\wt{\pi}|\wt{I}_\infty \to I$
is a smooth submersion onto $I$ and 
\bee\label{e11.23.10.20}
\wt{\pi}^{-1}(\ve)\cap \wt{I}_\infty = \p\ol{M_\ve}.
\eee
\item[(iv)]  The boundary hypersurfaces $\wt{X}_{\ad}$ and
  $\wt{I}_{\infty}$ are {\em fibred}: each is the total space of
a  principal circle-bundle
\bee
\wt{\phi}_{\ad} : \wt{X}_{\ad} \longrightarrow \wt{Y}_{\ad} = [\ol{\RR^3}; 
0,P],\;\;\wt{\phi}_\infty : \wt{I}_\infty \to \wt{Y}_\infty  = S^2 \times I.
\eee
\item[(v)] The boundary hypersurface $\wt{X}_{\ad}$ is itself a
  manifold with boundary.  Its boundary hypersurfaces (i.e.\ the
  connected components of $\p\wt{X}_{\ad}$) are denoted
  $\p_p\wt{X}_{\ad}$ for $p\in P\cup \{0\}$ 
(the internal boundary hypersurfaces) and $\p_\infty\wt{X}_{\ad}$
(the boundary at infinity).  Then we have
\bee\label{e1.28.10.20}
\p_p \wt{X}_{\ad} = \p \wt{X}_p = \wt{X}_{\ad} \cap \wt{X}_p\mbox{ and }
\p_\infty \wt{X}_\infty = \p \wt{I}_\infty
\eee
in the decomposition \eqref{e1.13.10.20}.
\item[(vi)] Given that $P=-P$, there is a smooth involution $\iota$ on
  $\wt{\cW}$ which acts on the fibres of $\pi$ ($\pi\circ \iota =
  \pi$) and which covers (the lift of) $x\mapsto -x$ on $\ol{\RR^3}$.
\end{enumerate}
\end{prop}

\begin{dfn} Granted this result, define $\cW = \wt{\cW}/\iota$
  (Figure~\ref{maintheorempic}).
  \end{dfn}

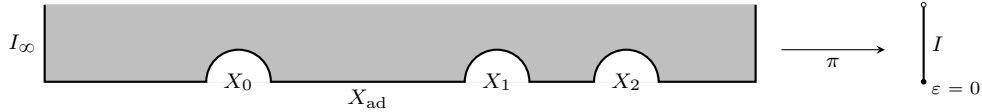
\begin{figure}[h]
\centering
\begin{tikzpicture}[>=stealth,scale=1.7]
\filldraw[lightgray] 
(-.5,.6) -- (-.5,0) -- (5,0) -- (5,0.6)-- cycle;
\filldraw[white] (.75,0) arc (180:0:.25) --cycle;
\filldraw[white] (2.75,0) arc (180:0:.25) --cycle;
\filldraw[white] (3.75,0) arc (180:0:.25) --cycle;
\draw[thick,black]
(-.5,.6) -- (-.5,0) -- 
(.75,0) arc (180:0:.25) --
(2.75,0) arc (180:0:.25) --
(3.75,0) arc (180:0:.25) --
(5,0) --(5,0.6);
\draw[->] (5.2,.25) -- (6, .25);
\draw[thick,black] (6.3,0) -- (6.3, .6);
\begin{scriptsize}
\draw (2,0) node[below] {$X_{\ad}$};
\draw (1.,0) node {$X_0$};
\draw (3.,0) node {$X_1$};
\draw (4.,0) node {$X_2$};
\draw (-.5,0.3) node[left] {$I_\infty$};
\draw (6.3,0.3) node[right] {$I$};
\filldraw (6.3,0) circle (0.5pt);
\draw (6.3,0.6) circle (0.5pt);
\filldraw[white] (6.3,0.6) circle (0.25pt);
\draw (5.6,.25) node[below] {$\pi$}; 
\end{scriptsize}
\begin{tiny}
\draw (6.3,-0.05) node[right] {$\ve=0$};
\end{tiny}
\end{tikzpicture}
\caption{Schematic picture of the space $\cW$ with boundary faces
  labelled.   For $\ve>0$, $\pi^{-1}(\ve)$ is the compactification
  $\ol{\SEN}_k$ of
  $\SEN_k$, but $\pi^{-1}(0) = X_{\ad} \cup X_0 \cup \cdots \cup
  X_k$.  $X_{\ad}$ is the total space of a circle-bundle with base
  $Y_{\ad}$ and $I_\infty$ is the total space of a circle-bundle with
  base $Y_\infty$.
}
\label{maintheorempic}
\end{figure}

This space has properties precisely analogous to (i)---(v) above 
and we shall use symbols unadorned by $\wt{\mbox{ }}$ for the
corresponding objects in $\cW$.   Since $\wt{X}_p$ and $\wt{X}_{-p}$
are identified by $\iota$, it is more convenient to enumerate the
images of these boundary hypersurfaces in $\cW$ by $X_1,\ldots,
X_k$. Each of these is still $\ol{\TN}$, but $X_0 = \ol{\AHd}$.   For
$\ve>0$, $\pi^{-1}(\ve) \subset \cW$ is the compactification
$\ol{\SEN}_{k,\ve}$ of the Sen space at scale $\ve$, the fibred
boundary being $\pi^{-1}(\ve)\cap I_\infty = \p(M_\ve/\iota) =
S^3/\cD_{k^*}$.

One more definition is needed before we can state the improved version
of our main theorem, that of the {\em rescaled vertical tangent bundle
} $T_{\phi}(\cW/I)$ of $\cW$, relative to $\pi$.
If $f: X \to Y$ is any submersion of manifolds without boundary, the
sub-bundle of vectors tangent to the fibres of $f$ will be 
denoted $T(X/Y)$ and the space of smooth $f$-vertical vector fields by
$\cV(X/Y)$.   

\begin{dfn}\label{def_rescaled}
Let $\rho$ be a smooth boundary-defining function of $X_{\ad}$, let
$\sigma_I$ be a smooth boundary-defining function of $I_\infty$. Assume
without loss that $\rho = \pi^*(\ve)$ in 
a collar neighbourhood of $I_\infty$.  A vector field $v$ on $\cW$ 
is called a {\em $\phi$-vector field} if $v$ is tangent to all
boundary hypersurfaces and in addition
\bee\label{e1.14.10.20}
v|X_{\ad} \in \cV(X_{\ad}/Y_{\ad}),\; v|I_\infty \in
\cV(I_\infty/Y_\infty)\mbox{ and }
v(\rho\sigma_I) \in \rho^2\sigma_I^2 C^\infty(\cW).
\eee
The space of all smooth $\phi$-vector fields on an open subset $U$ of
$\cW$ is denoted $\cV_{\phi}(U)$.  The subspace of $\cV_{\phi}(U/I)$
consists of $v\in \cV_{\phi}(U)$ which in addition are tangent to the
fibres of $\pi$.   By the Serre--Swan theorem, there is a vector
bundle to be denoted $T_{\phi}(\cW/I) \to \cW$
whose full space of sections $C^\infty(\cW,T_{\phi}(\cW/I))$ is equal
to $\cV_{\phi}(\cW/I)$.
\end{dfn}

This definition should be compared with the simpler
definition~\ref{d2.31.8.20} of $\cV_{\phi}(X)$ and $T_{\phi}X$ where
$X$ is a manifold with fibred boundary.  Indeed, for the restriction
of $T_\phi(\cW/I)$ to an interior fibre of $\pi$ we have
\bee
  T_{\phi}(\cW/I)|\pi^{-1}(\ve) = T_\phi \pi^{-1}(\ve) = T_{\phi}\ol{\SEN}_{k,\ve}\mbox{ for }\ve>0.
 \eee
We also have
\bee
  T_{\phi}(\cW/I)|X_\nu = T_{\phi} X_\nu\mbox{ for }\nu=0,\ldots,k.
 \eee
 The restriction of $T_\phi(\cW/I)$ to $X_{\ad}$ is  locally  spanned by the lifts from
  $M_{\ve}$ to $\cW$ of
\bee
\ve \frac{\p}{\p x_1}, \ve \frac{\p}{\p x_2}, \ve \frac{\p}{\p x_3},
\frac{\p}{\p \theta},
\eee
where $\theta$ is a locally defined angular fibre variable in
$X_{\ad}$.  The point is that this is a local frame for
$T_\phi(\cW/I)$ even at $\ve=0$.

Denote by $g_{\nu}$ the ALF
hyperK\"ahler metric on $X_\nu$, regarded as a smooth metric on
$T_\phi X_\nu$, cf.\ Proposition~\ref{prop_alf_extension}. (We allow
ourselves to blur the distinction between the complete metric on the
interior and the smooth metric on the compactification.)
Then 
the improved form of Theorem~\ref{version0} is as follows:
\begin{thm}
\label{mainthm}
Let $\cW$ and $T_\phi(\cW/I)$ be as defined above.  There exists
$\ve_0>0$ and a smooth metric $\bg_{\SEN}$ on $T_\phi(\cW/I)$ such that
$\bg_{\SEN}|\pi^{-1}(\ve)$ is a smooth $D_k$ ALF hyperK\"ahler metric on
$\SEN_{k}$.  Moreover,
\bee
\bg_{\SEN}|X_\nu = g_{\nu}\mbox{ and }
\bg_{\SEN}|X_{\ad} = g_{\ad}:=\frac{|\rd x|^2}{\ve^2} + \alpha_{\ad}^2,
\eee
viewed as a quadratic form on $T_\phi(\cW/I)|X_{\ad}$.  Here
$\alpha_{\ad} = \alpha|X_{\ad}$ and $\alpha$ is the connection
$1$-form defined in the Gibbons--Hawking Ansatz with potential
$h_{\ve}$, see \eqref{e1.22.8.18}.
        \label{main_thm}     \end{thm}

  \begin{rmk}
Putting $g_{\SEN,\ve}:=\bg_{\SEN}|\pi^{-1}(\ve)$ we recover
Theorem~\ref{version0}.   The rough interpretation of $g_{\SEN,\ve}$
as giving a superposition of the Atiyah--Hitchin and Taub--NUT
geometries is here
replaced by the statement that $\bg_{\SEN}|X_{\nu} = g_{\nu}$ for each
$\nu$. The
assertion that $\bg_{\SEN}$ is a smooth metric on $T_{\phi}(\cW/I)$
gives a precise meaning to the sense in which the family $g_{\ve}$ is
smooth in $\ve$ down to $\ve=0$.
\end{rmk}

The use of manifolds with corners in the analysis of partial
differential equations in non-compact and singular settings was
pioneered by Richard Melrose.  Of particular relevance to the
underlying analytical techniques are
\cite{mazzeo1990,GreenBook,MM_FB,scat}.    We note also references such
as \cite{gpaction,MZ17,MZ18,conlon2019} in which techniques including
real blow-up and rescaling the tangent bundle are used in a variety of
geometric contexts.   From this point of view, the combination of a
gluing theorem  with an adiabatic limit appears to be new---see
however, the following remark.

\begin{rmk}  We have set things up so that
the asymptotic size of the circle remains $\sim 2\pi$ as
$\ve \to 0$, while the length-scale of the base goes to
$\infty$. Alternatively, the rescaled metrics
$\ve^2g_{\SEN_k,\ve}$ give a family of metrics which 
are collapsing to the euclidean metric on $(\RR^3\setminus P)/\{\pm 1\}$.
In this formulation, we have a non-compact version of Foscolo's result
\cite{Foscolo16}, where a family of hyperK\"ahler
metrics on the K3 surface was constructed, collapsing to the flat
metric on $\TT^3/\{\pm 1\}$, away from a finite set of points.
\end{rmk}

\begin{rmk}
The fact that $\bg$ is {\em smooth} on $\cW$, rather than merely
having a polyhomogeneous conormal expansion---see
\cite{MZ17,MZ18}, for an example of this phenomenon---is due to 
special geometric features of the problem, which simplify the analysis
at a number of points.  We take full 
advantage of this, rather than attempting to develop the machinery
in generality.
\end{rmk}

\subsection{Further background}

ALF gravitational instantons can be interpreted and realised in a
number of  different ways. A gauge-theoretical model  was proposed by
Cherkis and Kapustin  in \cite{CherkisKapustin99}, where they showed
that the moduli space of a smooth and unit-charge  $SU(2)$ monopole
moving in the background of $k$ singular $U(2)$ monopoles is an 
 $A_{k-1}$ ALF   gravitational instanton of the Gibbons--Hawking form, and argued that
the moduli space of a  smooth and strongly centred charge-two $SU(2)$
monopole moving in the background of $k$ singular $U(2)$ monopoles is
a $D_k$  ALF  gravitational instanton.

Subsequently,  Cherkis and
Hitchin \cite{CherkisHitchin05} used twistor methods and a generalised
Legendre transform developed in  \cite{LindstromRocek88} and
\cite{IvanovRocek96}  for  a rigorous construction of $D_k$  ALF 
instantons. However, it was only shown  in \cite{Minerbe11} that all
$A_k$  ALF gravitational instantons are of the Gibbons--Hawking form
and, even more recently in \cite{ChenChen19},  that all  $D_k$ ALF 
gravitational instantons  are described by the Cherkis--Hitchin--Ivanov--Kapustin--Lindstr\"om--Ro\v{c}ek metric which emerged from the papers  \cite{LindstromRocek88,IvanovRocek96,CherkisKapustin99,CherkisHitchin05}. 

The interpretation of the $A_k$ and $D_k$ ALF spaces as moduli space
of monopoles with prescribed singularities provides useful intuition
about the geometry of these spaces. In particular, it suggests that,
at least for singularities which are well-separated from the origin,
one should be able to isolate a core region of the $D_k$ ALF space
where the smooth and strongly centred monopoles 
do  not `see' the $k$ singular monopoles and which should therefore be well-approximated by the Atiyah--Hitchin geometry. The picture also suggests that there should be an asymptotic region of the moduli space where the two smooth monopoles,  with their centre fixed at the origin,  are well-separated and move between (and over) the $k$ singularities. Orienting the line joining the smooth monopoles  amounts to double-covering  this part of the moduli space, and so this  double cover should  be 
 well-approximated by the moduli space of a  single monopole moving in background of $2k$ symmetrically spaced singularities, i.e. by the $A_{2k-1}$ ALF geometry. 

It is this intuitive picture,  also captured in Sen's proposal,  which our theorem makes precise. It differs from that underlying the  gluing construction using ALE instantons and the Eguchi-Hanson geometry,  carried out in \cite{Auvray18} and \cite{BiquardMinerbe11}, even though the mathematical techniques are related. It also clarifies that our method will not produce all $D_k$ ALF gravitational  instantons, but only those where the $k$ singularities are well-separated from the origin.

\subsection{Plan}
The proof of Theorem~\ref{mainthm} proceeds via a reformulation of
the problem it addresses  in a number of ways.   Instead of solving
for  hyperK\"ahler metrics we  use the formalism of  hyperK\"ahler
triples \cite{Donaldson:2:forms} to obtain an elliptic formulation of the gluing problem.
This method is described 
in \S\ref{triple_sec}. Then,  instead of working on non-compact
manifolds  we formulate the problem on compact spaces with fibred
boundaries. In  \S\ref{phi_sec}, we  explain  why these
compactified spaces are  natural  domains  for the family of
hyperK\"ahler metrics on the Sen spaces 
$\SEN_k$, and how the structure of the compact manifolds is such that the asymptotic
behaviour  of the metric can be  encoded  in smoothness and decay at all boundary
hypersurfaces. This is dealt with in \S\ref{s_gluing_space}.  In
\S\ref{formalsol_sec}, we construct a smooth 
triple on the fibres of $\cW$, which is hyperK\"ahler to all orders in
$\ve$.   The proof of Theorem~\ref{mainthm}
is completed in  final \S\ref{completion} by using the inverse function theorem
to perturb this family to be {\em exactly} hyperK\"ahler for all
sufficiently small positive $\ve$.   A number of technical results are
deferred to the appendices.

\subsection{Outlook}
Our construction  and main result can be extended in a
number of ways by replacing the space $\AHd$ with other dihedral ALF
gravitational instantons and adapting the adiabatic Gibbons--Hawking
Ansatz correspondingly.  

The simplest generalisation in this spirit is to replace $\AHd$ by
$\AH$, which is an example of a $D_1$ ALF gravitational instanton.
The branched cover of $\AH$ is  the manifold $\AHu$ whose asymptotic behaviour,
when written in the standard Gibbons--Hawking form, is that of a
singular $A_2$ space with a single pole of  weight $-2$. Thus one
could construct $D_k$ ALF gravitational instantons for $k\geq 2$  by
replacing $-\frac{2\ve}{|x|}$ by $-\frac{\ve}{|x|}$ in $h_{\ve}$ and
filling the hole near $x=0$ with $\AHu$ and then
dividing by the involution $\iota$.  
In this way one would obtain  $D_k$ instantons for $k\geq 2$  which are 
 not included in the family constructed here, but which arise in the
 limit  as  one pair of NUTs approaches zero.

More generally still, we could replace $h_{\ve}$ in \eqref{e1.9.1.20}
by 
$$
1  + (k'-2)\frac{\ve}{|x|} + \sum_{\nu=1}^k 
m_{\nu}\left(\frac{\ve}{2|x-p_{\nu}|} + \frac{\ve}{2|x+p_{\nu}|} \right) 
$$
for positive integers $m_{\nu}$.   This function is engineered so that
if we are given a $D_{k'}$ ALF gravitational instanton $M_0$ and for
each $\nu$ an $A_{m_{\nu}-1}$ gravitational instanton $M_{\nu}$, there
will be a 
version $\cW_m$, say,
of $\cW$ in which $X_0$ is the compactification of $M_0$ and $X_\nu$ is
the compactification of $M_{\nu}$.  It should be possible to
extend Theorem~\ref{mainthm} to construct $D_K$ ALF hyperK\"ahler
metrics on 
$\pi^{-1}(\ve) \subset \cW_m$,  where $K = k' + \sum_{\nu} m_\nu$. 
Again, the $D_K$ ALF metrics produced from these examples would
correspond to limiting cases of those obtained from
Theorem~\ref{mainthm}. 

Finally, one may wonder if   the geometrical interpretation of the
5-manifold $W$ in Theorem \ref{mainthm}  can be extended  from one
where each fibre has a metric to a fully geometrical picture,  with a
5-dimensional metric on $W$. This is possible for the (single NUT)
Taub--NUT geometry, which   can be extended to a warped product with a
Lorentzian geometry  satisfying  the (4+1)-dimensional  Einstein
equations \cite{GibbonsLuPope05}. The coordinate $t=1/\ve$ is a
natural time coordinate in this geometry, which fascinated Michael
Atiyah as a possible time-dependent model of the electron
\cite{AFS}. It would clearly  be interesting if this solution could be
generalised to natural  5-dimensional metrics on the manifold $W$ for
general ALF instantons, but we have not pursued this here. 

\section{HyperK\"ahler triples}

\label{triple_sec}

It is a remarkable fact that conformal geometry in four dimensions can
be described almost entirely in terms of the geometry of the bundle of
$2$-forms.  Here we review this, showing in particular how the
hyperK\"ahler condition in four dimensions can be completely described
in terms of orthonormal symplectic triples (Theorem~\ref{folk}).   The
section continues with some consequences for the analysis of
hyperK\"ahler metrics from this point of view.

Let $M$ be an oriented $4$-manifold.
\begin{dfn}
A {\em symplectic triple} on $M$ is a triple
$\omega= (\omega_1,\omega_2,\omega_3)$ of symplectic forms on $M$ such that the matrix $q$ with
$$
q_{ij} = \omega_i\wedge \omega_j
$$
is positive-definite at every point.  A {\em hyperK\"ahler triple} is
a symplectic triple for which $q$ is a multiple of the identity at
every point.
\end{dfn}

\begin{rmk}  Since 
$q$ is a symmetric $3\times 3$ matrix with values in
$\Lambda^4T^*$, it is worth spelling out what is meant by
`positive-definite' in this definition.   Let 
$\nu \in  C^\infty(M,\Lambda^4T^*)$ be any smooth
positive section, denote by $\nu^{-1}$ the dual section of $\Lambda^4
T$. (Such positive sections exist because $M$ is oriented.)  Then $\nu^{-1}q$ is a genuine
symmetric $3\times 3$ matrix at each point.   To say that $q$ is
positive-definite is to say that $\nu^{-1}q$ is positive-definite for
one or equivalently any positive section $\nu$ of $\Lambda^4T^*$.
Another way of giving this definition is to choose a riemannian metric
on $M$, and hence a Hodge $*$-operator.  

\end{rmk}

If $\omega$ is a hyperK\"ahler triple, taking the trace of $q_{ij}\,
\propto\, \delta_{ij}$, we obtain
\begin{equation}\label{e2.30.9.19}
\omega_i \wedge \omega_j =
\frac{1}{3}(\omega_1^2+\omega_2^2+\omega_3^2)\delta_{ij}.
\end{equation}
Thus this equation holds if and only if $\omega$ is a hyperK\"ahler
triple. The reason for the introduction of these triples, and the notation, is
explained by the following:

\begin{thm}  Let $M$ be an oriented $4$-manifold. If
  $(\omega_1,\omega_2,\omega_3)$ is a symplectic triple on $M$, then
  there is a unique metric $g$ on $M$ with
  \begin{equation}\label{e3.30.9.19}
    \Lambda^2_+(g) = \la \omega_1,\omega_2,\omega_3\ra
  \end{equation}
  and $\rd \mu_g = \tr(q)$.   If $\omega$ is a hyperK\"ahler triple, 
  then $g$ is hyperK\"ahler and the 
$\omega_j$ are the K\"ahler forms associated with the three complex
structures, 
$$
\omega_j(\xi,\eta) = g(\xi,I_j\eta).
$$
\label{folk}\end{thm}

Equivalently, the complex structures act on 1-forms according
\begin{equation}\label{e2.29.10.20}
  I_j \beta = *_g(\omega_j \wedge \beta).
\end{equation}

The above result has the status of a folk-theorem and there are many
closely related results in the hyperK\"ahler literature.
For example, Lemma~6.8 of \cite{HSDRS} is related to this result, but
there the existence of the metric $g$ is assumed, rather than derived
from the triple of forms.  A relatively recent reference for triples
of $2$-forms in $4$ dimensions is \cite{Donaldson:2:forms}.

For the convenience of the reader we recall the main points of the proof.
\begin{proof}[Proof of Theorem~\ref{folk}]  \hfill
\begin{enumerate}
\item  Let $V$ be an oriented four-dimensional real vector space, and
  let $H$ be the conformal quadratic form $H(\omega) =
  \omega\wedge\omega$ on $\Lambda^2 V^*$.  If $g$ is a 
  positive-definite inner product on $V$, then the space $P_g =
  \Lambda^2_+(g) \subset \Lambda^2 V^*$ of $g$-self-dual $2$-forms is
  a maximal $H$-positive subspace. As is well known, $P_g$ depends
  only upon the conformal class of $g$.   We claim that this is a
  bijection.  This is a manifestation of a low-dimensional special
  isomorphisms of homogeneous spaces.  The space of conformal inner
  products on $V$ may be identified with $\SL(4,\RR)/\SO(4)$, but we
  have the double covers $\SO(3,3) \to \SL(4,\RR)$ and $S(O(3)\times
  O(3)) \to  \SO(4)$. Thus the space of conformal inner products is
  the same as the homogeneous space $\SO(3,3)/S(O(3)\times O(3))$, and
  this is the Grassmannian of $H$-positive maximal subspaces of
  $\Lambda^2 V^*$.
\item By definition, a symplectic triple spans an $H$-positive subspace
  of $\Lambda^2 T^*$ at every point, and so by what has just been shown, it
  determines a conformal metric on $M$.  The condition $\rd \mu_g =
  \tr(q)$ clearly picks out a metric in the conformal class.
\item Let $P$ be an $H$-positive maximal sub-bundle of $\Lambda^2
  T^*M$ and let a volume form $\rd \mu$ on $M$ be given.   Then $\rd\mu$
turns the conformal metric $H|P$ into a genuine metric,
  $h$, say. We claim there is a unique connection 
  $\nabla$ on $P$ which preserves $h$ and which is
  torsion-free in the sense that
\bee\label{e1.29.10.20}
\mbox{skew}(\nabla \varphi) = \rd \varphi
\eee
for any local section $\varphi$ of $P$.  On the LHS, $\nabla \varphi$ is a
section of $T^*\otimes P \subset T^* \otimes \Lambda^2 T^*$, and
$\mbox{skew}$ is just the map into $\Lambda^3T^*$.

The claim can easily be proved by choosing a local orthonormal frame
$\varphi_1, \varphi_2,\varphi_3$ for $P$.  If $\nabla$ is a metric
connection on $P$, then for locally defined $1$-forms $u_j$, we have
\begin{eqnarray}
\nabla \varphi_1 &=&  u_3\otimes \varphi_2  - u_2 \otimes \varphi_3,  \nonumber\\
\nabla \varphi_2 &=&  -u_3 \otimes \varphi_1 + u_1\otimes \varphi_3, \\
\nabla \varphi_3 &=&  u_2 \otimes \varphi_1 - u_1\otimes \varphi_2. 
\nonumber
\end{eqnarray}
The torsion-free condition is then
\bee\label{e4.29.10.20}
\rd\varphi_1 =u_3 \wedge \varphi_2 - u_2\wedge \varphi_3,\;\mbox{ etc.}
\eee
and this system has a unique solution
\bee\label{e3.29.10.20}
u_1 =\hlf(*\rd \varphi_1  + I_3*\rd \varphi_2 - I_2*\rd \varphi_3),
\eee
with similar formulae for $u_2$ and $u_3$.   Here $(I_1,I_2,I_3)$ is the
triple of almost-complex structures determined by \eqref{e2.29.10.20}
with $\omega_j$ replaced by $\varphi_j$ and the quaternion relations
$I_1^2=I_2^2 = I_3^2 = I_1I_2I_3 = -1$ are used in the derivation of
\eqref{e3.29.10.20}. 
\item With $P$ as above, the unique $h$-compatible torsion-free
  connection just described must be equal to the Levi-Civita
  connection of the metric $g$ determined by $P$ and with volume
  element $\rd \mu$, since the Levi-Civita connection also has these
  two properties.   In particular, if $\omega$ is a hyperK\"ahler 
  triple, then we may apply the above with $\varphi_j = \omega_j$, $\rd
  \mu = \tr(q)$ to obtain that the $u_j =0$.  Thus the hyperK\"ahler
  triple gives a flat (with respect to the Levi-Civita connection)
  orthonormal trivialization of $\Lambda^2_+(g)$, and so $g$ is
  hyperK\"ahler with complex structures given by \eqref{e2.29.10.20}.
\end{enumerate}

\end{proof}

The formulation of the hyperK\"ahler condition in terms of triples of
forms has the advantage of leading to a nonlinear partial differential
equation which is elliptic modulo gauge freedom.  This has been
exploited already in, for example, \cite{Foscolo16,ChenChen19}.   We
now review the details of this.

\subsection{Perturbative formulation}

For any triple of $2$-forms $\omega$, set
\begin{equation}
Q(\omega) = 
\omega_i \wedge \omega_j -
\frac{1}{3}(\omega_1^2+\omega_2^2+\omega_3^2)\delta_{ij}.
\end{equation}
Then $Q$ is a symmetric, trace-free $3\times 3$ matrix, with values in
$\Lambda^4T^*M$ and by \eqref{e2.30.9.19}, $Q(\omega)=0$ if and only
if $\omega$ is a hyperK\"ahler triple.

We shall study the perturbative version of this equation.  That is,
we fix a symplectic triple $\omega$ with $Q(\omega)$ small and seek a
($C^1$-small) triple of $1$-forms $a$ so that 
\begin{equation}\label{e5.30.9.19}
Q(\omega + \rd a) = 0.
\end{equation}
More formally, $a\mapsto Q(\omega+\rd a)$ is a nonlinear differential
operator
\begin{equation}\label{e6.30.9.19}
\Omega^1(M)\otimes \RR^3 \longrightarrow C^\infty(M,
  S^2_0\RR^3\otimes \Lambda^4).
\end{equation}
This equation cannot be elliptic as the rank of the bundle on the left
here is $12$, while the rank on the right is $5$.  The difference in
ranks, $7$, is accounted for by the gauge-freedom of the problem.
Indeed,  \eqref{e5.30.9.19} is left invariant
by the action of the orientation-preserving diffeomorphisms
$\Diff^+(M)$; it is clearly also unchanged if $a$ is replaced by
$a+\rd f$, for any triple of functions $f=(f_1,f_2,f_3)$.  Thus there are
$7$ gauge degrees of freedom, and this count matches the difference in
the ranks of the bundles in \eqref{e6.30.9.19}. 

By fixing the gauge, we shall obtain an elliptic equation for the
triple $a$.  To state the
theorem, write $D$ for the coupled Dirac operator
\begin{equation}\label{e7.30.9.19}
\rd^* + \rd_+:  \Omega^1(M) \longrightarrow
\Omega^0(M) \oplus \Omega^2_+(M).
\end{equation}
This is determined by the metric $g(\omega)$ of Theorem~\ref{folk}.  To
avoid excessive notation, we shall not distinguish between $D$ and its
tripled version, in which every bundle in \eqref{e7.30.9.19} is
tensored with $\RR^3$.

\begin{lem}
Let $\omega$ be a symplectic triple as above and let $a$ be a triple
of $1$-forms.  Let $\wh{\omega} = \omega +\rd a$ and
let $q = (q_{ij}) = \omega_i\wedge\omega_j$.  Write
$p_{ij}$ for the inverse matrix $q^{-1}$ and $R = R_{ij}$ for the
matrix
\begin{equation}\label{e4.29.7.19}
R = \frac{1}{2}Q(\omega) + \frac{1}{2}Q(\rd a),
\end{equation}
Then the equation
\begin{equation}\label{e11.29.7.19}
\rd_+ a_j = -R_{js}p_{sk}\omega_k,
\end{equation}
implies that $\omega+\rd a$ is a hyperK\"ahler triple.
\label{pertlem}\end{lem}

\begin{rmk}
Here the summation convention holds for all repeated indices.
Furthermore, $q^{-1}$ is a matrix with values in $\Lambda^4 T$, $R$ is
a matrix with values in $\Lambda^4 T^*$ and these dual factors cancel
on the RHS of \eqref{e11.29.7.19}, giving a triple of `honest'
$2$-forms.
\end{rmk}

\begin{proof}
From the definitions, $Q(\omega +\rd a)=0$ is equivalent to the equation
\begin{equation}\label{e1.29.7.19}
Q(\omega,\omega) + 2Q(\omega,\rd a) + Q(\rd a,\rd a) = 0,
\end{equation}
where we have commited an abuse of notation by writing
$Q(\omega,\eta)$ for 
the polarized version of the quadratic form $Q$ (i.e.\
$Q(\omega,\omega) = Q(\omega)$).  Thus for 
triples of $2$-forms $\omega$ and $\eta$,
$Q(\omega,\eta)$ is by definition the projection onto the trace-free
symmetric part of the matrix $(\omega_i\wedge \eta_j)$. Thus
\eqref{e1.29.7.19} is implied by 
$$
Q(\omega)_{js} + 2 \rd a_j \wedge \omega_s + Q(\rd a)_{js}=0,
$$
which we rewrite as
\begin{equation}\label{e2.29.7.19}
  \rd a_j \wedge \omega_s = - R_{js},
\end{equation}
using the definition of $R$, and we claim this is equivalent to
\eqref{e11.29.7.19}. 

For the metric $g=g(\omega)$ of Theorem~\ref{folk}, $\Lambda^2_+(g)$
is spanned by the $\omega_j$, and so we have
\begin{equation}\label{e1.1.10.19}
\rd_+ a_j = u_{js}\omega_s
\end{equation}
for some collection of functions $u_{js}$. Then
\eqref{e2.29.7.19} is equivalent to
$$
u_{jk}\omega_k\wedge \omega_s =  u_{jk}q_{ks} = - R_{js}
$$
and multiplying by the inverse of $q$, we obtain
$$
u_{jk} = -R_{js}p_{sk}
$$
and hence, using \eqref{e1.1.10.19},
$$
\rd_+ a_j = -R_{js}p_{sk}\omega_k,
$$
as required.
\end{proof}

\begin{rmk}
If $\omega$ is itself a hyperK\"ahler triple then $Q(\omega)=0$
and \eqref{e11.29.7.19} takes the simpler form
$$
\rd_+ a_j = -\frac{1}{2}\nu^{-1}Q(\rd a)_{jk}\omega_k,
$$
where $\nu = q_{jj}/3 = (\omega^2_1+\omega^2_2 + \omega^2_3)/3.$
\end{rmk}

The following gives our elliptic formulation of the perturbative
hyperK\"ahler problem.
\begin{thm} Let the notation be as in Lemma~\ref{pertlem}.  Define a
  nonlinear mapping
  $$
  \cF: \Omega^1(M)\otimes \RR^3\longrightarrow
(  \Omega^0(M) \oplus \Omega^2_+(M) )\otimes \RR^3
  $$
  by
  \begin{equation}\label{e12.29.7.19}
    a_j \mapsto (\rd^* a_j , \rd_+ a_j + R_{js}p_{sk}\omega_k).
  \end{equation}
Then if $\omega+\rd a$ is a symplectic triple and $\cF(a)=0$, it
follows that $\omega+\rd a$ is a hyperK\"ahler triple. Furthermore,
the linearization of $\cF$ at $a=0$ is the 
(tripled version of the) Dirac operator $D_{g(\omega)}$.
\label{perturbative}\end{thm}

\begin{proof} Immediate from Lemma~\ref{pertlem}.
\end{proof}

Let us write $\cF$ in the form
\bee\label{e11.2.4.20}
\cF(a) = D a + e + \wh{r}(\rd a),
\eee
where $e = \cF(0)$ and $\wh{r}(\rd a)$ is the part of
$R_{js}p_{sk}\omega_k$ which is quadratic in $\rd a$.  Then $e$ will
be small if $\omega$ is approximately hyperK\"ahler.  To find a small
$a$ solving \eqref{e11.2.4.20}, we shall seek
\bee\label{e12.2.4.20}
a = D^*u,\;\; u \in (\Omega^0(M) \oplus \Omega^2_+(M))\otimes \RR^3.
\eee
Substituting into \eqref{e11.2.4.20},
\begin{equation}\label{e6.26.3.20}
\cF(a) = 0 \Leftrightarrow DD^* u = - e -\wh{r}(\rd D^*u).
\end{equation}
A standard Weitzenb\"ock formula relates $DD^*$ to the rough Laplacian
$\nabla^*\nabla$ of the metric $g= g(\omega)$, the curvature terms being
the self-dual part $W_+$ of the Weyl curvature and the scalar
curvature $s$.  These both vanish if $g$ is hyperK\"ahler and will be
small if $e$ is small.  This suggests that if $e$ is small enough, then
$DD^*$ should be invertible, given a suitable Fredholm framework for
$DD^*$, and then \eqref{e6.26.3.20} should be solvable for $u$ by the
implicit function theorem.    Since we want to solve $\cF(a)=0$ on an
ALF space, finding the right Fredholm framework is one of the
technical issues that we deal with this in this paper: then we shall
be able to apply the implicit function theorem to \eqref{e6.26.3.20}
to construct hyperK\"ahler triples on $\SEN_k$, thereby proving
Theorem~\ref{mainthm}.

\subsection{HyperK\"ahler triples for Gibbons--Hawking} 
\label{s_primitives_gh}

It is easy to verify that the forms
  \begin{eqnarray}
    \omega_1 &=&
\alpha \wedge \frac{\rd x_1}{\ve} + h_\ve \frac{\rd x_2\wedge \rd 
x_3}{\ve^2}, \nonumber\\
    \omega_2 &=&
\alpha \wedge \frac{\rd x_2}{\ve} + h_\ve \frac{\rd x_3\wedge \rd 
x_1}{\ve^2},
                 \label{e11.1.10.19} \\
        \omega_3 &=&
\alpha \wedge \frac{\rd x_3}{\ve} + h_\ve \frac{\rd x_1\wedge \rd 
x_2}{\ve^2} \nonumber
  \end{eqnarray}
  form a hyperK\"ahler triple for the adiabatic Gibbons--Hawking
  metric $g_{\ve}$ of \eqref{e1.22.8.18}.  In the rest of this section
  we study these forms near $x=0$: Corollary~\ref{prim_1} will be
  needed at the start of the gluing construction in \S\ref{formalsol_sec}.

In the punctured ball $B^*(0,\delta)\subset \RR^3$ we have
\begin{equation}\label{e11b.9.1.20}
  h_\ve = H_\ve + u_\ve,
\end{equation}
where
\begin{equation}\label{e11c.9.1.20}
H_\ve(x) = 1 + \mu \ve - \frac{2\ve}{|x|},\;\; u_{\ve}(x) = O(\ve|x|^2),
\end{equation}
and $u_{\ve}$ is smooth and harmonic in $B(0,\delta)$ (cf.\
\S\ref{improved}). 
There is a
corresponding decomposition
\bee\label{e23.23.10.20}
\alpha = \alpha_0 + \alpha_1,\; \rd\alpha_0 = *_{\ve} \rd H_{\ve},
\alpha_1 = *_{\ve} \rd u_{\ve},
\eee
(cf.\ \eqref{e6.10.10.20}) where $\alpha_0$ is a connection on the
principal $U(1)$-bundle $\phi^{-1}B^*(0,\delta)$ and $\alpha_1$ is a smooth $1$-form on $B(0,\delta)$.
 Since the triples \eqref{e11.1.10.19} are linear in $\alpha$
and $h_{\ve}$, we may write
\bee\label{e21.23.10.20}
\omega_j = \Omega_j + \eta_j,
\eee
where
\bee\label{e51.2.4.20}
\Omega_1 = \alpha_0 \wedge\frac{\rd x_1}{\ve} + H_{\ve}\frac{\rd 
  x_2\wedge \rd x_3}{\ve^2},\;
\eta_1 = \alpha_1 \wedge\frac{\rd x_1}{\ve} + u_{\ve}\frac{\rd 
  x_2\wedge \rd x_3}{\ve^2},\, \mbox{ etc}
\eee
in $\phi^{-1}B^*(0,\delta)$.

Now as a closed $2$-form on $B(0,\delta)$, $\eta_j$ must be exact.
The next result shows how to choose a primitive for $\eta_j$ whose
size is controlled by the size of $u_{\ve}$.  
In the following,  $|f|_{\ve}$ is
the pointwise length of a $p$-form with respect to the metric $|\rd
x|^2/\ve^2$.  The quantities in the next Proposition may depend
smoothly upon
$\ve$ as well as $x$ but we often suppress this from the notation. 
\begin{prop} \label{primitive2}
Let $u=u_\ve = O(\ve|x|^n)$ be smooth and harmonic in
$B(0,\delta)\subset \RR^3$, where $n$ is a positive integer. Let the $\eta_j$ be defined as in \eqref{e51.2.4.20}.
Then there exists a primitive $b_j$ of $\eta_j$ (i.e.\ $\rd b_j = \eta_j$)
such that $|b_j|_{\ve} = O(|x|^{n+1})$ in $B(0,\delta)$.
\end{prop}
\begin{proof} We need a quantitative form of the Poincar\'e
  lemma, which we state in the following slightly more general form:
  Let $f$ be a closed $p$-form in $B(0,\delta)$ such that $|f|_{\ve} = 
  O(|x|^m)$.  Then there exists
  a $(p-1)$-form $v$ in $B(0,\delta)$ with
  \bee\label{e25.23.10.20}
\rd v = f\mbox{ and }|v|_{\ve} = O(\ve^{-1}|x|^{m+1}).
  \eee
This follows at once from the proof of the Poincar\'e lemma using the radial
retraction of $B(0,\delta)$ to the origin, taking due care of the
fact that we are using the metric $|\rd x|^2/\ve^2$ to measure the
lengths or our forms.

Consider first the equation $\rd \alpha_1 = *_\ve \rd u_{\ve}$ for
$\alpha_1$.  We note that $|\rd u|_{\ve} = O(\ve^2|x|^{n-1})$, and so $*_\ve
\rd u$ is of the same order since $*_{\ve}$ is an isometry on forms
with respect to the metric $|\rd x|^2/\ve^2$.  From \eqref{e25.23.10.20},
$\alpha_1$ can be chosen in \eqref{e23.23.10.20} with $|\alpha_1|_{\ve}
=O(\ve|x|^{n})$.   Then
$|\eta_j|_{\ve} = O(\ve|x|^n)$, and applying
\eqref{e25.23.10.20} again, we find a primitive $b_j$ for $\eta_j$
with $|b_j|_{\ve} = O(|x|^{n+1})$ as claimed.
\end{proof}

Applying this in the case of interest:
\begin{cor}\label{prim_1}
With the above definitions, there is a triple of smooth $1$-forms $b_{\ve}$
in $B(0,\delta)$ such that
\begin{equation}
  \omega_{\ve} = \Omega_{\ve}+ \rd b_{\ve} \mbox{ in } B^*(0,\delta)
\eee
and 
\bee
|b|_{\ve} = O(|x|^3) \mbox{ in } B(0,\delta).
\eee

\end{cor}

\subsection{HyperK\"ahler triples for $\AH$}
\label{sect_AH_epsilon}
A hyperK\"ahler triple for the Atiyah--Hitchin metric was written down in 
\cite{IvanovRocek96}.   We shall not use this directly, because the
important thing for us is to compare the $\AH$ triple with the triple
of the asymptotic Taub--NUT model.  This is straightforward because
these metrics differ by exponentially small terms \eqref{AHmetricasy}.

Return to the family $g_{\AH,\ve}$ of Atiyah--Hitchin metrics which
are given asymptotically  (up to the involution $\iota$) by the negative-mass
Taub--NUT metric with potential
\bee
H'_{\ve}(x') = 1 + \mu\ve - \frac{2}{|x'|},
\eee
(cf.\ \eqref{e2.9.1.20}) for $|x'| > 1/\delta$, and let $g'_{\ve}$ be
the asymptotic metric \eqref{e2.10.10.20}.  Write the hyperK\"ahler
triple for $g_{\AH,\ve}$ as a sum of terms
\bee\label{e}
\omega_{\AH,\ve} = \Omega'_{\ve} + \eta',
\eee
where $\Omega'_{\ve}$ is the hyperK\"ahler triple of $g'_{\ve}$. 
The result parallel to Corollary~\ref{prim_1} is
\begin{prop}
  \label{prim_2}
Let the notation be as above. Then there exists
a triple of $1$-forms $a$ such that for $\ve \geq 0$,
$$
\omega_{\AH,\ve} = \Omega'_{\ve} + \rd a
$$
with $a$ (and all derivatives) exponentially decaying as $|x'|\to \infty$.
\end{prop}

\section{ALF spaces and manifolds with fibred boundary}

\label{phi_sec}

Let $X^n$ be a compact manifold with boundary.  For simplicity assume
$\p X$ is connected; otherwise the following discussion applies to
each of its connected components.  Recall that a {\em boundary
  defining function} (bdf) for $\p X$ in $X$ is a smooth function
$\rho\geq 0$, such that
\bee\label{e2.29.8.20}
\p X = \{p \in X: \rho(p) =0\}\mbox{ and }\rd\rho(p)\neq 0\mbox{ for
  all }p\in \p X.
\eee
The space of all
smooth vector fields on $X$ which are tangent to $\p X$ is denoted by 
$\cV_{\bo}(X)$.   If $p \in \p X$ then we may choose
local
coordinates in an open neighbourhood $U$ of $p$ in $X$ by taking local
coordinates $(y_1,\ldots,y_{n-1})$ centred at $p$ in $\p X$ and
adjoining the restriction to $U$ of $\rho$.  Any element of
$\cV_{\bo}(U)$ is then a smooth linear combination of the vector
fields
\bee\label{e3.29.8.20}
\rho \frac{\p}{\p\rho}, \frac{\p}{\p y_1},\ldots,\frac{\p}{\p
  y_{n-1}}.
\eee
(Warning: even though $\rho$ is globally defined on $X$, $\rho
\p_\rho$ is only defined when $\rho$ is used as here, as one of a
system of local coordinates.)
Either from this explicit description or by invoking the Serre--Swan
theorem, one sees that there is a vector bundle $T_{\bo}X$, the {\em
$\bo$-tangent bundle}, with the defining property that
$\cV_{\bo}(X)= C^\infty(X,T_{\bo}X)$, the full space of sections of
$T_{\bo}X$.   The study of $\bo$-geometry (and analysis) is the study
of metrics and differential operators defined on $X$, but using
$T_{\bo}X$ instead of $TX$, see \cite{GreenBook}.   Working in local
coordinates and making the change of variables $t = - \log \rho$ shows
that $\bo$-geometry on $X$ is 
(essentially) the same as asymptotically cylindrical geometry on the
interior of $X$.

\subsection{Fibred boundaries}
\label{phisect}
To capture ALF geometry in a similar fashion, 
we need to assume that $X$ has a {\em fibred boundary}.  For a
systematic account, see \cite{MM_FB}.

\begin{dfn}  Let $X^n$ be a compact manifold with connected boundary. We
  say that $X$ has fibred boundary, or $\phi$-structure, if $\p X$ is the
  total space of a smooth submersion $\phi$ from $\p X$ onto $Y$,  the
  latter being a  compact manifold without boundary.
\label{d1.31.8.20}\end{dfn}

Since we assumed $\p X$ is connected and compact, $Y$ must also be connected, and
it follows that the fibres $\phi^{-1}(y)$ are all mutually diffeomorphic.

\begin{dfn}\label{d2.31.8.20}
For any smooth fibration (i.e. surjective submersion)  $\phi: W \to
Y$, denote by $T(W/Y)$ the {\em vertical subbundle} of vectors tangent
to the fibres of $\phi$.  Let $X$ be a manifold with fibred boundary
as in the preceding definition and fix a boundary defining function $\rho$.
The space of $\phi$-vector fields, denoted $\cV_{\phi}(X)$ is the
subspace of $\bo$-vector fields $v$ such that
\bee\label{e1.31.8.20}
v|\p X \in C^\infty(\p X, T(\p X/Y))\mbox{ and }v(\rho) \in \rho^2
C^\infty(X).
\eee
\end{dfn}

Observe that the first of the conditions \eqref{e1.31.8.20} depends
only on the fibration of $\p X$, but the second depends on the choice
of $\rho$.  Nonetheless, with the $1$-jet of $\rho$ fixed at $\p X$,
there is again a smooth vector bundle, to be called the $\phi$-tangent
bundle $T_\phi X$, with the defining property $\cV_\phi(X)
=C^\infty(X,T_\phi X)$.

In order to give a local basis for $T_\phi X$, pick a point
$p \in \p X$ and choose local  coordinates
$(y_1,\ldots,y_b,z_1,\ldots, z_f)$ on $\p X$ centred at $p$ and adapted to
$\phi$ in the sense that
\bee
\phi(y_1,\ldots,y_b,z_1,\ldots,z_f) = (y_1,\ldots,y_b).
\eee
By adjoining $\rho$ we get a system of local coordinates in an open
subset $U$ of $X$.  Then the space $\cV_{\phi}(U)$
of $\phi$-vector fields in $U$ will be the space of smooth linear
combinations of the vector fields
\begin{equation}\label{e31.2.4.20}
  \rho^2\frac{\p}{\p \rho},\, \rho\frac{\p}{\p y_1},\ldots,
  \rho\frac{\p}{\p y_b},  \frac{\p}{\p z_1},\ldots, \frac{\p}{\p z_f}.
\end{equation}

\begin{dfn}
  Given a manifold $X$ with fibred boundary as above, a 
$\phi$-metric is a smooth metric on $T_{\phi} X$ which takes the
special form
\bee\label{e1.30.8.20}
g_\phi = \frac{\rd \rho^2}{\rho^4} + \frac{\phi^*h_1}{\rho^2} + h_2
\eee
in a collar neighbourhood $\p X^{+}$ of $\p X$: here
$h_1$ is a symmetric $2$-tensor on $Y\times [0,\rho_0)$,
positive-definite at $\rho=0$, and $h_2$ is a symmetric $2$-tensor
which is positive-definite on the vertical tangent bundle $T(\p
X/Y)$.
\end{dfn}

To see what kind of asymptotic geometry a $\phi$-metric captures, let
$X$ be a manifold with fibred boundary as above and let $C(Y) =
(0,\infty)_r \times Y$ be the cone over $Y$. Let
$\p X^+$ denote the pull-back of $\p X \to Y$ over $C(Y)$. If we
write $h_1|Y = H_1$ and $h_2|\p X = H_2$, then with $r
= 1/\rho$ in \eqref{e1.30.8.20},
\bee\label{e1.12.10.20}
g_\phi = \rd r^2 + r^2 \phi^*H_1 + H_2 + O(1/r) \mbox{ for }r\to \infty.
\eee
Thus the restriction of $g_\phi$ to $X\setminus \p X$ is a metric asymptotic to the metric
\bee
\label{e1a.12.10.20}
g_\phi^+ = \rd r^2 + r^2 \phi^*H_1 + H_2 
\eee
on $\p X^+$.  The first two terms give a cone metric on the base
$C(Y)$, while the $H_2$-term gives a smooth family of metrics on the
fibres $\phi^{-1}(y)$, which however do not suffer any rescaling with
$r$: the fibres remain the same size
and shape as $r\to \infty$ along any given generator of $C(Y)$.  In
the special case that $Y = 
S^3$, $C(Y) = \RR^3\setminus 0$ and $\p X \to Y$ is a circle-bundle,
we recover $A$-type ALF geometry in four dimensions, and $D$-type ALF
geometry just corresponds to the quotient of this by an involution.

\subsection{Compactification of spaces with Gibbons--Hawking asymptotics}
\label{s1.26.11.20}
Suppose that $h(x)$ is positive and harmonic in the set
$U = \{\delta|x| > 1\} \subset \RR^3$ (for some $\delta>0$) and that
\bee\label{e4.31.8.20}
h(x) = 1+ \frac{\ell}{2|x|} + O(|x|^{-2}) \mbox{ in }U,
\eee
where $\ell \in \ZZ$.  
Denote by $\psi : V \to U$ the circle-bundle of degree $-\ell$ over
$U$ so that the Gibbons--Hawking Ansatz
\bee\label{e11.15.10.20}
g = h(x)|\rd x|^2 + h(x)^{-1}\alpha^2,\;\; \rd \alpha = * \rd h
\eee
defines an ALF metric on $V$.

\begin{prop} \label{prop_alf_extension}
Let $g$ be as in \eqref{e11.15.10.20}.  Then there is a partial
compactification
\bee\label{e12.15.10.20}
\ol{\psi} : \ol{V} \to  \ol{U}
\eee
of $\psi$ and a smooth connection-form $\ol{\alpha}$ on
$\ol{V}$ (viewed as a circle-principal bundle over $\ol{U}$) for
which $\alpha =\ol{\alpha}|U$ satisfies the second of
\eqref{e11.15.10.20} and $g$ is the restriction to $V$ of a smooth
$\phi$-metric $g_{\phi}$ on $\ol{V}$.
\end{prop}

\begin{proof}
The partial compactification $\ol{U}$ is defined to be the radial
compactification of $U$.  Thus
\bee
\ol{U} = S^{2} \times [0,\delta)_\rho,
\eee
where $U$ is identified with the interior of $\ol{U}$ by inverted
polar coordinates, $x\longmapsto (
x/|x|, 1/|x|)  \in S^{2}\times (0,\delta)$.   Then $\ol{U}$ is a (topological) manifold
with boundary and it has a unique smooth structure extending that of
$U$  with respect to which $\rho$ is a smooth boundary defining
function for $\p\ol{U}$.  More informally, $\ol{U}$ is obtained from
$U$ by adjoining the `sphere at infinity' of $\RR^3$. 
(Radial compactification is discussed in
detail in \cite{scat}.)

Since $h$ is harmonic in $\{\delta|x|>1\}$, it has a multipole
expansion, each term of which is a homogenous harmonic function of
$\RR^3\setminus \{ 0\}$ of non-positive degree.  Each such multipole term
extends smoothly to $\ol{U}$ and it follows that the coefficients $h$
and $h^{-1}$ in 
\eqref{e4.31.8.20} extend smoothly to $\ol{U}$.   As in
\eqref{e1.12.10.20} in inverted
polar coordinates $x = y/\rho$, where $|y|=1$, 
$$
|\rd x|^2 = \frac{\rd \rho^2}{\rho^4} + \frac{|\rd y|^2}{\rho^2},
$$
from which it follows that the first term in $g$ is the restriction
of the `base part' of a $\phi$-metric on $\ol{U}$.

To construct $\ol{\psi}: \ol{V} \to \ol{U}$, an $S^1$-bundle of degree
$-\ell$ over $\ol{U}$, we simply take the $S^1$-bundle $\Sigma \to
S^2$ of degree $-\ell$ and define $\ol{V} = p_1^*\Sigma$,
where $p_1 : \ol{U} \to S^2$ is the
projection on the first factor.  This gives $\ol{\psi} : \ol{V} \to
\ol{U}$ as required.  Viewing $\Sigma$ as a principal
$S^1$-bundle, it has a standard smooth connection whose curvature
is
$(-\ell/2)$ times the 
round area-form on $S^2$. Denote by
$\alpha_0$ the pull-back of this connection to $\ol{V}$. Then
$\alpha_0$ is smooth on $\ol{V}$ and 
\bee
\rd \alpha_0 = *\rd\left(1 + \frac{\ell}{2|x|}\right)\mbox{ in }V.
\eee
We now define $\ol{\alpha} = \alpha_0 + \alpha_1$, where $\alpha_1$ is
a smooth form $\ol{U}$ satisfying
\bee\label{e5.31.8.20}
\rd \alpha_1 = *\rd
\left(h -\frac{\ell}{2|x|}\right)
\eee
in $U$, so that $\ol{\alpha}|U$ solves the second of
\eqref{e11.15.10.20}.  By working term-by-term with the multipole
expansion of $h$ (or by using a retraction to the boundary), it is not
hard to see that this can be solved for $\alpha_1 \in
\Omega^1(\ol{U})$, as required.  Then $\ol{\alpha} = \alpha_0 +
\alpha_1$ is smooth on $\ol{V}$ and its restriction $\alpha =
\ol{\alpha}|V$ satisfies the second of \eqref{e11.15.10.20}.  In this
way, we see that $g$ is indeed the restriction of a smooth
$\phi$-metric on $\ol{V}$.
\end{proof}

The Atiyah--Hitchin metric has Gibbons--Hawking asymptotics, modulo
terms that are {\em exponentially small} at infinity
\eqref{AHmetricasy}.  Because of this, it is convenient to have the
following definition:

\begin{dfn}  Let $X$ be a $4$-manifold with fibred boundary, and let
  $g_\phi$ be a $\phi$-metric on $X$.  We say that $g_{\phi}$ is {\em
    strongly  ALF} if there is a (compactified) ALF Gibbons--Hawking
  metric $g_{0,\phi}$, say, in a neighbourhood $\ol{V}$ of $\p X$
  (cf.\ Prop.~\ref{prop_alf_extension}) such that $g_\phi - g_{0,\phi}
  = O(\rho^{\infty}))$ near $\p X$.  In this definition we allow the
  base of the fibration of $\p X$ to be $S^2$ or $S^2/\{\pm 1\}$, the
  latter being needed as always for the $D_k$ ALF geometry.
  \label{strongALF}\end{dfn}


\subsection{Real blow-up}
\label{s_buhalf}
In this section we give a quick explicit description of the real
blow-up of the origin in $\RR^n\times [0,\ve_0)$, since this is needed
for the construction of our spaces $\wt{\cW}$ and $\cW$.
For a systematic account of real blow-up, two
references among many are \cite{gpaction,daomwc}.  

The real blow-up
\bee
B = [ \RR^n \times [0,\ve_0); (0,0)]
\eee
of $\RR^n\times [0,\ve_0)$ in the origin may be defined as the
quotient of
\bee
B' := \{ (Z,R,S) \in \RR^n \times [0,\infty)^2 : (Z,R)\neq (0,0),\; RS  < \ve_0 \}
\eee
by the $\RR_+$-action
\bee
t\cdot (Z,R,S) = (tZ, tR, t^{-1}S).
\eee
Denote by $[Z,R,S]$ the $\RR_+$-orbit of $(Z,R,S) \in B'$---we think
of $[Z,R,S]$ as homogeneous coordinates on $B$.
The condition $RS < \ve_0$ descends to the quotient and we see that $B$ is a manifold with corners 
up to codimension $2$, with boundary hypersurfaces 
\bee 
F_0 = \ff(B) =\{S=0\},\;\; Y = \{R=0\}
\eee 
meeting in the corner $\{R=S=0\}$.  The blow-down map is given by the
formula 
\bee
\beta : B \longrightarrow \RR^n\times[0,\ve_0),\;\;
\beta : (Z,R,S) \mapsto (SZ, SR)
\eee
in homogeneous coordinates. 
From its definition, $\beta$ is a
diffeomorphism from $B\setminus F_0$ onto $\RR^3\times
[0,\ve_0)\setminus (0,0)$, but $\beta$ crushes $F_0$ to $(0,0)$ (Figure~\ref{bu_diag}).

As the notation suggests, $F_0$ is the front face (exceptional
divisor) of the blow-up, while $Y$ is the (lift of the) `old'
boundary, in other words the lift to $B$ of $\{\ve =0\}\subset \RR^3
\times [0,\ve_0)$. 

The map
\bee\label{e11.29.10.20}
\RR^n\times [0,\ve_0) \ni (x',\ve') \longmapsto [x', 1,\ve'] \in
B\setminus Y
\eee
is a diffeomorphism, and so we have local coordinates on the open
subset $B\setminus Y$ of $B$. In these coordinates, $\beta(x',\ve') =
(\ve'x', \ve')$.   Similarly, the map
\bee\label{e12.29.10.20}
\RR^n\times [0,\ve_0) \ni (x,\ve) \longmapsto [x, \ve,1] \in
B\setminus F_0
\eee
is a diffeomorphism, and so we have local coordinates $(x,\ve)$ on
$B\setminus F_0$ with respect to which $\beta(x,\ve) = x$. These two
charts cover everything but the corner $Y\cap F_0$ of $B$.  Here we
have a neighbourhood
$S^{n-1} \times H$, where 
\bee
H = \{(\rho,\sigma) \in [0,\infty)^2 : \rho\sigma < \ve_0\}
\eee
and a diffeomorphism
\bee
S^{n-1} \times H \ni (y,\rho,\sigma) \longrightarrow [y,\rho,\sigma]
\in B\setminus \{z=0\}.
\eee
In these coordinates, $\beta(y,\rho,\sigma) = (\sigma y, \rho\sigma)$  and
$\rho$ and $\sigma$ are local defining functions for $Y$ and $F_0$
respectively.

The relationships between these different coordinate charts are as follows:
\bee
y = \frac{x}{|x|} = \frac{x'}{|x'|},\; \rho = \frac{1}{|x'|} =
\frac{\ve}{|x|},\; \sigma = \ve'|x'| = |x|.
\eee\begin{figure}[h]
\centering
\begin{tikzpicture}[>=stealth,scale=1.5]
\filldraw[lightgray] 
(-3,2) -- (-3,0) -- (3,0) -- (3,2)-- cycle;
\filldraw[white] (1,0) arc (0:180:1) --cycle;
\draw[thick,black]
(3,0) -- (1,0) arc (0:180:1) -- (-3,0);
\draw (-2,0) node[below] {$Y$};
\draw (-0.8,0.3) node[right] {$F_0=\ff(B)$};
%
%
\draw[very thick,->] (1,0) -- (1,0.6);
\draw[very thick,->] (1,0) -- (1.6,0);
\draw(1,.3) node[right] {$\rho$};
\draw(1.3,-0.0) node[below] {$\sigma$};
\draw[very thick,->] (2,0) -- (2,0.6);
\draw[very thick,->] (2,0) -- (2.6,0);
\draw(2.3,0.0) node[below] {$x$};
\draw(2,0.3) node[right] {$\ve$};
%
%
\draw[very thick,->] (0,1) -- (0.6,1);
\draw(.3,1) node[above] {$x'$};
\draw[very thick,->] (0,1) -- (0,1.6);
\draw(0,1.3) node[left] {$\ve'$};
\end{tikzpicture}
\caption{The blow-up $B$ of $\RR^n\times [0,\ve_0)$}
\label{bu_diag}
\end{figure}

We observe that $F_0$ is canonically identifiable with the radial
compactification of $\RR^n$ (really the radial compactification
of $T_0\RR^n$) and that $Y = [\RR^n;0]$, the real blow-up of $\RR^n$
at the origin.

In addition to the blow-down map $\beta$, we have the lift to $B$ of
the projection map $\RR^n\times [0,\ve_0) \to [0,\ve_0)$.  In terms of
the homogeneous coordinates, 
\bee
\pi[Z,R,S] = RS
\eee
and in the above local coordinates,
\bee
\pi(x',\ve') = \ve',  \pi(y,\rho,\sigma) = \rho\sigma, \pi(x,\ve) =
\ve.
\eee

Using the above coordinate systems, the following lifting results are easy to
verify.
\begin{prop}\label{nu_prop}
The map $(x,\ve) \mapsto x/\ve$, defined on $\RR^n\times (0,\ve_0)$,
extends to define a smooth mapping
  \bee
r: B \to F_0,\; r[Z,R,S] = [Z,R,0].
  \eee
Similarly, if $f(\ve,x) = \ve/|x|$, then $\beta^*f$ extends to define
a smooth function on $B\setminus \Sigma$, where 
$\Sigma = \{[0,1,\ve']: \ve'\geq 0\}$ is the lift of $x=0$ to $B$.  $\beta^*f$ has
the simple singularity $1/|x'|$ near $\Sigma$ and vanishes identically
on $Y$.
\label{plift1}
  \end{prop}

\begin{prop} Let $(x_1,\ldots,x_n)$ be euclidean coordinates on
  $\RR^n$.  The vector field $\ve \p/\p x_j$ lifts to define a smooth
  vector field on $B$, described as follows: away from
  $Y$ the lift is just $\p/\p x'_j$; away from $F_0$  it is $\ve \p/\p
  x_j$; and near the corner $F_0\cap Y$ it is
a smooth linear combination of the vector fields
\bee
\rho^2\frac{\p}{\p \rho} - \rho\sigma\frac{\p}{\p \sigma} \mbox{ and }
\rho\frac{\p}{\p y}, 
\eee
where $\p/\p y$ is shorthand for a smooth vector field on $S^{n-1}$.
\label{plift2}
\end{prop}

\begin{rmk}  The {\em lift} of a vector field here is defined as the push-forward
  by the inverse
  of $\beta$ from $\RR^n\times [0,\ve_0)\setminus (0,0) \to B\setminus
F_0$.  In general a smooth vector field on $\RR^n\times [0,\ve_0)$ will have
a smooth lift to $B$ iff it vanishes at $(0,0)$; the Proposition shows
what the lift looks like in the case of a particular class of such
vector fields.  
\end{rmk}

\section{The `gluing space' $\cW$}

\label{s_gluing_space}

The construction of $\wt{\cW}$ proceeds in several steps.  This is the
major effort, for $\cW$
is just the quotient of $\wt{\cW}$ by an involution $\iota$.   We shall
construct $\wt{\cW}$ by gluing together two pieces, $\wt{\cW_0}$ and
$\wt{\cW_1}$. The first of these is essentially the product $\AHr \times
[0,\ve_0)$ while the second is a resolved version of the family of
$M_\ve$.

\subsection{Resolution of the adiabatic Gibbons--Hawking family}

To construct $\wt{\cW_1}$, we need to resolve the base $\RR^3\times
[0,\ve_0) \cap \{|x| >2\ve\}$ of the family of Gibbons--Hawking
manifolds $M_{\ve}$.   Let us write $I = [0,\ve_0)$ in what follows.

Pick $\delta \in (0,1/2)$.  In what follows one should think of this
as being small but fixed, whereas we allow $\ve$ to vary in
$[0,\delta^2)$. 
Define
\bee\label{e11.10.10.20}
\wt{\cB} = [\ol{\RR^3}\times I; (0,0),P\times \{0\}] \cap \{|x|> \ve/\delta
\}.
\eee
(Note
that we have passed to the radial compactification of $\RR^3$.)
The blow-ups in  \eqref{e11.10.10.20}  are independent of each other (so can be carried
out in any order) and
each blow-up is modelled by the explicit description given in
\S\ref{s_buhalf}.  Thus $\wt{\cB}$ is a $4$-manifold with corners
and we label the boundary hypersurfaces in
the obvious way: for $p$ in $P\cup \{0\}$, $F_p$ is the front face of the
blow-up of $(p,0)$, while the lift of
$\{\ve =0\}$ to $\wt{\cB}$ is denoted by $\wt{Y}$.
Then $\wt{Y} = [\ol{\RR^3}; P,0]$ is the real
blow-up of
the points $0,\pm p_1,\ldots, \pm p_k$ in $\ol{\RR^3}$.  

 For $p\neq 0$, $F_p$ is a copy of $\ol{\RR^3}$, with euclidean
coordinate $x'_p = (x-p)/\ve$.  Setting $x' = x/\ve$,  $F_0$ is the part
$\{\delta |x'| > 1\}\subset  \ol{\RR^3}$ of the front face of the blow-up at
$0$.  See Figures~\ref{pic_un} and \ref{pic_res}.

\begin{figure}[h]
\centering
\begin{tikzpicture}[>=stealth,scale=1.5]
\filldraw[lightgray] 
(-.5,.6) -- (-.5,0) -- (5,0) -- (5,0.6)-- cycle;
\draw[thick,black]
(-.5,.6) -- (-.5,0) -- 
(5,0) --(5,0.6);
\filldraw[white] (1,0) -- (1.2,1) -- (.8,1) --cycle;
\draw[->] (5.2,.25) -- (6, .25);
\draw[thick,black] (6.3,0) -- (6.3, .6);
\begin{scriptsize}
\filldraw (1.,0) circle(1pt) node[below] {$0$};
\filldraw[white] (1,0) circle(0.5pt);
\filldraw (3.,0) circle(1pt) node[below] {$p$};
\filldraw (4,0) circle(1pt) node[below] {$q$};
\draw[thick,black,dotted] (3,0.) --(3,0.6);
\draw[thick,black,dotted] (4,0.) --(4,0.6);
\filldraw[white] (4,0) circle(.5pt);
\filldraw[white] (3,0) circle(.5pt);
\draw (-.5,0.3) node[left] {$S^2 \times I$};
\draw (6.3,0.3) node[right] {$I$};
\filldraw (6.3,0) circle (0.5pt);
\draw (6.3,0.6) circle (0.5pt);

\draw (5.6,.25) node[below] {$\pi$}; 
\end{scriptsize}
\begin{tiny}
\draw (6.3,-0.05) node[right] {$\ve=0$};
\end{tiny}
\end{tikzpicture}
\caption{The space $\{|x| > 2\ve\}\subset \ol{\RR^3}\times
  [0,\ve_0)$.}
\label{pic_un}
\end{figure}

\begin{figure}[h]\label{resolve_h}
\centering
\begin{tikzpicture}[>=stealth,scale=1.5]
\filldraw[lightgray] 
(-.5,.6) -- (-.5,0) -- (5,0) -- (5,0.6)-- cycle;
\filldraw[white] (.75,0) arc (180:0:.25) --cycle;
\filldraw[white] (2.75,0) arc (180:0:.25) --cycle;
\filldraw[white] (3.75,0) arc (180:0:.25) --cycle;
\draw[thick,black]
(-.5,.6) -- (-.5,0) -- 
(.75,0) arc (180:0:.25) --
(2.75,0) arc (180:0:.25) --
(3.75,0) arc (180:0:.25) --
(5,0) --(5,0.6);
\filldraw[white] (1,0) -- (1.2,1) -- (.8,1) --cycle;
\draw[->] (5.2,.25) -- (6, .25);
\draw[thick,black] (6.3,0) -- (6.3, .6);
\draw[thick,black,dotted] (3,0.25) --(3,0.6);
\draw[thick,black,dotted] (4,0.25) --(4,0.6);
\begin{scriptsize}
\draw (2,0) node[below] {$\wt{Y}$};
\draw (1.,0) node {$F_0$};
\draw (3.,0) node {$F_p$};
\draw(4,.4) node[right] {$\Sigma_{q}$};
\draw(3,.4) node[right] {$\Sigma_{p}$};
\draw (4.,0) node {$F_q$};
\draw (-.5,0.3) node[left] {$S^2\times I$};
\draw (6.3,0.3) node[right] {$I$};
\filldraw (6.3,0) circle (0.5pt);
\draw (6.3,0.6) circle (0.5pt);
\filldraw[white] (6.3,0.6) circle (0.25pt);
\draw (5.6,.25) node[below] {$\pi$}; 
\end{scriptsize}
\begin{tiny}
\draw (6.3,-0.05) node[right] {$\ve=0$};
\end{tiny}
\end{tikzpicture}
\caption{Schematic picture of the resolved base $\wt{\cB}$.}\label{pic_res}
\end{figure}

Recall from \S\ref{AGH} the fibred structure $\phi: M_{\ve} \to
\{|x|>2\ve\}$ of the manifold on which the metric $g_{\ve}$ lives.  This extends
smoothly to $\ol{M}_{\ve}$, the space obtained by adjoining a
boundary at spatial infinity (cf.\
Proposition~\ref{prop_alf_extension}).  Pulling $\phi$ 
back by the obvious map $\wt{\cB} \to  \ol{\RR^3}$ (composite of blow-down
and projection on the first factor) we obtain a $5$-manifold
$\wt{\cW_1}$ with 
$S^1$-action such that $\wt{\cW}_1/S^1=\wt{\cB}$.  We abuse notation
by denoting this quotient map also by $\phi$.  The NUTs in
$\wt{\cW_1}$ become the 
$\phi^{-1}(\Sigma_p)$, for $p\in P$, and $\wt{\cW}_1 \setminus \bigcup_p
\phi^{-1}\Sigma_{p}$ is a circle-bundle over
$\wt{\cB}\setminus \bigcup_p \Sigma_p$. In addition to the fibration
$\phi$, there is a smooth map $\pi: \cW_1 \to I$ given by composing $\phi$
with the map $\cB \to I$ given by blow-down followed by projection of
the product $\ol{\RR^3}\times I$ to $I$.

The point of this is that the family of metrics $g_{\ve}$ will lift to
define a smooth metric $\bg_1$, say, on $T_{\phi}(\wt{\cW}_1/I)$, with
$$
\bg_1|\pi^{-1}(\ve) = g_{\ve}\mbox{ for }\ve >0,\;
\bg_1| \phi^{-1}(F_p) =g_{\TN}\mbox{ and }
\bg_1|\wt{X}_{\ad} = g_{\ad}.
$$ 
(Cf.\ Theorem~\ref{mainthm}).   We shall not prove this here: the
corresponding statement at the level of triples is given in
Lemma~\ref{l11.2.12.20}.  Instead,
let us complete the construction of $\wt{\cW}$,
`filling the hole' in $\phi^{-1}(F_0)$ with a family of
Atiyah--Hitchin manifolds.

\label{s1.28.11.20}
\begin{center}
\begin{figure}[h]
\begin{tikzpicture}[>=stealth, scale=1.5]
  \filldraw[lightgray, semitransparent, domain=.3333:2, smooth, variable = \x] plot 
  ({\x},{1/\x}) -- (2,0) -- (0,0) -- (0,3) -- (.3333,3); 
    \draw[domain=.1:2, smooth, variable = \x,thick,dotted] plot 
    ({\x},{.3/\x}); 
  \filldraw[gray, semitransparent, domain=.5:2, smooth, variable = \x] plot 
  ({\x},{1/\x}) -- (2,0) -- (0,0) -- (0,2);
  \draw[black,very thick] (2,0) --(0,0) -- (0,3);
\begin{tiny}
  \draw  (2,0)  node[below]
            {\scriptsize $\rho=\delta$};
                 \draw (.7,.7) node[right] {$\wt{\cV}_1$};
\draw (0,2) node[left] {$\sigma = \delta$};
\draw (0,0) node[below] {$(0,0)$};
\draw (1.3,.8) node[right] {$\rho\sigma = \ve_0$};
\end{tiny}
  \begin{scope}[xshift=4cm]
    \filldraw[lightgray, semitransparent] (0,0) -- (3,0) --
    (3,.5) -- (0.3,.5) -- (0,0);
        \filldraw[gray, semitransparent] (0,0) -- (1,0) --
        (1,.5) -- (0.3,.5) -- (0,0);
            \draw[very thick] (0,0) -- (3,0);
            \draw[thick, dotted] (0.085,0.15) -- (3,0.15);
\begin{tiny}
            \draw (0,0) node[below]
            {$\rho=0$};
\draw (.5,.2) node[above] {$\wt{\cV}_0$};
            \draw (-0.2,0.75) node[left]
            {$\{\delta\rho= \ve'\}$};
\draw[thin,->] (-.22,.7).. controls (-.1,.7) and (.1,.5) ..(0.18,.38);
                 \draw[very thin] (1,0) node[below]
                 {$\rho=\delta$};
                 \draw(1.5,.5) node[above] {$\ve' = \ve_0$};
\end{tiny}
            \end{scope}
\end{tikzpicture}
\caption{Gluing $\wt{\cW}_1$ (left) to $\wt{\cW}_0$ (right). The schematic
  picture shows the subsets $\wt{\cV}_1 \subset \wt{\cW}_1$ and
  $\wt{\cV}_0\subset \wt{\cW}_0$ in dark grey.  These are
matched up by the gluing map $\vc{\kappa}$.
The dotted lines represent $\pi_1^{-1}(\ve)$ and $\pi_0^{-1}(\ve)$.
  These are also matched up by $\vc{\kappa}$.}
\label{fig_glue}\end{figure}
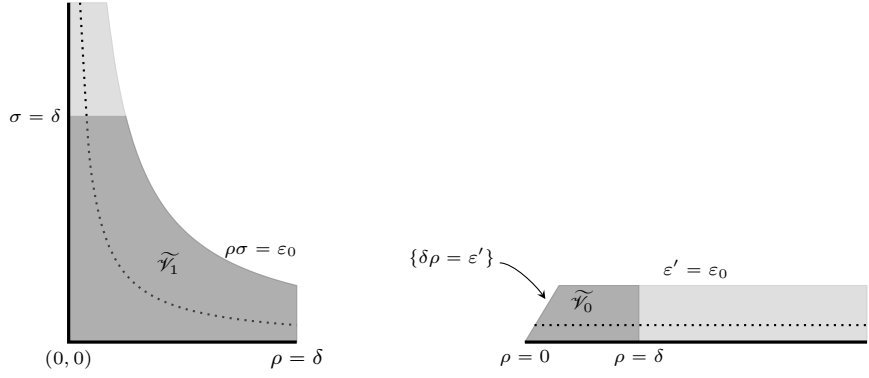
\end{center}

So return to $\AHr$ with its asymptotic region a circle-bundle $\psi$
with base $\{\delta|x'| >1\}$. Define
\bee
\wt{\cW}_0 = \{(z',\ve') \in \AHr \times [0,\ve_0) : |\psi(z')| <
\delta/\ve'\},
\eee
where it is understood that the condition applies only to $z' \in
\psi^{-1}\{\delta|x'| >1\}$.  Denote by $\pi_0:\cW_0 \to I$ the
(restriction of) the projection to $I$. 

Define
\bee
\wt{\cU}_0 := \{(x',\ve') : 0 \leq \ve < \ve_0, \delta^{-1} < |x'| < \delta
\ve^{-1}\},\;\; \wt{\cV}_0 = \psi^{-1}\wt{\cU}_0
\eee
and
\bee
\wt{\cU}_1 =S^2 \times H \subset \cB\mbox{ where }
H = \{(\rho,\sigma) \in [0,\delta)^2 : \rho\sigma < \ve_0\}.
\eee
Put $\wt{\cV}_1 = \phi^{-1}\wt{\cU}_1$ so that $\wt{\cV}_1 \subset
\wt{\cW}_1$.  

Then
\bee\label{bs_def}
\bs : \wt{\cU}_1 \to \wt{\cU}_0,\;\;  \bs :(y,\rho,\sigma) \mapsto (x' = y/\rho,
\ve' = \rho\sigma)
\eee
is a diffeomorphism. For fixed $\ve>0$, $\bs$ restricts to the
scaling map $x' = x/\ve$ of \S\ref{improved} (Figure~\ref{fig_glue}).

As in \S\ref{improved}, $\bs$ is covered by a diffeomorphism
$\vc{\kappa}: \cV_1 \to \cV_0$.  We can now make the following
definition:

\begin{dfn}
Define $\wt{\cW}$ as the quotient of the disjoint union
\bee
\wt{\cW}_0 \amalg \wt{\cW}_1
\eee
by the identification $\vc{\kappa}$.  The involution $\iota$ lifts
from the pieces to define an involution on $\wt{\cW}$ and the
quotient $\wt{\cW}/\iota$ is denoted $\cW$.
\label{d1.29.11.20}\end{dfn}

\begin{rmk}\label{r1.2.12.20}  From the above construction and by a slight
abuse of notation we regard $\cW_1$ and $\cW_0$ as open subsets of
$\cW$, intersecting in $\cV_1$.
\end{rmk}

We claim that this is the manifold whose existence was claimed in
Proposition~\ref{list1}. 

\subsection{Proof of Proposition~\ref{list1}}

We check the claimed properties one by one.
\begin{enumerate}
\item[(i)] The maps $\pi_0: \cW_0 \to I$, $\pi_1 : \cW_1 \to I$
  defined above fit together under the gluing $\vc{\kappa}$, 
yielding 
a smooth proper map $\pi : \wt{\cW}
  \to I$.
\item[(ii)] This assertion about the boundary hypersurfaces of
  $\wt{\cW}$ is  clear from the construction: $\wt{X}_p =
  \phi^{-1}(F_p)$, $\wt{X}_{\ad} = \phi^{-1}(\wt{Y})$, $\wt{I}_\infty
  = \phi^{-1}(S^2 \times I)$, and $\wt{X}_0 = \AHr$ from the gluing of
  $\wt{\cW}_0$ to $\wt{\cW}_1$, restricted to $\phi^{-1}(F_0)\subset
  \wt{\cW}_1$ and 
  $\{\ve'=0\}\subset \wt{\cW}_0$.
\item[(iii)--(v)] are clear from the definition.
\item[(vi)] The involutions of $\cW_1$ and $\cW_0$ are intertwined by
  $\vc{\kappa}$, and so glue together to define an involution $\iota$
  of $\wt{\cW}$. 
\end{enumerate}

Setting $\cW = \wt{\cW}/\iota$, we obtain the `gluing space' on which
we shall construct our smooth families of hyperK\"ahler metrics in the
next sections.

  \section{Formal solution}
  \label{formalsol_sec}
From now on, $\cW$ will be as constructed in the previous section, and
the notation will be as there.   The goal of this section is the
construction of smooth families of approximate hyperK\"ahler triples on the fibres of
$\pi$ inside $\cW$, whose limits, at $\ve=0$ are the triples of
$g_{\AH}$ at $X_0$, $g_{\TN}$ at the other $X_\nu$ and
the  triple of $g_{\ad}$ at $X_{\ad}$.

The steps in this
construction are an initial approximation,
followed by an iterative argument that improves this approximation
order by order in $\ve$. 

We begin with some necessary technical preliminaries about fibrewise
symplectic and hyperK\"ahler triples on $\cW$, and then proceed to the
statement of the main result of this section.

\begin{notn} (Naming the parts of $\cW$.)
The boundary hypersurfaces of $\cW$ are
denoted by $X_\nu$, for $\nu=0,\ldots,k$, $X_{\ad}$ and
$I_\infty$, corresponding to the boundary hypersurfaces of $\wt{\cW}$.
Boundary defining functions for these hypersurfaces will
  be denoted $\sigma_\nu$ for $X_\nu$, $\rho$ for
  $X_{\ad}$ and $\sigma_I$ for $I_\infty$.   For small positive
  $\delta$, let $U_\nu = \{\sigma_\nu < \delta\}$ be a tubular
  neighbourhood of $X_\nu$ and $V = \{\rho < \delta\}$ be a tubular
  neighbourhood of $X_{\ad}$. 

We may and shall
  assume:
\begin{enumerate}
\item $\rho$ is circle-invariant and is equal to $1/|x'_p|$ near
  $X_{\nu}\cap X_{\ad}$;
\item $\sigma_\nu = |x-p_\nu|$ near $X_{\nu}\cap X_{\ad}$;
\item $\rho\sigma_0\cdots\sigma_k = \ve$;
\item   $\sigma_\nu = \ve$ in 
$\Omega_{\nu}: =U_\nu \setminus V$ and $\rho = \ve$ 
in an open set of the form $V\setminus \bigcup_{\nu} U_{\nu}$.
\end{enumerate}
Set
  $\sigma = \sigma_0\sigma_1\ldots \sigma_k\sigma_I$.

From the construction of $\cW$, for each $\nu$ there is a retraction
$r_\nu : U_{\nu} \to X_{\nu}$ (cf.\ Prop.~\ref{plift1}).   Near $X_{\nu}\cap X_{\ad}$, this is a
map of circle-bundles covering the rescaling map $(x,\ve) \mapsto
(x-p_\nu)/\ve$. 
  \label{n_bdf} \end{notn}

\subsection{Fibrewise symplectic and hyperK\"ahler triples on $\cW$}

\begin{notn} If $U\subset \cW$ is an open set,  denote by $\Omega^k_{\phi}(U)$ the space of smooth
  sections over $U$ of $\Lambda^kT^*_\phi(\cW/I)$.  The {\em relative
    differential} is denoted
\begin{equation}\label{e1.12.3.20}
  \rd_\pi : \Omega^k_{\phi}(U) \longrightarrow \Omega^{k+1}_{\phi}(U).
\end{equation}
Further, if $U$ meets $X_{\ad}$,  write   $\Omega^k_{\phi,\ei}(U)$ for the subspace of {\em
  essentially invariant} 
  forms.  This is the subspace of forms $\alpha$ such that
  $\cL_{\p_\theta}\alpha = O((\rho\sigma_I)^\infty)$.  Write
$\Omega^k_{\phi,\eb}(U)$ for the subspace of {\em essentially basic}
forms.  These are the essentially invariant forms $\alpha$ which also
satisfy $\iota_{\p_\theta}\alpha = O((\rho\sigma_I)^\infty)$. 
\end{notn}

\begin{rmk} Observe that if $\alpha$ is essentially invariant (resp.\
  essentially basic) then  $\rd_\pi\alpha$ is essentially invariant  (resp.\ essentially basic).
\end{rmk}

\begin{dfn} 
By a {\em symplectic triple} $\bom$ on an open subset $U$ of $\cW$ we shall always
mean a triple $(\bom_1,\bom_2,\bom_3)$ with $\bom_j \in
\Omega_\phi^2(U)$, such that $\rd_\pi \bom_j=0$ and such that the
$3\times 3$ matrix
\begin{equation}\label{e1.13.3.20}
(\omega_j\wedge \omega_k)\mbox{ is positive-definite at every point of }U.
\end{equation}
A symplectic triple on $U\subset \cW$ is called {\em hyperK\"ahler} 
if  $Q(\bom)=0$, where
\begin{equation}\label{e2.13.3.20}
  Q(\bom)_{jk} = \bom_j\wedge \bom_k -
  \frac{1}{3}(\bom_1^2+\bom_2^2 +\bom_3^2)\delta_{jk}.
  \end{equation}
\end{dfn}

\begin{rmk} While it would be more accurate to call the triples
  appearing in this definition {\em relative} symplectic or hyperK\"ahler
  triples, we believe that no serious confusion will result from this
  definition.  However, the reader should bear in mind that symplectic
  triples on $\cW$ are to be thought of informally as {\em
    $\ve$-dependent smooth families} of
  symplectic triples on the `Sen space' $\SEN_k$ with some rather
  strong control on their behaviour in the limit as $\ve\to0$.  As
  previously in this paper, we shall try to be consistent in our use
  of bold symbols for $\ve$-dependent families, viewed as data on the
  5-dimensional space $\cW$.
  
  The $3\times 3$ matrices appearing in \eqref{e1.13.3.20}
  and \eqref{e2.13.3.20} take values in the trivial real line-bundle
  $\lambda:=\Lambda^4 T^*_\phi(\cW/I)$. The condition that \eqref{e1.13.3.20} be
  positive-definite makes sense for one and hence any trivialization
  of this bundle.
\end{rmk}

\begin{notn}\label{notn1.22.3.20}  

 For $\nu=0,1,\ldots,k$, we shall denote by $\omega_\nu$ the
 hyperK\"ahler triple of $g_\nu$.  (We omit the straightforward proof
 that the triples discussed in \S\ref{s_primitives_gh} and
 \S\ref{sect_AH_epsilon} extend to define triples on the rescaled 
 tangent bundles $T_\phi X_{\nu}$ over the compactifications.)  We
 also denote by $\omega_{\ad}$ the hyperK\"ahler triple of the
 $g_{\ad}$ on $X_{\ad}$. 

\end{notn}

The main theorem to be proved in this section is the following:
\begin{thm} On $\cW$, there is a smooth symplectic triple $\bze$ 
satisfying
  \begin{equation}\label{e1.17.1.20}
Q(\bze) \in \dot{C}^\infty(\cW: S^2\RR^3\otimes \lambda)
\end{equation}
and such that
\bee\label{e1.17.11.20}
\bze|X_{\ad} = \omega_{\ad},\; \bze|X_\nu = \omega_{\nu} \mbox{ for
}\nu=0,1,\ldots,k.
\eee
\label{formal_thm}
\end{thm}

\begin{rmk}  Recall that $f \in \dot{C}^\infty(\cW)$ means that $f$ is
  smooth with all 
  derivatives rapidly decreasing at all boundary hypersurfaces of
  $\cW$. 
\end{rmk}

\subsection{Initial approximation}

\begin{lem}\label{l11.2.12.20}
Denote by $\bom_1$ the lift to $\cW_1$ of the hyperK\"ahler triple
\eqref{e11.1.10.19} of the family of metrics $g_{\ve}$.  Then $\bom_1$
extends to a smooth triple on $\cW_1\subset \cW$, with 
\bee
\bom_1|X_{\nu} = \omega_\nu\mbox{ for }\nu=1,\ldots,k
\eee
and
\bee\label{e1.2.12.20}
\bom_1|X_{\ad} = \omega_{\ad}.
\eee
\end{lem}

Here we regard $\cW_1$ as an open subset of $\cW$, cf.\
Remark~\ref{r1.2.12.20}. In \eqref{e1.2.12.20} 
\bee\label{e2.2.12.20}
\omega_{\ad,1} = \alpha_{\ad}\wedge e_1 + e_2\wedge e_3,
\eee
and similarly for the other components of the triple, where $e_j = \rd
x_j/\ve$ and $\alpha_{\ad} = \alpha|X_{\ad}$.  Thus $\omega_{\ad}$ is
the hyperK\"ahler triple of $g_{\ad}$.

\begin{proof}
We verify that all the pieces in the formula \eqref{e11.1.10.19} lift
smoothly as claimed.  

Write $\bh$ for the lift to $\wt{\cB}$ of the function
$(x,\ve) \mapsto h_{\ve}(x)$.  The from Prop.~\ref{plift1}, the term
$\ve/|x-p|$ has a smooth lift to $[\ol{\RR}^3\times I ;
(p,0)]\setminus \Sigma_p$, and
this lift vanishes on the  (lifts of) $\{\ve=0\}$ and $S^2\times I$
(spatial infinity). Near the front face it is equal to $1/|x'_p|$.
The lift (really pullback) of this term to $\wt{\cB}$ is therefore
smooth away from $\Sigma_p$ and vanishes at all boundary hypersurfaces
other than $F_p$.   This shows:
\bee\label{e1.29.11.20}
\bh|\wt{Y} = 1, \bh|F_p  = 1 +\frac{1}{2|x'_p|}, \bh|S^2\times I = 1.
\eee

We saw in Prop~\ref{plift2} that the lifts of the $\ve \p_{x_j}$ are
smooth on the blow-up.  From the local forms of the lifts, it is easy
to see that they are in fact smooth sections of $T_{\phi}(\cW/I)$ near
$X_{\ad}$. (One
needs to check Def.~\ref{def_rescaled}). If $\theta$ is a local fibre
coordinate then $\p/\p\theta$ is also a section of $T_{\phi}(\cW/I)$
and it is not hard to see that the $\ve\p_{x_j}$ and $\p_\theta$
together span $T_{\phi}(\cW/I)$ near $X_{\ad}$. Away from $X_{\ad}$,
$T_{\phi}(\cW/I)$ is just the ordinary $\pi$-vertical tangent bundle
$T(\cW/I)$.  Dually, the $1$-forms $\rd x_j/\ve$ and $\alpha$ locally
span $T^*_{\phi}(\cW/I)$ near $X_{\ad}$, and taking into account what
we have already proved about $\bh$, it follows that $\bom_1$ is smooth
near $X_{\ad}$ and also that \eqref{e1.2.12.20} holds.

In a set of the form $|x'_{\nu}| < R$ ($x'_{\nu} = (x-p_{\nu})/\ve$)
the triple is smooth in the $x'_{\nu}$ and restricts to define the
Taub--NUT triple (by \eqref{e1.29.11.20}). 

\end{proof}

Similarly, the family $\omega_{\AH,\ve}$ of hyperK\"ahler triples for
the metrics $g_{\AH,\ve}$ fit together to form a smooth triple on
$\cW_0$:

\begin{lem}
The definition
$$
\bom_0|\pi_0^{-1}(\ve) := \omega_{\AH,\ve}
$$
determines a smooth hyperK\"ahler triple $\bom_0$ on the subset $\cW_0$ of $\cW$,
with $\bom_0|X_0 = \omega_{\AH,0}$.
\label{l1.29.11.20}\end{lem}

\begin{proof}
Over the interior of $X_0$, in other words in any subset of the form
$\{\rho > \rho_0\}$, where $\rho_0>0$, there is nothing to check.
Near the corner $\rho = 0$, we have the triple of a Gibbons--Hawking
metric, up to exponentially small terms, and the same observations
used in the proof of Lemma~\ref{l1.29.11.20} apply here as well.
\end{proof}

Our initial approximation to $\bze$ is furnished by a simple patching
of $\bom_0$ and $\bom_1$:
\begin{prop}\label{p1.17.1.20}
  There exists a smooth symplectic triple $\bom^\chi$ on $\cW$ satisfying
  \begin{equation}\label{e7.14.3.20}
    Q(\bom^\chi)  \in \ve^3\rho^\infty\sigma_I^\infty
    C^\infty(\cW, S^2_0\RR^3\otimes \lambda).
    \end{equation}
and such that
\bee\label{e1.3.4.20}
\bom^\chi|X_{\ad} = \omega_{\ad},\;\bom^\chi|X_\nu = \omega_\nu
 \mbox{ for
}\nu=0,1,\ldots,k.
\eee
(Recall that we have defined $\lambda$ to be the `relative density
bundle' $\lambda = \Lambda^4T^*_\phi(\cW/I)$.)
\end{prop}  

\begin{proof}  By Lemma~\ref{l11.2.12.20}, $\bom_1$ is smooth in
  $\cW_1$ and satisfies the boundary conditions at $X_{\ad}$ and the
  $X_{\nu}$ except $\nu=0$.  Thus $\bom^{\chi}$ will be obtained by
  patching $\bom_0$ to $\bom_1$, modifying these forms only in a
  neighbourhood $U_0\cap V$ of the corner $Z=X_{0}\cap X_{\ad}$. 

Recall Corollary~\ref{prim_1} gave the formula $\omega_{\ve} =
\Omega_{\ve} + \rd b_{\ve}$ near $x=0$, where $\Omega_{\ve}$ is the
triple \eqref{e51.2.4.20} associated to 
the family of harmonic functions 
\bee\label{e3.17.11.20}
1 + \mu \ve -\frac{2\ve}{|x|}. 
\eee
Then the family $\Omega_{\ve}$ gives a smooth triple $\bom_Z$, say, in
$U_0 \cap V$ and Cor.~\ref{prim_1} and Prop.~\ref{prim_2} give
\begin{equation}\label{e1.14.3.20}
  a \in \rho^{\infty}\Omega^1_{\phi}(U_0\cap V)\otimes \RR^3\mbox{
    such that } \bom_0 = \bom_Z + \rd_\pi a  
\end{equation} 
and
\begin{equation}\label{e2.14.3.20}
b \in \sigma_0^3\Omega^1_{\phi,\eb}(U_0\cap V)\otimes \RR^3\mbox{
    such that } \bom_{\ad} = \bom_Z + \rd_\pi b,
\end{equation}
in $U_0\cap
V$.

Let $\chi(t)$ be a smooth cut-off function 
equal to $1$ for
$t\leq \delta/2$ and vanishing for $t\geq \delta$, $\chi(t)\in [0,1]$
for all $t$, and define
\begin{equation}\label{e11.14.3.20}
  \bom^\chi  = \bom_Z +\rd_\pi(\chi(\rho)a +\chi(\sigma_0)b).
\end{equation}
This triple is initially 
defined only in $U_0\cap V$, but the claim is that it
extends naturally to the whole of $\cW$ (possibly after shrinking
$\ve_0$). 

To prove the claim, 
consider first the restriction of $\bom^\chi$ to $\{\sigma_0< \delta/2\}$.
Then $\chi(\sigma_0)=1$ and so
\bee\label{e21.17.1.20}
\bom^\chi  = \bom_Z + \rd_{\pi} a + \rd_{\pi} (\chi(\rho)b) 
= \bom_0 + \rd_{\pi} (\chi(\rho)b) \mbox{ in
}[0,\delta/2)_{\sigma_0}\times [0,\delta)_\rho.
\eee
Both terms on the RHS extend to the neighbourhood $\{\sigma_0<
\delta/2\}$ of $X_0$.

Similarly, we have
\bee \label{e22.14.3.20}
\bom^{\chi} = \bom_1 + \rd_\pi(\chi(\sigma_0)a)\mbox{ in }\{\rho < \delta/2\}
\eee
and so $\bom^{\chi}$ extends smoothly to $\cW_1$. Observe that
$\bom^\chi$ is identically equal to $\bom_0$ in $\cW_0 \setminus
U_0\cap V$ and is identically equal to $\bom_1$ in $\cW_1 \setminus
U_0\cap V$. 

To prove \eqref{e7.14.3.20} and \eqref{e1.3.4.20}, note quite generally that
\bee\label{l1.3.4.20}
\alpha \in \rho^n\sigma_0^m\Omega^k_{\phi}(U_0\cap V) \Rightarrow
\rd_\pi \alpha \in \rho^{n+1}\sigma_0^{m+1}\Omega^{k+1}_\phi(U_0\cap
V),\eee
as follows from the explicit forms of the basis vectors of
$T_{\phi}(\cW/I)$ near $Z$.  Then \eqref{e1.14.3.20} and
\eqref{e2.14.3.20} give
\begin{equation}\label{e23.14.3.20}
  \rd_\pi(\chi(\sigma_0)a) = O(\rho^\infty)\mbox{ and }
\rd_\pi(\chi(\rho)b) \in \rho\sigma_0^3\Omega^2_{\phi,\eb}.
\end{equation}
Using these in \eqref{e21.17.1.20} and \eqref{e22.14.3.20} now gives
\eqref{e1.3.4.20}. 

It remains to verify \eqref{e7.14.3.20}.  To do so, observe first that
$\bom^\chi$ is smooth, and will be a symplectic triple if $\delta$ is
chosen small enough.  It is clear that $Q(\bom^\chi)$ will also be
smooth.  Because of this {\em a priori} smoothness, it is sufficient
to compute the order of vanishing of $Q(\bom^\chi)$ near $X_0$ and $X_{\ad}$ but away from the
corner.  In $U_0\setminus X_{\ad}$,
 use \eqref{e21.17.1.20}.  Then
\begin{eqnarray}
  Q(\bom^\chi) &  = &Q(\bom_0) + 2Q(\bom_0, \rd_{\pi}(\chi(\rho)b)) +
  Q(\rd_{\pi}(\chi(\rho)b), \rd_{\pi}(\chi(\rho)b)) \\
&=& 2Q(\bom_0, \rd_{\pi}(\chi(\rho)b)) + O(\rho^\infty\sigma_0^6)
\end{eqnarray}
because $\bom_0$ is hyperK\"ahler and $Q$ applied to an essentially
basic form is $O(\rho^\infty)$ for degree reasons.  
The first term is zero because
$\bom_0$ is hyperK\"ahler, and the second is $O(\rho\sigma_0^3)$ for $\sigma_0
\to0$ away from $\rho=0$ (cf.\ \eqref{e23.14.3.20}).  In similar fashion,
we see that $Q(\bom^\chi) =
O(\rho^\infty)$ for $\rho\to 0$ away from $\sigma_0=0$.  Smoothness
now gives that $Q(\bom^\chi) = O(\ve^3\rho^\infty)$ in $U_0 \cap V$
since $\ve = \rho\sigma_0$ in this set.   This error term is
moreover supported in $U_0\cap V$ which gives the additional vanishing
in $\sigma_I$.   
\end{proof}

\begin{notn}  In the interest of readability we shall write $\rd$ for
  $\rd_{\pi}$ in the rest of this section.
\end{notn}

Theorem~\ref{formal_thm} is proved by induction.  The inductive
assumption is that we have found
\begin{equation}\label{e1.14.8.19}
\bc  \in  \sigma_I^2\Omega^1_{\phi,\eb}(\cW)\otimes \RR^3
\end{equation}
such that
\begin{equation}\label{e2.14.8.19}
Q(\bom^\chi + \ve \rd \bc) =
\ve^NF + \ve^{N+3}G,
\end{equation}
where
\bee\label{e2a.14.8.19}
F \in \rho^\infty\sigma^\infty_I C^\infty(\cW, S^2_0\RR^3\otimes
\lambda),\;\;
G \in\sigma^\infty
C^\infty_{\ei}(S^2_0\RR^3\otimes \lambda). 
\eee
Equations \eqref{e2.14.8.19} and
\eqref{e2a.14.8.19} imply that the restriction of $\bom^\chi +
\ve\rd\bc$ to $\pi^{-1}(\ve)$ is a symplectic triple that is
`approximately hyperK\"ahler to order $\ve^N$'.
We need to keep track of the fine structure of the error term as in
\eqref{e2a.14.8.19} to be sure of the
smoothness of the triple $\bze$ that is claimed in
Theorem~\ref{formal_thm}.

We shall construct $\ba$ defined near $\bigcup X_\nu$ and $\bbb$ defined near
$X_{\ad}$, essentially basic and decaying near spatial infinity, so that with
\begin{equation}\label{e3.14.8.19}
\bc' =  \bc + \ve^{N-1}\rd \ba + \ve^{N+1}\rd \bbb,
\end{equation}
we have
\begin{equation}\label{e4.14.8.19}
Q(\bom^\chi + \ve \rd \bc')  = \ve^{N+1}F' + \ve^{N+4}G',
\end{equation}
where $F'$ and $G'$ are in the same spaces as $F$ and $G$ in
\eqref{e2a.14.8.19}.  Thus we have improved the error term in
\eqref{e2.14.8.19} by one order in $\ve$. 

The induction starts because of Proposition~\ref{p1.17.1.20}, which is
the case $N=3$ of \eqref{e2.14.8.19}.  Given \eqref{e2.14.8.19}, the
required $\ba$ and $\bbb$ are obtained by solving a Poisson equation respectively
over $\bigcup X_\nu$ and on the base $Y_{\ad}$ of 
the $S^1$-bundle $X_{\ad}\to Y_{\ad}$.

\subsection{Construction of $\ba$: linear theory}
\label{ba_construction_linear}
We gather in this subsection the linear theory of the equation
$Q(\omega,\rd a) = f$, on an ALF gravitational instanton $X$, where
the RHS is rapidly decreasing near $\p X$.

The following is an explicit version of the infinitesimal
diffeomorphism gauge invariance of the linearized equations:

\begin{lem}  Let $X$ be a hyperK\"ahler $4$-manifold
  with hyperK\"ahler triple $\omega$.  For any vector field 
$v$,  we have $Q(\omega,\rd(\iota_v\omega)) = 0$.
\label{l2.17.1.20}\end{lem}

\begin{proof}
The equation is gauge-invariant, so we have
$Q(\phi^*_t(\omega))=\phi_t^*Q(\omega) = 0$ where
$\phi_t$ is the one-parameter family of diffeomorphisms generated by
$v$.  Then the derivative at $t=0$ of $\phi_t^*(\omega)$ is just
$\rd(\iota_v\omega)$ by Cartan's formula, and the Lemma follows at once.
\end{proof}

\begin{thm}\label{th_poisson_1}
Let $(X,g)$ be a strongly ALF gravitational instanton with hyperK\"ahler triple $\omega$.
Suppose that $f \in \dot{C}^{\infty}(X,S^2_0\RR^3\otimes\lambda)$.
(Recall this means that all derivatives of $f$ vanish at $\p X$.)

Then there exists
\begin{equation}\label{e11.25.3.20}
u \in \rho^2\Omega^1_{\phi,\eb}(X)\otimes\RR^3\mbox{ such that }
  Q(\omega,\rd u) = f.
\end{equation}
(Here $\rho$ is the boundary defining function of $\p X$.)
\end{thm}

\begin{proof}  
 As in \S\ref{triple_sec}, regard $f$ as a section of
 $(\Lambda^0\oplus \Lambda^2_+)\otimes \RR^3$ using the triple to
 identify $\Lambda^2_+$ with the product $\RR^3$-bundle over $X$ and
 $\sum \omega_i^2$ to trivialize $\lambda$.   If
 we solve $DD^*\phi = f$ where $D = \rd^* + \rd_+$ is the Dirac
 operator as in \eqref{e7.30.9.19}, then $u_0 = D^*\phi$ will solve
 $Du_0 = f$.   Now the bundle $(\Lambda^0\oplus \Lambda^2_+)$ can also
 be identified with the product $\RR^4$-bundle over $X$, again using
 the triple $\omega$.  Thus $DD^*$ is just
 $12$ copies of the scalar laplacian and each component of
 $DD^*\phi = f$ can be solved for
 \begin{equation}
 \phi \in  \rho C^{\infty}_{\ei}(X, (\Lambda^0\oplus
 \Lambda^2_+)\otimes \RR^3
 \eee
by invoking Theorem~\ref{t11.28.3.20}.
Then $u_0 = D^*\phi$ is $O(\rho^2)$, essentially invariant, and $Du_0
=f$, which also implies $Q(\omega,\rd u_0) =f$.

In order to get an essentially basic solution $u$, we shall find a
vector field $v$, supported near $\p X$, such that 
\begin{equation}\label{e12.17.11.20}
u = u_0 + \iota_v\omega
\end{equation}
is essentially basic.  By Lemma~\ref{l2.17.1.20}, we shall still have
$Q(\omega,\rd a) =f$.   In an asymptotic Gibbons--Hawking chart with local
coordinates $(x,\theta)$, write
$$
u_0 = u_{00}\alpha + \mbox{terms in }\rd x_j.
$$
Here $u_{00}$ is a triple $(v_1,v_2,v_3)$ of essentially invariant functions defined
near $\p X$: regard this triple as a vector field $v_j\p_{x_j}$
defined near $\p X$, and it is easy to check that
$$
\iota_v \omega = -u_{00}\alpha + \mbox{terms in }\rd x_j.
$$
Then $u$ defined by \eqref{e12.17.11.20} is essentially
basic as required (we use a cut-off function to extend $v$ smoothly
away from the given neighbourhood of $\p X$).
\end{proof}

\subsection{Construction of $\ba$}  

\begin{prop}
\label{iteration1}
Given $\bc$ satisfying \eqref{e1.14.8.19} and \eqref{e2.14.8.19}, there
exists $\ba \in \rho^2\sigma^\infty \Omega^1_{\phi,\eb}(\cW)$ such that 
\begin{equation}\label{e5.14.8.19}
Q(\bom^\chi + \ve \rd \bc + \ve^{N}\rd \ba)  \in \ve^{N+1}F' + \ve^{N+3}G',
\end{equation}
where $F'$ and $G'$ are in the spaces shown in \eqref{e2a.14.8.19}.

Moreover $\ba$ can be chosen to be supported arbitrarily close to 
$\bigcup X_\nu$ (and in particular away from spatial infinity $I_\infty$).
\end{prop}
\begin{proof}

Let us write $\bom' = \bom^\chi + \ve \rd \bc$.  If
\begin{equation}\label{e1.15.3.20}
\ba \in \rho^2\sigma_I^2\Omega^1_{\phi,\eb}(\cW)
  \end{equation}
  then we calculate
  \begin{equation}\label{e1.24.3.20}
    Q(\bom' + \ve^N\rd \ba) = Q(\bom') +
    \ve^N Q(\bom',\rd \ba) + \ve^{2N}Q(\rd \ba,\rd \ba).
  \end{equation}
  Since $\ba$ is essentially basic, the third term on the RHS is
$O(\ve^{2N}\rho^\infty)$.  Since the correction term $\bc$ is also
essentially basic, the second term on the RHS differs from
$\ve^{N}Q(\bom_\nu,\rd \ba)$ by $O(\ve^{N+1}\rho^{\infty})$ in each
collar neighbourhood $U_\nu$.   Thus, using $F'$ and $G'$ to denote
elements of the spaces \eqref{e2a.14.8.19} that are allowed to vary
from line to line, \eqref{e1.24.3.20} can be rewritten
\begin{equation}\label{e1.25.3.20}
Q(\bom' + \ve^N\rd \ba) = \ve^N F' + \ve^NQ(\bom_\nu, \rd \ba) + \ve^{N+3}G'
  \end{equation}
  in each $U_\nu$.
  
This equation has a well-defined leading term at $X_\nu$ obtained by
dividing by $\ve^N$ and taking the limit $\sigma_\nu \to 0$.   The
leading term, $f_\nu$ say, of $F'$ at $X_\nu$ is equal to the leading
term of $Q(\bom')$ at $X_\nu$.  Thus 
the leading term of the RHS of \eqref{e1.25.3.20} at $X_{\nu}$ is
\begin{equation}
  f_\nu + Q(\omega_\nu,\rd \ba|X_\nu).
  \end{equation}
Now $f_\nu$ is $O(\rho^\infty)$ on $X_{\nu}$ so by 
Theorem~\ref{th_poisson_1}, there exists a solution $a_\nu\in
\rho^2\Omega^1_{\phi,\eb}(X_\nu)\otimes \RR^3$ so that
\begin{equation}
  Q(\omega_\nu,\rd a_\nu) = - f_\nu
\end{equation}
on $X_\nu$.  Define $\ba$ in $U_\nu$ to be
$\chi(\sigma_\nu)r^*_\nu(a_\nu)$, $r_\nu$ from Notation~\ref{n_bdf}.  As the neighbourhoods $U_\nu$
are pairwise disjoint, we may regard this as defining $\ba$ over the
whole of $\cW$, and so defined, $\ba$ is supported in the union of the
$U_\nu$.  We claim that $\ba$ satisfies
\eqref{e5.14.8.19}.   

In $U_\nu$, we have, from \eqref{e1.25.3.20},
\begin{eqnarray}
  Q(\bom' + \ve^{N}\rd \ba) &=& \ve^N(f_\nu +
                                Q(\bom^{\chi},\rd(\chi(\sigma_{\nu}) r^*_\nu(a_\nu))
+ \ve^{N+1}F' + \ve^{N+3}G' \\
  &=& 
              \ve^NQ(\bom^\chi,\rd\chi(\sigma_{\nu})\wedge
      r_\nu^*a_\nu) 
      +    \ve^{N+1}F'     + \ve^{N+3}G',
\end{eqnarray}
by choice of $f_{\nu}$.  Because $\chi(\sigma_{\nu})$ is identically
equal to $1$ on $X_{\nu}$, $\rd \chi(\sigma_{\nu}) =
O(\rho\sigma^{\infty}_{\nu})$, using \eqref{l1.3.4.20}).  Since
$a_\nu = O(\rho^2)$,  the first term on the RHS is
$O(\sigma^\infty\ve^N\rho^3) = O(\ve^{N+3}\sigma^\infty)$. Hence this
term can be absorbed by the $\ve^{N+3}G'$ term, and \eqref{e5.14.8.19}
is proved.
\end{proof}

\subsection{Construction of $\bbb$:  linear theory}

In this section we summarize the linear theory for the Laplacian of
the adiabatic family of metrics $g_{\ve}$.   By a straightforward
calculation, for
\begin{equation}\label{e3.25.3.20}
  g =  h_{\ve}\,\frac{|\rd x|^2}{\ve^2} + h^{-1}_{\ve}\alpha^2,
\end{equation}
we have
\begin{equation}\label{e2.7.8.18}
\Delta_{g_{\ve}} = \ve^2h^{-1}\wt\Delta_0 - h \p_\theta^2,
\end{equation}
where $\Delta_0$ is the laplacian of the (unrescaled) euclidean metric,
$\wt{\Delta}_0$ is its horizontal lift and $\p_\theta$ denotes the
generator of the circle action.  As an aside, if in local coordinates,
$$
\alpha = \rd\theta + \sum a_j\rd x_j
$$
then
$$
\wt{\Delta}_0 = -\sum \nabla_j^2,\mbox{ where }
\nabla_j = \frac{\p}{\p x_j} - a_j\frac{\p}{\p\theta}.
$$

In particular, if $u$ is invariant, then regarding it without change
of notation as a function on $\RR^3$, we have
\bee\label{e101.24.3.20}
\Delta_{g_{\ve}}u = \ve^2h^{-1}\Delta_0 u.
\eee

Recall that $D = \rd^*+\rd_+$,
$$
D : \Omega^1 \longrightarrow \Omega^0\oplus\Omega^2_+.
$$
As in \S\ref{triple_sec}, on a hyperK\"ahler $4$-manifold $M$, $\Lambda^2_+$ has a flat
orthonormal trivialization by a hyperK\"ahler triple $(\omega_j)$.  Using this trivialization,
if
$$
\phi = (\phi_0, \phi_j\omega_j) \in  \Omega^0\oplus\Omega^2_+
$$
then $DD^*$ acts as the scalar Laplacian on the coefficients
$(\phi_0,\ldots, \phi_3)$.

We shall need the formula for $D^*$ for the metric
\eqref{e3.25.3.20}, acting on invariant functions.  A simple
calculation gives that if
$$
D^*\phi = w_0 e_0 + w_j e_j,
$$
where
$$
e_0 = \frac{\alpha}{\sqrt{h}},\;\; e_j = \sqrt{h}\,\frac{\rd x_j}{\ve},
$$
then
\begin{equation}\label{e11.24.3.20}
\begin{bmatrix} w_0 \\ w_1 \\ w_2 \\w_3 \end{bmatrix} = 
\frac{\ve}{\sqrt{h}}\begin{bmatrix} 0 & \p_1 & \p_2 & \p_3 \\ 
-\p_1 & 0 &-\p_3 & \p_2 \\ 
-\p_2 & \p_3 & 0 &-\p_1 \\
-\p_3 & -\p_2 & \p_1 & 0
\end{bmatrix} 
\begin{bmatrix} \phi_0 \\ \phi_1 \\ \phi_2 \\ \phi_3 \end{bmatrix}.
\end{equation}

(One can verify $DD^* = \Delta_g$ directly from this formula and its adjoint.)

In the next theorem, we write $\bom_{\ad}$ for the lift of the hyperK\"ahler
triple of \eqref{e3.25.3.20} to a collar neighbourhood $V$ of
$X_{\ad}$ in $\cW$. The notation for boundary defining functions is as
in paragraph \ref{n_bdf}.

\begin{thm}
  Let
  \bee \label{e5.25.3.20}
  G \in  \sigma^{\infty}C^\infty_{\ei}(V,S^2_0\RR^3\otimes \lambda)
  \eee
  be an essentially invariant section.  Then there exists
  \bee\label{e6.25.3.20}
  \bbb \in \sigma^2_I\Omega^1_{\phi,\eb}(V)\otimes \RR^3
  \eee
such that
\bee\label{e21.25.3.20}
  Q(\bom_{\ad},\rd \bbb) = \ve G + O(\ve^\infty\sigma_I^\infty)
 \eee
in $V$.
\label{poisson_2}\end{thm}

\begin{proof}  Suppose that $G$ is exactly $S^1$-invariant. Then we
  may regard $G$ as a function on $V/S^1$.  Because $G$ 
  vanishes to all orders in $\sigma_\nu$ near $X_\nu$, we may regard
  $G$ as a smooth function on $\RR^3\times [0,\ve_0)$, which
  vanishes to all orders in $|x-p|$ at every point $p$ of $P$.

  We solve \eqref{e21.25.3.20} in two stages.  First, the formula
\bee\label{e7.25.3.20}
\bu(x,\ve) = \frac{1}{4\pi\ve}\int
\frac{1}{|x-y|}h_\ve(y)G(y,\ve)\,\rd y
\eee
gives a function on $\RR^3$ such that
\bee\label{e8.25.3.20}
\Delta_{g_{\ve}}\bu(x,\ve) = \ve G.
\eee
Moreover, $\ve \bu(x,\ve)$ is smooth in all variables because the
singularities in $h$ at the points of $P$ are cancelled by the
vanishing of $G$ to all orders at these points. It is also
$O(|x|^{-1})$ for $|x|\to \infty$.

As before, regard $G$ and $\bu$ as sections of $(\Lambda^0\oplus
\Lambda^2_+)\otimes \RR^3$.  Then if $\bbb_0 = 
D^*_{g_\ve}\bu$, we have  $D_{g_{\ve}}\bbb_0 = \ve G$.  
From the formula \eqref{e11.24.3.20}, $\bbb_0$, is initially defined
on $\RR^3\setminus P$ and lifts to a smooth section of
$\Omega^1_{\phi,\ei}(V)\otimes \RR^3$.  Indeed,
near each of the $X_\nu$, $h_\ve = 1 +O(\sigma_\nu)$, where the `$O$' is smooth
for small $\sigma_\nu$.   

We can now correct
$\bbb_0$ exactly as we did in Theorem~\ref{th_poisson_1} to obtain
$\bbb$ satisfying \eqref{e6.25.3.20} and \eqref{e21.25.3.20}.

We started the proof by assuming that $G$ was exactly invariant. If
$G$ is only essentially invariant, we can write $G =
G_0+G_1$ where $G_0$ is exactly invariant and $G_1$ is
$O(\ve^\infty\sigma_I^{\infty})$. Then we apply the previous argument
with $G$ replaced by $G_0$ to obtain the result. 
\end{proof}

\subsection{Construction of $\bbb$}
\label{bbb_construction}
The second half of the inductive step is contained in the following
result:

\begin{prop}
Let $\bc$ and $\ba$ be as in Proposition~\ref{iteration1}.  Then there exists $\bbb$,
supported near $\rho=0$, essentially basic and $O(\sigma^2)$,
so that
\bee\label{e3.24.3.20}
  Q(\bom'  + \ve^N \rd \ba + \ve^{N+2}\rd \bbb) = \ve^{N+1}F'' +
  \ve^{N+4}G'',
\eee
where $F''$ and $G''$ are in the spaces in \eqref{e2a.14.8.19}.
\label{iteration2}\end{prop}
\begin{proof}
  Let us write $\bom'' = \bom' + \ve^N\rd \ba$.  If we calculate the LHS
  of \eqref{e3.24.3.20} in $V$, assuming that $\bbb$ is essentially
  basic in $V$, our collar neighbourhood of $X_{\ad}$,
  we obtain
  \begin{equation}\label{e32.25.3.20}
    Q(\bom'' + \ve^{N+2}\rd\bbb) = Q(\bom'') +
    \ve^{N+2}Q(\bom'',\rd\bbb) + O(\ve^{2N+4}\rho^\infty),
  \end{equation}
  where as before the last term vanishes to all orders in $\rho$
  because $\bbb$ is essentially basic.  Using $F''$ and $G''$ for
  generic functions in the spaces \eqref{e2a.14.8.19} which may vary
  from line to line, simplifications analogous to those in the proof
  of Proposition~\ref{iteration1} yield
\begin{equation}\label{e31.25.3.20}
    Q(\bom'' + \ve^{N+2}\rd\bbb) = \ve^{N+1}F'' + \ve^{N+3}G'' +
    \ve^{N+2}Q(\bom_{\ad},\rd\bbb).
  \end{equation}
Now $G''$ satisfies the hypotheses of $G$ in Theorem~\ref{poisson_2} so
there exists $\bbb$ as in \eqref{e6.25.3.20} and satisfying
\eqref{e21.25.3.20} (with $G$ replaced by $-G'$).   In order to extend
$\bbb$ to $\cW$, we must replace it by $\chi(\rho)\bbb$.  Then
\begin{eqnarray}
  Q(\bom'' + \ve^{N+2}\rd(\chi(\rho)\bbb)) &=& \ve^{N+1}F' +
                                               \ve^{N+4}G'' +
  \ve^{N+2}Q(\bom_{\ad}, \rd\chi\wedge \bbb).
\end{eqnarray}
The last term is $O(\rho^\infty)$ because $\chi$ is identically $1$
near $X_{\ad}$ and $\bbb$ is smooth near each $X_\nu$. Thus this last
term can be absorbed into $\ve^{N+1}F'$, yielding \eqref{e3.24.3.20}.
\end{proof}

\subsection{Completion of proof of Theorem~\ref{formal_thm}}

The inductive argument given in
\S\S\ref{ba_construction_linear}--\ref{bbb_construction} shows that
there is a solution $\wh{\bze}$ of \eqref{e1.17.1.20} in formal power
series in $\ve$.  However, it is well known that given such a formal power
series, there exists smooth $\bze$ whose derivatives at all boundary
hypersurfaces agree with those of $\wh{\bze}$ (Borel's Lemma).  This
observation completes the proof of Theorem~\ref{formal_thm}.

\begin{rmk}
Our proof yields rather more than what is stated in
Theorem~\ref{formal_thm}.  From the construction we have seen that
$\bze-\bom^\chi$ is smooth,
essentially basic and $O(\ve\rho^2\sigma^2_I)$ on $\cW$.
\end{rmk}

  \section{Completion of Proof}

\label{completion}

To complete the proof of our main theorem, we need to modify $\bze$ in
Theorem~\ref{formal_thm} by a triple $\ba \in
\Omega^1_{\phi}(\cW)\otimes \RR^3$, say, so that
\begin{equation}
Q(\bze + \rd \ba) = 0
\end{equation}
on (the fibres of) $\cW$.  

This is an application of the implicit function theorem, uniformly on
the fibre $\pi^{-1}(\ve)$, for $\ve>0$.  The key step is the uniform
invertibility result for the linearization, Theorem~\ref{t1.1.4.20},
below. To have this invertibility, we shall need to use the freedom to
choose $\ve_0$ to be very small.

Let us agree to denote by $\bg_\zeta$ the metric on $T_\phi(\cW/I)$
determined by the symplectic triple $\bze$, and by $\Delta_{\zeta}$
the associated Laplacian.

We shall use the reformulation $\ba
= D^*_{\zeta} \bu$ discussed in \eqref{e12.2.4.20}--\eqref{e6.26.3.20}
to reduce our work to the study of the {\em scalar} Laplacian $\Delta_{\zeta}$ of $\bg_{\zeta}$.

\subsection{(Fibrewise) differential operators on $\cW$}

The $\phi$-vertical tangent bundle $T_\phi(\cW/I)$ was introduced in
Definition~\ref{def_rescaled}.  Denote by $\cV_{\phi}(\cW/I)$ the space of
smooth sections of $T_\phi(\cW/I)$.    This space of vector fields is used to
define the relevant space of differential operators on $\cW$:

\begin{dfn}\label{d2.26.3.20} The space $\Diff^m_{\phi}(\cW/I)$ is the
space of differential operators which are polynomial (of degree $\leq
m$) in the vector fields from $\cV_\phi(\cW/I)$ with smooth
coefficients.  The space $\Diff^m_{\bo}(\cW/I)$ is the set of
differential operators on $\cW$ which are polynomial in
$\cV_{\bo}(\cW/I)$,  the $\bo$-vector
fields on $\cW$ that are tangent to the fibres of $\pi$.
\end{dfn}

These definitions also make sense for differential operators acting
between sections of vector bundles over $\cW$.  From the definitions, we
see that $(\rho\sigma_I)^m\Diff^m_{\bo}(\cW/I) \subset
\Diff^m_{\phi}(\cW/I)$. 

\begin{ex} We have already seen the relative exterior
derivative $\rd_\pi$ as an example of an operator in 
$\Diff^1_{\phi}(\cW/I)$. 

  If $\bg$ is a smooth metric on $T_\phi(\cW/I)$, then the Laplacian
  $\Delta_{\bg}$ is an operator in $\Diff^2_{\phi}(\cW/I)$.
\end{ex}

\begin{dfn} A function $f$ on $\cW$ is said to be {\em essentially
    invariant} if
\bee\label{e61.5.4.20}
  \frac{\p f}{\p \theta} = O((\rho\sigma_I)^\infty)
 \eee
on $\cW$. 
The definition makes sense for a given choice of $S^1$-action near
$X_{\ad}\cup I_\infty$.   The subclass $\Diff^m_{\phi,\ei}(\cW/I)$
consists of those operators in $\Diff^m_{\phi}(\cW/I)$ with
essentially invariant coefficients.
\end{dfn}

\subsection{Function spaces}

Let $\cW$ be as before.  Pick a smooth $\bo$-density on $\cW$ and
introduce the space $L^2_{\bo}(\cW)$ and the Sobolev spaces $H^m_{\bo}(\cW)$,
\begin{equation}
  H^m_{\bo}(\cW) = \{u\in L^2_{\bo}(\cW) :  Pu \in L^2_{\bo}(\cW) \mbox{ for
    all }P \in \Diff^m_{\bo}(\cW)\}.
\end{equation}

Let $H^{n,m}_{\phi,\bo}(\cW)$ be the space of functions $u$ such that $Pu \in
H^m_{\bo}(\cW)$ for all $P \in \Diff^n_{\phi}(\cW/I)$.  Then by
definition, $P$ extends to define a bounded map $H^{n,m}_{\phi,\bo} \to
H^m_{\bo}$ for any $m$.  Define also $H^{n,m}_{\bo,\bo}(\cW)$ to be the
space of $u$ such that $Qu \in H^m_{\bo}(\cW)$ for all
$Q\in\Diff^n_{\bo}(\cW/I)$.

In order to deal with the Laplacian, we shall need to split off the
$S^1$-invariant component of elements in these 
Sobolev spaces.  This only makes sense near $X_{\ad}\cup I_\infty$,
which means that our definitions are a little complicated.  Recall
that $V$ is a fixed tubular neighbourhood of $X_{\ad}$.
\begin{dfn}
  Let $\beta > \alpha+2$, $\alpha>0$, and fix a bump function
  $\chi(t)=1$ for $t\leq 1/2$ and equal to zero for $t\geq 1$.  Define  
  \bee\label{e1.26.3.20}
  \cR_{\alpha,\beta,m}(\cW) = \{\chi(\rho/\delta)f_0 + f_1\}\mbox{ where }
f_0 \in (\rho\sigma_I)^{\alpha+2}H^m_{\bo}(V)\mbox{ and }\frac{\p 
    f_0}{\p\theta} =0,\; f_1 \in (\rho\sigma_I)^\beta H^m_{\bo}(\cW). 
  \eee 
\label{d3.26.3.20}\end{dfn}
Since $\alpha+2<\beta$, this allows for the zero-Fourier mode $f_0$
to be larger than the non-zero Fourier modes.  In practice, we shall
take $\alpha \in (0,1)$ and $\beta$ can be as large as we like.  The
space just defined will serve as a {\em range space} for the
Laplacian.  The definition of the domain is similar:

\begin{dfn} With $\alpha$ and $\beta$ as in
  Definition~\ref{d3.26.3.20}, let
\bee\label{e2.26.3.20}
\cD_{\alpha,\beta,m+2}(\cW) = \{\chi(\rho/\delta)u_0 + u_1\}\mbox{ where }
u_0 \in (\rho\sigma_I)^{\alpha}H^{2,m}_{\bo,\bo}(V) \mbox{ and }
 \frac{\p u_0}{\p\theta} = 0,\; u_1 \in 
 (\rho\sigma_I)^\beta H^{2,m}_{\phi,\bo}(\cW).
\eee
\end{dfn}

The Laplacian of a smooth, essentially invariant metric on
$T_\phi(\cW/I)$ extends to define a bounded linear map 
$$\cD_{\alpha,\beta,m+2}(\cW)\to \cR_{\alpha,\beta,m}(\cW)
$$
for every $m$   and $\beta > \alpha+2$.

The main linear result to be given in this section is the
invertibility of the Laplacian between these spaces.

\begin{thm}\label{t1.1.4.20}   Let $\bg_\zeta$ be the metric on
  $T_\phi(\cW/I)$ determined by the symplectic triple $\bze$ on $\cW$, and
  let $\Delta_\zeta$ be the associated Laplacian.  Fix
$\alpha\in (0,1)$ and $\beta >   \alpha+2$ and $m$.  Then there exists
$\ve_0>0$ so that with $\cW = \pi^{-1}[0,\ve_0)$,
  \bee\label{e111.5.4.20}
  \Delta_\zeta : \cD_{\alpha,\beta,m+2}(\cW) \longrightarrow \cR_{\alpha,\beta,m}(\cW)
\eee
is invertible.
\end{thm}

\begin{proof}
  This follows by patching inverses on the $X_\nu$ to an
  `adiabatic' inverse on $V$.  For this we first need to localize functions
  on $\cW$ near the different boundary hypersurfaces.

  To simplify notation,  having fixed $\alpha,\beta$ and $m$, write
  \bee\label{e5.28.3.20}
\cR = \cR_{\alpha,\beta,m},\;\; \cD = \cD_{\alpha,\beta,m+2}
  \eee
and write $\cR(\cW)$, $\cR(X_\nu)$ etc.\ to distinguish between function
spaces on $\cW$ and on $X_\nu$ (cf.\
\eqref{e11.3.4.20}--\eqref{e12.3.4.20} below).

As before, let $\chi(t)$ be a standard cut-off function, $0\leq
\chi(t) \leq 1$, $\chi(t)=1$ for $t\leq \frac12$ and vanishing for $t\geq
1$.  Let $\delta >0$ be small.  Let $\chi_\nu =
\chi(\sigma_\nu/\delta)$ and let $\chi_{\ad} = 1 - \sum_\nu
\chi_\nu$.   If $f\in C^\infty(\cW)$ then
$\chi_\nu f$ is smooth and supported in $U_\nu$ and $\chi_{\ad} f$ is
smooth and supported in $V$.

Now we identify $U_\nu$ with a subset of the product $X_\nu \times
[0,\ve_0)$ by the map
$w \mapsto (\kappa_\nu(w) ,\pi(w))$, where $\kappa_\nu$ was introduced
in \eqref{notn1.22.3.20}.   Under this identification, a set of the
form $\sigma_\nu < a$ maps to the 
subset
\bee\label{e112.5.4.20}
\{(x'_\nu,\theta_\nu,\ve) : |x'_\nu| < a\ve^{-1}\} \subset X_\nu\times[0,\ve_0),
\eee
and in particular $\chi_\nu f$, when transferred to the product, will
be compactly supported in each slice $X_\nu \times \{\ve\}$ for
$\ve>0$ (though these compact sets grow as $\ve\to0$).

From Theorem~\ref{t11.26.3.20}, we have an inverse $G_\nu :
\cR(X_\nu) \to \cD(X_\nu)$ of the Laplacian $\Delta_{g_\nu}$.  Using
the identification of $U_\nu$ with the product \eqref{e112.5.4.20}, now
  define a lift $\bG_\nu$ of $G_\nu$ to act on functions on $\cW$ by the formula:
  \bee
  \bG_\nu(f) = \eta_\nu G_\nu \chi_\nu f,
  \eee
  where\footnote{We have here suppressed explicit mention of the
    map identifying $U_\nu$ with the product}
  \bee\label{31.17.11.20}
  \eta_\nu = \chi(\log \sigma_\nu/\log \delta).
  \eee
 Then $\eta_\nu$ goes from $1$ to $0$
as $\sigma_\nu$ goes from $\delta$ to $\sqrt{\delta}$ and in
particular $\eta_\nu$ is identically $1$ on $\supp(\chi_\nu)$, so
\bee\label{e21.28.3.20}
\eta_\nu\chi_\nu = \chi_\nu.
\eee
From the boundedness of $G_\nu: \cR(X_\nu)\to \cD(X_\nu)$, it follows that
\bee\label{e1a.28.3.20}
\bG_\nu: \cR(\cW) \to \cD(\cW)\mbox{ is bounded},
\eee
and that
\bee\label{e1.28.3.20}
\Delta_{\bg}\bG_\nu = \chi_\nu - e_{\nu}.
\eee
The key point is that we can choose $\delta$ so that 
\bee\label{e3.28.3.20}
  \mbox{ the operator norm of }e_\nu:\cR(\cW) \to \cR(\cW)\mbox{ is
    bounded by } \frac{1}{10(k+1)} + C_\nu(\delta)\ve_0.
  \eee
To prove this, let $f \in \cR(\cW)$ and observe that $\bG_\nu f$ is
supported in $U_\nu$, and in this set, $\Delta_{\zeta} =
\Delta_{\bg_{\nu}} + O(\ve_0)$, where $g_\nu$ is the original ALF hyperK\"ahler
metric on the hypersurface $X_\nu$.  Thus
\begin{eqnarray}
  \Delta_{\zeta}\bG_{\nu} &=& \Delta_{g_{\nu}}\bG_{\nu} + O(\ve_0)
                            \nonumber \\
                        &=& \Delta_{g_\nu}\eta_\nu G_\nu \chi_\nu + O(\ve_0)
                            \nonumber \\
                        &=& \eta_\nu \Delta_{g_\nu}G_\nu \chi_\nu 
                            +[\Delta_{g_\nu}, \eta_\nu]G_\nu\chi_\nu +
                            O(\ve_0)
                            \nonumber \\
                        &=& \chi_\nu 
                            +[\Delta_{g_\nu}, \eta_\nu]G_\nu\chi_\nu +
                            O(\ve_0) \label{e22.28.3.20}
\end{eqnarray}
using \eqref{e21.28.3.20} to obtain the first term in
\eqref{e22.28.3.20}.   The commutator is an operator in
$\Diff^1_{\phi}(\cW/I)$ of the form
\bee\label{e23.28.3.20}
[\Delta_{g_{\nu}},\eta_\nu] = c_0 + c_1\nabla,
\eee
where $c_0 = \Delta_{g_\nu}\eta_\nu$ and $c_1 = \nabla \eta_{\nu}$. 
Now
for any given smooth function $\beta$ with compact support in $U_\nu$, 
$\beta G_\nu\chi_\nu$ defines a bounded linear map $\cR\to \cD$.
Noting that functions in $\cD$ decay like $\rho^{\alpha}$ while
functions in $\cR$ decay at the faster rate $\rho^{\alpha+2}$, in
order that \eqref{e23.28.3.20} have small norm, we require that
the coefficient of $\nabla$ be bounded by $o(\delta)\rho$ and the
order-$0$ term be bounded by $o(\delta)\rho^2$.  This is where the
specific formula for $\eta$ comes in. Indeed, from \eqref{l1.3.4.20} we have
\bee\label{e24.28.3.20}
\rd\eta_\nu = \chi'\left(\frac{\log \sigma_\nu}{\log \delta}\right)\frac{\rd
  \sigma_\nu}{\sigma_\nu\log \delta}
\eee
and since the norm of $\sigma_\nu^{-1}\rd \sigma_\nu$ is $O(\rho)$, this
term is bounded by a multiple of $\rho/|\log\delta|$.  Similarly
$|\nabla^2\eta| = O(\rho^2/|\log\delta|)$.    Thus the operator norm
of $[\Delta_{g_\nu},\eta_\nu]G_\nu$ is controlled by $1/|\log
\delta|$, and by taking $\delta$ sufficiently small, we can make the
operator norm of this term 
$<\frac{1}{10(k+1)}$.  Then
\eqref{e3.28.3.20} is obtained.

We now need to complete the definition of $\bG$ by finding an
approximate inverse localized near $V$.   For this, set
\bee\label{e201.5.4.20}
\eta_{\ad}(x) = 1 - \sum_{\nu} \chi\left(\frac{\log
    2\sigma_{\nu}}{2\log\delta}\right)
\eee
so that $\eta_{\ad}$ is identically $1$ on the support of $\chi_{\ad}$
and goes from $1$ to $0$ as any of the $\sigma_\nu$ goes from
$\delta/2$ to $\delta^2/2$.   We shall construct an operator
$\bG_{\ad}$ as a sum $\sum\eta_{\ad} G_n \chi_{\ad}$ where $G_n$ acts
on the $n$-th Fourier coefficient of $\chi_{\ad}f \in \cR(\cW)$.  Our
operator will have properties analogous to those of $\bG_{\nu}$,
\bee
\bG_{\ad}: \cR(\cW) \to \cD(\cW)\mbox{ is bounded and }
\Delta_{\bg}\bG_{\ad} = \chi_{\ad} - e_{\ad}
\eee
and we can choose $\delta$ so that 
\bee\label{e8.28.3.20}
  \mbox{ the operator norm of }e_{\ad}:\cR(\cW) \to \cR(\cW)\mbox{ is
    bounded by }\frac{1}{10} + C_{\ad}(\delta)\ve_0.
  \eee
For the construction of $\bG_{\ad}$ we 
are localized to $V$, 
we have the $S^1$-action and  for $f\in \cR(\cW)$ we may
split $\chi_{\ad}f$ into its Fourier modes.  For the zero Fourier mode,
define $u_0 \in \cD(\cW)$ by the formula
\bee\label{e1.5.4.20}
u_0 = \ve^{-2}\eta_{\ad}G_0(\chi_{\ad}f_0),
\eee
where $G_0$ is the Green's operator of the euclidean Laplacian
$\RR^3$.  It is not hard to check that this is bounded between the
given spaces, and
\bee\label{e2.5.4.20}
(\ve^2\Delta_0)u_0 = [\Delta_0,\eta_{\ad}] G_0 \chi_{\ad}f_0
+ \chi_{\ad}f_0.
\eee
As discussed above, the operator norm of the first term can be made as
small as we please by choosing $\delta$ sufficiently small: the
derivatives of $\eta_{\ad}$ give factors of $1/|\log\delta|$ in the
coefficients of the commutator. 

On the $n$-th Fourier mode we invert the model operator
\bee\label{e3.5.4.20}
\ve^{2}\Delta_0 + n^2
\eee
using the explicit Green's operator
\bee\label{e4.5.4.20}
G_n(x-x') = \frac{e^{-|n||x-x'|/\ve}}{4\pi|x-x'|}
  \eee
  on $\RR^3$.    Then the formula
  \bee\label{e5.5.4.20}
  \wh{u}_n(x) = \eta_{\ad}(x)\int G_n(x-x')\chi_{\ad}(x')\wh{f}_n(x')\,\rd
  x',\; n\neq 0
  \eee
gives the $n$-th Fourier coefficient of a function $u$ in
$\cD(\cW)$ if $\chi_{\ad}\wh{f}_n$ is the $n$-th Fourier
coefficient of $\chi_{\ad}f$, with  $f \in (\rho\sigma_I)^\beta H^m_{\bo} \subset
\cR_{\alpha,\beta,m}(\cW)$.

Combining the definitions \eqref{e1.5.4.20} and \eqref{e5.5.4.20}, we obtain an operator $\bG_{\ad}$, which is
a bounded linear map from $\cR(\cW) \to
\cD(\cW)$.

Now, in $V$, $\Delta_{\zeta}$ differs from $\Delta_{\ad}$ by an
operator $\rho A$  where  $A\in \Diff^{2}_{\phi}(\cW/I)$,
\bee\label{e6.5.4.20}
\Delta_{\zeta} = \Delta_{\ad} + \rho A\mbox{ in }V.
\eee
Then
\bee\label{e7.5.4.20}
\Delta_{\zeta}\bG_{\ad} = \chi_{\ad} +O\left(\frac{1}{|\log\delta|}\right) + \rho A\bG_{\ad}.
\eee
Choose $\delta$ so small that the operator norm of the second term on
the RHS is less than $\frac{1}{10}$.   With $\delta$ fixed in this
way, the support of $A\bG_{\ad}$ is bounded away from the $X_\nu$ and
so $\rho$ can be bounded here by $C_{\ad}(\delta)\ve_0$.

Choosing $\delta$ to satisfy \eqref{e3.28.3.20} and \eqref{e8.28.3.20},
and defining
  \bee\label{e11.5.4.20}
  \bG = \sum_{\nu} \bG_{\nu} + \bG_{\ad},
  \eee
we have a bounded
operator $\cR(\cW)\to \cD(\cW)$ with the property
\bee\label{e12.5.4.20}
\Delta_{\zeta}\bG = 1 - e_\zeta
\eee
where the operator norm of $e_\zeta$ has the form $\frac{1}{5} + C\ve_0$.
Thus, picking $\ve_0$ sufficiently small, $1-e_\zeta$ is invertible and
$\bG(1-e_\zeta)^{-1}$ is the required inverse.
\end{proof}

\begin{rmk}
The glued inverse operator $\bG$ appears to depend upon $m$, in that
in general if $m$ is increased, we shall need to take $\delta$
smaller.  However, the argument shows that if
\bee
f \in \bigcap_{m,n\geq 0} \ve^n\cR_{\alpha,\beta,m}
\eee
then the solution $u$ of $\Delta_\zeta u =f$ given by the Theorem will
lie in the intersection
\bee
u \in \bigcap_{m,n\geq 0} \ve^n\cR_{\alpha,\beta,m+2}.
\eee
\end{rmk}

\begin{thm}
  Let $\bze$ be as in Theorem~\ref{formal_thm}. Then there exists
  $\ve_0>0$ and 
  \bee
  \ba \in \ve^{\infty}\sigma_I^2\Omega^1_{\phi,\ei}(\cW)\otimes \RR^3
  \eee
  such that $\bze +\rd \ba$ is a hyperK\"ahler triple on $\cW$.
\end{thm}

\begin{proof}
We obtain a finite-regularity solution by the implicit function
theorem, and then iterate to obtain smoothness.  Seek $\ba = D^*G\phi$,
where $\bG_{\zeta}$ is the inverse of $\Delta_{\zeta}$ from
Theorem~\ref{t1.1.4.20}. 

By \eqref{e6.26.3.20} we require
$\phi$ to solve the nonlinear equation
\bee\label{e7.26.3.20}
\phi = -e - \wh{r}(\rd D^*G\phi),
\eee
where $e = Q(\bze)$ and $\wh{r}$ is quadratic.

We find a solution $\phi \in
\cR_{\alpha,\beta,m}$.   Since $\rd D^*$ maps $\cD_{\alpha,\beta,m+2}$
into $\cR_{\alpha,\beta,m}$, we need to know that $u \mapsto u\otimes
u$ is bounded from $\cR$ to $\cR$. If we choose $m > 5/2$ (remember
that $\cW$ is $5$-dimensional) then since the weights force decay, this
this is indeed satisfied.  Because $Q(\bze)$ is smooth and rapidly
decreasing in $\ve$ and $\sigma_I$, by taking $\ve_0$ small,
$Q(\bze)$ can be arranged 
to have very small norm in any fixed $\cR_{\alpha,\beta,m}(\cW)$.
Because $\wh{r}$ is quadratic, $\phi \mapsto - e -\wh{r}(\rd
D^*G\phi)$ is a contraction for $\ve_0$ small enough, giving 
a solution to \eqref{e7.26.3.20} for any given $m$.

For the regularity statement,  note first that the solution will be smooth in
any bounded open subset of the interior of $\cW$ by elliptic regularity,
because $Q(\bze)$ is itself smooth. To see boundary regularity
consider first the boundary components $X_\nu$ and $X_{\ad}$, staying
away from spatial infinity $I_\infty$. The solution for a given $m$ is
defined in a subset $\ve< \ve_0$, and by construction this solution
has $\bo$-regularity of order $m$ at $\pi^{-1}(0)$. The solution is
also $O(\ve^n)$ for every $n$, because this is true of $Q(\bze)$.   If
we pass to a larger value $m_1>m$, then we obtain a different solution
$u_1$ defined in $\ve < \ve_1$, where in general, $\ve_1 <
\ve_0$. However, the solution is unique, so $u_0|\{\ve< \ve_1\} =
u_1$, and it follows that $u_0$ also has $\bo$-regularity of order
$m_1$. Hence in any neighbourhood $O$ of any boundary point of
$\pi^{-1}(0)$, $u \in \ve^{\infty}H^\infty_{\bo}(O)$ and so is smooth
and vanishes to all orders at the boundary of $O$.

For the regularity at $I_\infty$, we need to take a closer look at the
asymptotic form of $\bg_{\zeta}$.  We claim that mod
$O((\ve\sigma_I)^\infty)$, $\bg_\zeta$ is given by the
Gibbons--Hawking Ansatz near $I_\infty$.

Recall that a hyperK\"ahler metric in $4$ dimensions can be expressed
using the Gibbons--Hawking Ansatz if it admits an isometric
triholomorphic $S^1$-action. The
  euclidean coordinates $x_j$ emerge as the three components of the
  hyperK\"ahler moment map for this action, and the Gibbons--Hawking
  form of the metric follows by using these coordinates. 

  If we write $\bze = \bze_0 + \bze_1$ and correspondingly $\bg_\zeta
  = \bg_0 + \bg_1$ where $\bze_0$ and $\bg_0$ are exactly
  $S^1$-invariant, then the error terms $\bze_1$ and $\bg_1$ will be
  $O((\ve\sigma_I)^\infty)$.    Now $\bze$ is a modification of
  $\bom_{\ad}$ by essentially basic forms, which means that
  \bee
  \iota_{\p_\theta}(\bze-\bom_{\ad}),\;\;
    \iota_{\p_\theta}(\bze_0-\bom_{\ad}) \mbox{ are }
    O((\ve\sigma_I)^\infty) \mbox{ near }I_\infty.
    \eee
Thus the $x_j$ are approximate moment maps for $\bg_0$ and following
through the Gibbons--Hawking Ansatz we obtain a harmonic function
$\bh$, say, defined near $I_\infty$, such that
\bee
\bg_0 = \bh\frac{|\rd x|^2}{\ve^2} + \bh^{-1}\alpha_\zeta^2.
\eee
The smoothness of decaying solutions of the Laplace equation now
follows as it did for $X_\nu$ in Theorem~\ref{t11.28.3.20}.
\end{proof}

\section*{Acknowledgements}
Both authors  thank Lorenzo  Foscolo and  Sergey Cherkis for numerous
discussions and helpful comments. Singer is grateful to Richard
Melrose for  inspiring discussions over many years. In particular we
explored gluing theorems for K\"ahler metrics of constant scalar
curvature from the point of view of the present paper in work which
unfortunately remains unpublished.  He also acknowledges useful
discussions with Pierre Albin, Joel Fine, Jesse Gell-Redman, Peter
Hintz,  Rafe Mazzeo and Andr\'as Vasy.  This material is based upon work supported by the National Science Foundation under Grant No. 1440140, while the second author was in residence at the Mathematical Sciences Research Institute in Berkeley, California, during the Fall Semester of 2019.

\appendix

\section{The manifold  $\CP_1\times \CP_1$  and  dual ellipses }
\label{cpapp}
\subsection{Homogeneous coordinates and projection operators}
The non-compact manifolds obtained from $\CP_1 \times \CP_1$ by removing the diagonal and anti-diagonal  $\CP_1$   can be interpreted in terms of oriented ellipses in two  dual ways. The goal of this appendix is to derive this picture, which was used in \S\ref{AHreview} to explain the geometry and mutual relationships of the  spaces $\AHd,\AH,\AHu$ and $\AHr$.

We begin by fixing  notation for  the vector space $\CC^2$ and its projective space $\CP_1$. We write  
\bee
\label{wzfirst}
 z= \begin{pmatrix} z_1 \\ z_2\end{pmatrix},  \quad 
w=  \begin{pmatrix} w_1 \\ w_2\end{pmatrix}, 
\eee
for   elements $z,w \in \CC^2$ which also serve as homogeneous coordinates for $\CP_1$. We will require both the anti-symmetric bilinear form 
\bee
z \wedge w = z_1w_2-w_1z_2 
\eee
and the  sesquilinear form $\langle \cdot ,\cdot \rangle$ on $\CC^2$ , linear in the second argument, given by
\bee
\langle w, z\rangle = \bar w_1 z_1 + \bar w_2 z_2. 
\eee
Defining
\bee
w^\perp= \begin{pmatrix}- \bar w_2 \\   \phantom{-}\bar w_1 \end{pmatrix} ,
\eee
we  note that $
z \wedge w^\perp = \langle w, z \rangle$.

We are interested in the   non-compact manifolds obtained from $\CP_1 \times \CP_1$ by removing the diagonal or anti-diagonal, i.e.,
\begin{align}
\CP_1 \times \CP_1\setminus \CPD & =\{(z,w) \in \CC^2\times \CC^2|z\wedge w \neq 0\}/\sim \; ,\nonumber \\
\CP_1 \times \CP_1\setminus \CPAD & =\{(z,w) \in \CC^2\times \CC^2|z\wedge w^\perp\neq 0\}/\sim \;,\end{align}
where $\sim$  is division by the scaling  action of  $\CC^*\times \CC^* $  on  $(z,w)$. 
A convenient description of these quotient spaces  is in terms of the projection operators
\bee
P(z,w)  =  \frac{1}{w\wedge z}  z\langle \bar w^\perp,   \cdot\rangle ,  \qquad
Q(z,w)=\frac{1}{ \langle w,  z\rangle }  z \langle w, \cdot \rangle, 
\eee
naturally representing points in, respectively, $\CP_1 \times \CP_1\setminus \CPD$  and $\CP_1 \times \CP_1\setminus \CPAD$.  
Clearly  $P^2=P$ and $Q^2=Q$. With 
\bee
Q^\dagger (z,w)
= \frac{1}{ \langle z,  w\rangle }  w \langle z, \cdot \rangle, 
\eee
 we also note the  identities
\bee
\label{projid}
PQ =Q, \qquad 
PQ^\dagger =0,
\eee
and 
\bee
\label{projidual}
QP =P, \qquad 
QP^\dagger =0. 
\eee
If we now define  traceless  $2\times 2$ matrices $M$ and $N$  via
\bee
P= \frac 12 (\text{id} + M ), \qquad Q =\frac 12 (\text{id} +N),
\eee
then 
\bee
\label{XYsquare}
M^2=N^2=\text{id},
\eee
as well as 
$Q^\dagger  =  \frac 12 (\text{id} + N^\dagger )$. The identities  \eqref{projid} are equivalent to 
\bee
\label{keyresults}
MN=\text{id} +N-M, \qquad M N^\dagger +  N^\dagger + M + \text{id}=0.
\eee
 Before we leave the discussion of the projectors $P$ and $Q$, we note that 
 \bee 
 \label{PQherm}
P ^\dagger (z,w) = P(z,w)   \Leftrightarrow w^\perp=z, \qquad  Q^\dagger (z,w) = Q (z,w)   \Leftrightarrow w=z,
\eee
 so that $P$ is Hermitian precisely on the anti-diagonal inside $\CP_1 \times \CP_1\setminus \CPD$ and $Q$ is Hermitian precisely on the diagonal inside $\CP_1 \times \CP_1\setminus \CPAD$.

\subsection{Ellipses in euclidean space}
To obtain the description of $\CP_1 \times \CP_1\setminus \CPD$ and $\CP_1 \times \CP_1\setminus \CPAD$ in \S\ref{AHreview}, we   use the  Pauli matrices 
\bee
\sigma_1= \begin{pmatrix} 0 & 1 \\ 1 & 0\end{pmatrix}, \quad \sigma_2= \begin{pmatrix} 0 & -i \\ i  & \phantom{-} 0\end{pmatrix}, \quad \sigma_3= \begin{pmatrix} 1  & \phantom{-} 0 \\ 0 & -1\end{pmatrix},
\eee
 to expand   
\bee
M= X_1\sigma_1 + X_2\sigma_2 + X_3\sigma_3, \qquad N= Y_1\sigma_1 + Y_2\sigma_2 + Y_3\sigma_3,
\eee
and  then assemble the Cartesian components  into vectors  $X= (X_1,X_2,X_3)^t, Y=(Y_1,Y_2,Y_3)^t $ in $ \CC^3$,
with  real and imaginary parts
\bee
X= \tx +i \xi, \quad Y=y+i \eta.
\eee
Then the constraint \eqref{XYsquare} implies
\bee
\label{xyconstraints}
|\tx|^2= |\xi|^2 +1,\quad \tx\cdot \xi =0  \qquad |y|^2= |\eta|^2 +1, \quad y\cdot \eta=0.
\eee

We thus arrive at  pairs of  vectors $(\tx, \xi)$ and $(y,\eta)$ satisfying the constraints  \eqref{xyconstraints} as natural coordinates on, respectively $\CP_1\times \CP_1\setminus \CPD$ and  $\CP_1\times \CP_1\setminus \CPAD$.  They make explicit the isomorphisms 
\bee
\CP_1\times \CP_1\setminus \CPD \simeq T^*S^2,\qquad \CP_1\times \CP_1\setminus \CPAD \simeq TS^2, 
\eee
 with  $m =\tx/|\tx|, n=y/|y|$    taking values on the  round sphere in  euclidean space, and $\xi$ and $\eta$ being co-tangent and tangent vectors at $m$ and $n$. It follows from \eqref{PQherm}  that 
$X=m$   on the anti-diagonal $\CPAD$ inside $\CP_1\times \CP_1\setminus \CPD$, so that $m$ is a natural coordinate there. Similarly, $Y=n$   on the diagonal $\CPD$ inside $\CP_1\times \CP_1\setminus \CPAD$, so that $n$ is a natural coordinate there. This picture is  consistent with the self-intersection numbers of the  zero-section of $TS^2$ and the  diagonal 
 inside $\CP_1\times \CP_1$  both being $+2$, and the self-intersection numbers of  the  zero-section of $T^*S^2$ and the   anti-diagonal   inside $\CP_1\times \CP_1$ both  being $-2$.
 
The  coordinates $(\tx,\xi)$  and $(y,\eta) $  naturally parametrise oriented ellipses  up to scale in euclidean space,  which we call the $X$-  and $Y$-ellipse. 
In this interpretation,   $\tx$ and $\xi$ are major and minor axes of the $X$-ellipse,   
 while  $y$ and $\eta$  are the   major and minor axes of the $Y$-ellipse. When $\xi=0$ 
 the $X$-ellipse degenerates into a line along $\tx$ and when $\eta=0$, the $Y$-ellipse degenerates into a line along $y$. The special case of the ellipse becoming a circle is only obtained in the limit of $|\xi| \rightarrow \infty$ for the $X$-ellipse and $|\eta| \rightarrow \infty $ for the $Y$-ellipse.

 We can now  state and prove the main result of this appendix. 
  \begin{lem}
   The $X$- and $Y$-ellipses   are dual to each other in the sense that 
   \bee
 \label{axisdeff}
 \tx = \frac{y\times  \eta}{|\eta|^2}, \quad \xi =   - \frac{ \eta}{ |\eta|^2},
 \eee 
where $|\eta|\neq0$. 
 This map is an involution of  $\CP_1\times \CP_1 \setminus (\CPAD\cup \CPD)$ , where we also have 
   \bee
 \label{axisdef}
 y = \frac{\tx\times  \xi}{|\xi|^2}, \quad \eta =   - \frac{ \xi}{ |\xi|^2}.
 \eee 
  In particular, the degeneration of the $X$-ellipse into a line corresponds to the degeneration of the $Y$-ellipse into a  circle at right angles to that line,   and conversely. 
 \end{lem}
{\em  Proof}: 
We deduce from \eqref{keyresults}  that  the Hermitian matrices $X$ and $Y$ characterising the $X$- and $Y$-ellipses satisfy
\bee
\label{first}
 M(N -  N^\dagger)= 2\, \text{id} + N + N^\dagger.
\eee
Writing $\sigma$ for the vector with cartesian component $\sigma_1,\sigma_2, \sigma_3$, we 
have 
\bee
  N-N^\dagger =2i \eta\cdot \sigma, \quad (N- N^\dagger)^2 = -4|\eta|^2
 \eee
 showing that $N-N^\dagger$ is invertible when $|\eta|\neq 0$, with  inverse
 \bee
 ( N-N^\dagger)^{-1} =- \frac{N-N^\dagger}{4|\eta|^2}. 
 \eee
 Multiplying  \eqref{first} from the right by this inverse, and using 
and 
\bee
 NN^\dagger -  N^\dagger N   = 4  y\times\eta \cdot \sigma
\eee
we deduce
\begin{align}
M& =-\frac{(N+ N^\dagger  +2)(N-  N^\dagger)}{4\eta^2} \nonumber \\
 &= \frac{NN^\dagger   -  N^\dagger   N -2 (N-N^\dagger)}  {4|\eta|^2} \nonumber \\
 &=\frac{y\times\eta \cdot \sigma - i \eta\cdot \sigma}{|\eta|^2}, 
\end{align}
 which is  the claimed relation \eqref{axisdef}. The proof of  \eqref{axisdeff}  is elementary, but also follows from the dual relation of projectors \eqref{projidual}. One checks that the degeneration of the $X$-ellipse into lines along $\tx$ as $|\xi|\rightarrow 0$  makes the $Y$-ellipse degenerate into a circle  of infinite radius in the plane orthogonal to $\tx$. Conversely,  in the limit $|\eta|\rightarrow 0$ the $Y$-ellipses degenerate  into lines along $y $ while the  $X$-ellipses become  circles of infinite radius in the orthogonal plane.
 \hfill $\Box$ 

\subsection{Symmetries}
In  \S\ref{AHreview}  of the main text  we also make use of discrete symmetries of $\CP_1 \times \CP_1\setminus \CPAD$. The factor switching map
\bee
s:  (z,w) \mapsto (w,z)
\eee
 induces the map $Q\mapsto Q^\dagger$ at the level of projectors and hence 
 \bee
 s: TS^2 \rightarrow TS^2, \quad Y\mapsto  \bar Y.
 \eee
 When $|\eta|\neq 0$, it also  induces the map $X\mapsto -X$.
 
 The factor switching map composed with the antipodal map on both factors 
 \bee
 r:   (z,w) \mapsto (w^\perp,z^\perp)
 \eee
 induces the map $Q\mapsto  \text{id}- Q$ at the level of projectors and hence 
  \bee
 r: TS^2 \rightarrow TS^2, \quad Y\mapsto  -  Y.
 \eee
 When $|\eta|\neq 0$, it also  induces the map $X\mapsto \bar X$.
 The maps $s$ and $r$ commute, and generate the Vierergruppe. The product $a=rs=sr$ is the antipodal map on both factors, so 
\bee
a: (z,w) \mapsto (z^\perp,w^\perp). 
\eee
 It maps 
  \bee
 a: TS^2 \rightarrow TS^2, \quad Y\mapsto  - \bar  Y.
 \eee
 When $|\eta|\neq 0$, it also  induces the map $X\mapsto - \bar X$.

 The $SU(2)$ action on the homogeneous coordinates induces  the adjoint  $SO(3)$ action  on both $X$ and $Y$, so a rotation 
 \bee
 X \mapsto G X, \quad  Y\mapsto GY, 
 \eee
 of the complex vectors $X,Y \in \CC^3$ by $G\in SO(3)$.  This action commutes with the action of the Vierergruppe given above.

To make contact with the discussion in the main text we also require a  lift of the Vierergruppe to the  subgroup $\mathcal{D}_2$  of $SU(2)$.  This can be achieve by noting that the traceless complex  matrix $M$ can be expressed  in terms of the magnitudes of $\tx$ and $\xi$  as 
\bee
M =g(|\tx|\sigma_3 + i|\xi| \sigma_1) g^\dagger, \qquad g \in SU(2)
\eee
and that, for given $M$,  this fixes $g$  up to sign  when $\xi \neq 0$. Then one checks that, with the matrices $R_\ell=\exp(-i\frac{\pi}{\ell} \sigma_3)$ and $S=-i\sigma_2$ defined in \eqref{e1.30.9.19},  the right-multiplication by $S$,  
\bee
g\mapsto gS,
\eee
induces the map $s: (\tx,\xi)\mapsto (-\tx,-\xi)$ given in \eqref{sra}, and the right-multiplication by $R_\ell$, 
\bee
g\mapsto gR_\ell, 
\eee
induces the map $(\tx,\xi)\mapsto (\tx,R_m(2\pi/\ell )\xi)$, where we  used the notation defined after \eqref{Uact}. 
In particular, the right-multiplication by $R_2$ induces  the map $r$ given in  \eqref{e1.30.9.19}.  Identifying $(g,|\xi|)$  for $\xi\neq 0 $ with $g( |\xi|,0)^t \in \CC^2$, this is the  lift of the Vierergruppe to the binary dihedral group  $\mathcal{D}_2$ acting on $ \CC^2$ which is   used in the main text.

\section{Manifolds with corners}

\subsection{Definitions}

In this paper, we have used manifolds with corners (MWCs) systematically to
resolve singularities (for example the indeterminacy in the adiabatic
Gibbons--Hawking family $g_{\ve}$) and to obtain a smooth family of
$D_k$ ALF gravitational instantons on the Sen space through the
introduction of the space $\cW$.  We gather here
the most important definitions of the theory, the aim being to make
the rest of the paper more self-contained, rather than to give a
systematic development. For more details, the reader is referred to
\cite{daomwc} or the short summary in \cite{CCN}.  Another good
introduction is contained in \cite{gpaction}.

We give start with an extrinsic definition of MWC, referring to the
above references for the intrinsic approach. 

Let $M$ be a real manifold of dimension $n$. A subset $X\subset M$ is an
$n$-dimensional manifold with corners (MWC) if $X$ is a finite non-empty
intersection of `half-spaces' $H_j = \{\rho_j \geq 0\}$, where the
$\rho_j \in C^\infty(M)$ and $\rd \rho_j \neq 0$ on the zero-set of
$\rho_j$.  (Here $j$ lies in some finite index set $J$.)  This condition guarantees that $Z_j = \{\rho_j=0\}$ is a
smooth embedded submanifold of $M$ of codimension $1$.

We assume that the interior $X^{\circ}$ of $X$ (in $M$) is non-empty,
so that $X^{\circ}$ is an $n$-manifold.  We assume also that there is
no redundancy in the set $\{H_j\}$, so that the intersection of any
proper subset of the $H_j$ is strictly larger than $X$.
In particular $Y_j = X\cap Z_j$ is non-empty, and more importantly its
interior in $Z_j$ is a manifold of dimension of $n-1$.  It is
customary to suppose that the $Y_j$ are {\em connected}. (This can be
achieved by renumbering, and possibly shrinking the ambient manifold $M$.)
The $Y_j$ are called the {\em  boundary hypersurfaces} of $X$ and
$\rho_j$ is the {\em boundary defining function (bdf)} of $Y_j$.

The final technical point is that all non-empty intersections of the
boundary hypersurfaces should be cut out transversally 
by
the $\rho_j$---we do not want any pair of boundary hypersurfaces to
meet tangentially, for example.   Thus we insist that 
if
\begin{equation}
  \rho_j(p) = 0\mbox{ for }j = 1,\ldots,k
\end{equation}
then
\begin{equation}
  \rd \rho_1(p)\wedge \cdots \rd \rho_k(p) \neq 0.
\end{equation}
(It is to be understood that this holds for all subsets of $J$.)  It
follows in particular that at most $n$ boundary hypersurfaces can meet
in $X$.

\begin{ex}
If we have a finite collection of generically chosen half-spaces in
$\RR^n$, then their intersection (if non-empty) will be a manifold
with corners.  One may think of a general manifold with corners as a
`curvilinear version' of this, though of course there is no reason for
a general MWC to be homeomorphic to a ball.
\end{ex}

\begin{ex}
A closed (solid) octahedron in $\RR^3$ is {\em not} an example of
a MWC because four faces come together at each vertex, which is not
allowed in a MWC of dimension $3$.
\end{ex}

In this paper, all our MWCs have corners only up to codimension $2$:
in other words, there are non-empty intersections of certain pairs of
boundary hypersurfaces, but any intersection of three boundary
hypersurfaces is empty.   We now explain what is meant by {\em adapted
  coordinates} in this setting.

\begin{ex}\label{ex1.30.3.20}
  First of all, suppose that $p$ lies
on some boundary hypersurface $Y$ but is not in any intersection
$Y\cap Y'$ of boundary hypersurfaces.  Then adapted coordinates in a
neighbourhood $\Omega$ of $p$ in $X$ are $(\rho,y_1,\ldots,y_{n-1})$,
where the $y_j$ are local coordinates on $Y\cap\Omega$ centred at $p$:
thus $p$ is identified with the origin of this coordinate system.
\end{ex}
\begin{ex}\label{ex2.30.3.20}
Similarly, if $p\in Y\cap Y'$ and the bdfs of $Y$ and $Y'$ are
respectively $\rho$ and $\sigma$, adapted coordinates in a
neighbourhood $\Omega$ of $p$ in $X$ are
$(\rho,\sigma,y_1,\ldots,y_{n-2})$, where now the $y_j$ are local
coordinates on $Y\cap Y'\cap \Omega$ (which is an ordinary ($n-2$)-manifold
because there are no corners of codimension $3$ or more), centred at
$p$. Again, $p$ is identified with the origin of this coordinate system.
\end{ex}

\subsection{The $\bo$-tangent bundle}

The references mentioned above develop a suitable category of MWCs and
smooth maps.  In this development, the so-called $\bo$-tangent bundle
of MWCs is the `correct' replacement for the tangent bundle in
ordinary differential analysis. 

Let $X$ be a compact MWC, of dimension $n$.   The set of all smooth
vector fields which are tangent to all boundary faces of $X$ is
denoted $\cV_{\bo}(X)$.   There is a smooth vector bundle the
$\bo$-tangent bundle $T_{\bo}X$ with the property that
$C^\infty(X,T_{\bo}X) = \cV_{\bo}(X)$, where on the LHS we have {\em
  unrestricted} smooth sections over $X$.

Let us give a local description of $T_{\bo}X$ in the case of the two
examples above.

\begin{ex}\label{e3.30.3.20}  With the notation of
  Example~\ref{ex1.30.3.20}, a local basis for $T_{\bo}\Omega$ is given
  by the vector fields
\begin{equation}\label{e11.30.3.20}
  \rho\frac{\p}{\p \rho},
  \frac{\p}{\p y_1},\ldots,   \frac{\p}{\p y_{n-1}}
\end{equation}
and $\cV_{\bo}(\Omega)$ is the space of all linear combinations of
these vector fields with  coefficients in $C^\infty(\Omega)$. 
\end{ex}

\begin{ex}\label{e4.30.3.20}  With the notation of
  Example~\ref{ex2.30.3.20}, a local basis for $T_{\bo}\Omega$ is given
  by the vector fields
\begin{equation}\label{e12.30.3.20}
  \rho\frac{\p}{\p \rho},   \sigma\frac{\p}{\p \sigma},
  \frac{\p}{\p y_1},\ldots,   \frac{\p}{\p y_{n-2}}
\end{equation}
and $\cV_{\bo}(\Omega)$ is the space of all linear combinations of
these vector fields with  coefficients in $C^\infty(\Omega)$. 
\end{ex}

If we change adapted local coordinates in either of these examples, we
get new local bases \eqref{e11.30.3.20} or \eqref{e12.30.3.20} and it
is easy to see that these are related by transition functions in
$C^\infty(\Omega)$. This is one way to verify the existence of the
bundle $T_{\bo}X$.  More abstractly, one may invoke the Serre--Swan theorem.

For every point $p$ of $X$ there is an `evaluation map'
$T_{\bo,p}X \to T_p X$, but this is not an isomorphism if $p\in \p
X$. On the other hand the restriction of $T_{\bo}X$ to the interior
$X^{\circ}$ of $X$ is canonically isomorphic to $TX^{\circ}$. However,
if $v\in \cV_{\bo}(X)$, the smoothness of the coefficients up to and
including the boundary of $X$ means that $v|X^{\circ}$ will have some
controlled vanishing near each of the boundary hypersurfaces.

It is convenient to introduce the following notation.
\begin{notn}  Let $X$ be a MWC as above. 
Let $\Omega$ be an open set of $X$. Then $C^\infty(\Omega)$ is the
space of smooth functions on $\Omega$ (up to an including the
boundary of $X$, if $\Omega \cap \p X\neq\emptyset$).  The space
$C^\infty_0(\Omega)$ is the subspace of functions with compact support
in $\Omega$.  The support of such a function meets $\p X$
in a compact subset of $\p X \cap \Omega$ but need not be empty.   By
contrast, the subspace $\dot{C}^\infty(\Omega)$ consists of those
functions which vanish to all orders at the boundary hypersurfaces
with non-empty intersection with $\Omega$. If $f\in
\dot{C}^\infty(\Omega)$ we also say that $f$ is {\em rapidly decreasing} at
$\p \Omega$ and this has to be understood in the precise sense of the
previous sentence.  In the setting of Example~\ref{ex2.30.3.20}, we say
that $f\in C^\infty(\Omega)$ is rapidly decreasing at $Y$ if $f$
vanishes to all orders in the bdf $\rho$ of $Y$.
\end{notn}

\section{Analysis of the Laplacian of a strongly ALF space}

The essential analytical input we need for the proof of the main
theorem is a good understanding of the Poisson equation
\bee\label{e31.30.3.20}
\Delta_g u = f
\end{equation} 
on a strongly ALF manifold $(X,g)$ in the sense of Definition~\ref{strongALF}.
The geometric microlocal
approach to the analysis of elliptic operators in this setting was
first undertaken in \cite{MM_FB} and was further developed for the
purposes of Hodge theory in \cite{Vaillant} and \cite{HHM}.  The
strong ALF property of $g$ leads to stronger results than those in the
literature, so we describe our results here and explain how they are
obtained.

\begin{dfn}
Let $\rd \mu_{\bo}$ be any smooth $\bo$-density on $X$,
  and let $L^2_{\bo}(X)$ be the resulting space of $L^2$ functions. 
  For positive integers  $H^{n,m}_{\phi,\bo}(X)$ is the Sobolev space
  of functions on $X$ with $m$ $\bo$-derivatives and $n$
  $\phi$-derivatives in $L^2_{\bo}(X)$:
  $$
  \Diff^n_{\bo}(X)\Diff^m_{\phi}(X)u \subset L^2_{\bo}(X).
  $$
As previously, the subscript $\ei$ will be used to denote functions
which are essentially invariant near $\p X$.  Write $H_{\bo}^s(X)$ for
$H^{s,0}_{\bo,\phi}(X)$. 
\end{dfn}

\begin{rmk} Because of the algebraic properties of the $\bo$- and
  $\phi$- vector fields, $H^{n,m}_{\phi,\bo}(X)$ could also have been
  defined as the set of $u$ for which
$$
v_1\ldots v_{n+m} u \in L^2_{\bo}(X)
$$
for any collection of vector fields of which $n$ are in $\cV_{\bo}$
and $m$ are in $\cV_{\phi}$.
\end{rmk}

We now define a family of domains and ranges for $\Delta_g$ so that
\bee
\Delta_g : \cD_{\alpha,\beta,m+2}(X) \longrightarrow
\cR_{\alpha,\beta,m}(X)
\eee
is a bounded invertible linear mapping.  There is some flexibility in
the definition, and we choose to make the {\em range}
space as close as possible to a $\bo$-Sobolev space, albeit one in
which the invariant and non-invariant components with respect to the
$S^1$-action are weighted differently. 
For $m$ a positive integer and any numbers $\alpha>0$, $\alpha \not\in\ZZ$ and
$\beta> \alpha+2$, define
\begin{equation}\label{e11.3.4.20}
  \cR_{\alpha,\beta,m}(X) = \{\chi f_0 + f_1\}\mbox{ where }
f_0 \in \rho^{\alpha+2}H^m_{\bo}(V)\mbox{ and }\frac{\p 
    f_0}{\p\theta} =0,\; f_1 \in \rho^\beta H^m_{\bo}(X). 
\end{equation}
In this definition, $V$ is a collar neighbourhood of $\p X$, $\rho$ is
a bdf identically equal to $1/|x|$ in $V$ and $\chi$
is cut-off function with compact support in $V$ and identically $1$ in
a neighbourhood of $\p X$. We always take $\beta > \alpha+2$, so that the
invariant component decays more slowly than the non-invariant
component.

For the domain, define
\bee\label{e12.3.4.20}
\cD_{\alpha,\beta,m+2}(X) = \{\chi u_0 + u_1\}\mbox{ where }
u_0 \in \rho^{\alpha}H^{m+2}_{\bo}(V) + \cH_\alpha,\;
 \frac{\p u_0}{\p\theta} = 0,\; u_1 \in 
 \rho^\beta H^{2,m}_{\phi,\bo}(X)
\eee
where everything is already defined apart from $\cH_{\alpha}$. This is a finite-dimensional space of finite-order harmonic multipole 
expansions on $\RR^3$,
\begin{equation}
  \cH_{\alpha} = \left\{\sum_{j=1}^{[\alpha]} H_j(x)\right\}
\end{equation}
where $h_j(x)$ is homogeneous of degree $-j$ on $\RR^3\setminus
\{0\}$. Here if $\alpha \in (0,1)$, $\cH_\alpha=\{0\}$ by definition.

Both the domain and range are independent of the choice of cut-off
$\chi$. 
\begin{thm}
  Let the definitions be as above with $\alpha >0$,
  $\alpha\not\in\ZZ$, and
  $\beta > \alpha +2$.    Then there is a 
  bounded inverse of $\Delta_g$,
  \bee
  G:\cR_{\alpha,\beta,m} \to
  \cD_{\alpha,\beta,m+2},\;\; \Delta_g G = G\Delta_g = 1.
  \eee
  \label{t11.26.3.20}
  \end{thm}

If $f\in \dot{C}^\infty(X)$, then it lies in the intersection over all
$(\alpha,\beta,m)$ of $\cR_{\alpha,\beta,m}$, the solution $u=Gf$ lies
in all the $\cD_{\alpha,\beta,m}$, and so:

\begin{thm}\label{t11.28.3.20}
Let $f \in \dot{C}^\infty(X)$.  Then there exists unique $u$
\begin{equation}\label{e30.25.2.20}
 u \in  \rho C^\infty_{\ei}(X)
\end{equation}
solving $\Delta_g u = f$.
\end{thm}

The condition $\alpha\not\in \ZZ$ is needed to avoid the indicial
roots of the three-dimensional Laplacian acting on the invariant part
$u_0$ \cite{GreenBook}.  These indicial roots are precisely the integers in this case,
corresponding to the homogeneous solutions $\Delta H =0$ on
$\RR^3\setminus \{0\}$.

The essentially invariant part of $u$ in \eqref{e30.25.2.20} is
not merely 
smooth, but has an asymptotic expansion in homogeneous harmonic
functions (multipole expansion).  That is, there is a sequence of functions $H_j$ on
$\RR^3\setminus \{0\}$ such that for any $N$,
\begin{equation} \label{e23.21.11.20}
u - \sum_{j=1}^N H_j \in \rho^{N+1} C^\infty(X)\mbox{ near }\p X
\end{equation}
where $H_j$ is homogeneous of degree $-j$. 
Furthermore this equation can be differentiated any number of times
and remains valid.

\begin{rmk}  For a general ALF manifold, one cannot expect a smooth
(up to the boundary)  solution $u$ of the Poisson equation, even for
$f \in \dot{C}^\infty(X)$.  The most that can be expected is that $u$
will have a polyhomogeneous conormal expansion, cf.\ \cite[Prop.\
17]{HHM}.
\end{rmk}

\begin{proof}[Proof of Theorem~\ref{t11.26.3.20}]
We use the geometric microlocal approach, and assume that the reader
has some familiarity with \cite{MM_FB}.  

The inverse operator $G$ is constructed in stages, using the class
$\Psi^*_\phi(X)$ of $\phi$-pseudodifferential operators on $X$. These
operators have kernels whose structure is clearest on the `stretched
product' $X^2_{\phi}$, which is a certain blow-up of the cartesian
square $X^2 = X\times X$.   

The boundary faces of $X^2_{\phi}$ are the two front faces, 
$\ff_{\phi}$ and $\ff_{\bo}$, as well as the old boundary
hypersurfaces, which are the 
lifts of $\p X\times X$ and $X \times \p X $.  Boundary defining
functions will be denoted by $\rho_{\phi}, \rho_{\bo}, \rho, \rho'$.

We give a somewhat rough, local, description of $X^2_{\phi}$ 
in terms of a local asymptotic
Gibbons--Hawking chart $(x,\theta)$ near $\p X$.  On the interior $X^\circ \times
X^\circ$  (near $\p X \times \p X$) we have coordinates $(x,\theta,x',\theta')$.  In
$X^2_{\phi}$, a neighbourhood of the corner $\ff_{\bo} \cap (X\times \p 
X)$ corresponds to the region $|x| \gg 1$, $|x'|/|x|$ bounded
(by some $\delta<1$) and we
may take $\rho_{\bo} =1/|x|$ and $\rho' =|x|/|x'|$ near this corner.
Similarly, near $\ff_{\bo}\cap (\p X \cap X)$ we may take $\rho_{\bo}
= 1/|x'|$ and $\rho = |x'|/|x|$.

The front face $\ff_{\phi}$ is the total space of a fibration over $\p
X$, with fibres $\ol{\RR^3}\times S^1$, and a neighbourhood of
$\ff_{\phi}$ in $X^2_{\phi}$ corresponds to extending this to a
neighbourhood $V$ of $\p X$ in $X$.  Working locally in the base we
identify this with the product
$$
U\times S^1 \times \ol{\RR^3}\times S^1
$$
with local coordinates $(x',\theta', w,\psi)$. (To have this
trivialization, $x'$ must lie in a sufficiently small neighbourhood of a point
in $\p\ol{\RR^3}$.)     In these coordinates, the two projections
$X^2_{\phi} \to X$ 
 are
given by
\bee\label{e5.19.1.20}
\pi_1:(x',\theta',w,\psi) \mapsto (x'+w, \theta' + \psi),\;\;
\pi_2:(x',\theta',w,\psi) \mapsto (x', \theta').
\eee
(We could equally well have given a description of $\ff_{\phi}$ as a
fibration over the boundary of the first factor of $X$, rather than
the second.)

The intersection $\ff_{\phi}\cap \ff_{\bo}$ is
non-empty and $1/|w|$ and $1/|x'|$ are local defining functions for
$\ff_{\bo}$ and $\ff_{\phi}$ respectively, near this corner.

The space $\Psi^m_{\phi}(X)$ of $\phi$-pseudodifferential operators of
order $m$ consists of those operators whose Schwarz kernels, lifted to
$X^2_{\phi}$, are smooth away from the (lift of the) diagonal, rapidly
decreasing at all boundary faces apart from $\ff_{\phi}$, and with a
standard pseudodifferential singularity along the lifted diagonal 
$\{w=0,\psi=0\}$.  If $m=-\infty$, we have kernels in
$C^\infty(X^2_{\phi})$ with the same rapid decrease at all boundary
faces apart from $\ff_{\phi}$. 

Because of the circle-action near the boundary, we can introduce the 
subclass $\Psi^{m}_{\phi,\ei}(X)$ of essentially invariant operators
defined by insisting that $R$ be essentially invariant
with respect to the diagonal circle action on
$X^2_{\phi}$. Explicitly, a kernel in $\Psi^{m}_{\phi,\ei}(X)$ has the
(local) form
\bee \label{e2.19.11.20}
R(x',w,\psi) \mod \rho_{\phi}^\infty \Psi^{m}_{\phi}(X).
\eee
One of the achievements of \cite{MM_FB} is the proof that
$\Psi^*_\phi(X)$ is closed under composition.  This, and the
surjectivity of the $\phi$-symbol map, means that many of
the standard arguments involving pseudodifferential operators on
manifolds without boundary can be adapted to the study of
$\phi$-differential operators on $X$.

This is used in the first step in constructing $G$:

{\noindent \bf Step 1}:  There exists
$P \in \Psi^{-2}_{\phi,\ei}(X)$ and $R \in
\Psi^{-\infty}_{\phi,\ei}(X)$, formally self-adjoint, such that
\begin{equation}\label{e15.17.2.20}
P\Delta_g = 1 - R,\;\; \Delta_g P= 1- R.
\end{equation}
This is an adaptation of the usual symbolic argument in the theory of elliptic
pseudodifferential operators, see \cite{MM_FB}.  Since the metric is
strongly ALF, the operator $\Delta_g$ is essentially invariant.  An
examination of the usual proof shows that
$P$ (and hence $R$) can be chosen to be essentially
invariant as claimed.

The error term $R$ in \eqref{e15.17.2.20} is not compact $L^2 \to
L^2$, so this result does not yield a Fredholm result for
$\Delta_g$. We need to improve the error term to one which vanishes
at all boundary hypersurfaces.  Let $Q = Q(x',\theta',w,\psi)$ be
smooth on $X^2_{\phi}$ with support near $\ff_{\phi}$.  The Laplacian
$\Delta_g$ differs from the Laplacian $\Delta_h$  of the Gibbons--Hawking
metric $g_h$ by rapidly decreasing terms, so we have
\bee\label{e16a.19.11.20}
\Delta_g Q = \Delta_h Q + O(\rho_\phi^\infty)Q
\eee
where
\bee\label{e16.19.11.20}
\Delta_h Q = 
-\frac{1}{h(x'+w)}\wt{\nabla}^2_w Q -h(x'+w)\p_\psi^2Q
\eee
with
$$
\wt{\nabla}_j = \frac{\p}{\p w_j} - a_j(x'+w)\frac{\p}{\p\psi}
$$
(cf.\ \eqref{e2.7.8.18}).

We shall explain how to solve $\Delta_h Q = R$ (mod rapidly decreasing
terms).  Then replacing $R$ by $R+Q$ will give our improved
parametrix, from which Theorems~\ref{t11.26.3.20} and
\ref{t11.28.3.20} can be proved directly.

To save on notation, replace $R$ by its invariant part, so $R =
R(x',w,\psi)$.  Since we are working modulo $\rho_\phi^\infty$ in any
case, this will not cause any harm.  Now decompose $R = R_0 + R_1$
where $R_0$ is the zero Fourier-mode of $R$ with respect to $\psi$ and 
$R_1$ is the sum of the non-zero Fourier-modes.

{\noindent \bf Step 2:}  There exist smooth functions $Q_0 \in
\rho_{\bo}\rho\rho' C^{\infty}(X^2_{\phi})$ and $Q_1 \in
\Psi^{-\infty}_{\phi,\ei}(X)$ such that
\bee\label{e21.21.11.20}
\Delta_h Q_0 - R_0 \in (\rho_{\phi}\rho_{\bo}\rho)^\infty \rho' C^\infty(X^2_{\phi})
\eee
and
\bee\label{e22.21.11.20}
\Delta_h Q_1 - R_1 \in \dot{C}^\infty(X^2_{\phi})
\eee
(rapid decrease at all boundary hypersurfaces).
Both of $Q_0$ and $Q_1$ are supported near $\ff_{\phi}\cup \ff_{\bo}$
and $Q_0$ is $S^1\times S^1$-invariant (independent of $\psi$ as well
as $\theta'$).

Let us start with $Q_1$.  It is obtained order by order in powers of
$\rho_{\phi}$ starting with the leading term.  Let $q_1 =
Q_1|\ff_{\phi}$ and $r_1 = R_1|\ff_{\phi}$.  Solving $\Delta Q_1 =
R_1$ to leading order at $\ff_{\phi}$ means setting $x' = \infty$ in
\eqref{e16.19.11.20}.  This yields the equation
\bee\label{e1.31.12.20}
\Delta_{\RR^3\times S^1} q_1 = r_1
\eee
which can easily be solved using the Fourier transform.  If we denote
by  $\wh{f}(\eta,n)$ the Fourier transform with
respect to $w$ and $\psi$ of a function on $\RR^3\times S^1$, then we
may solve our equation by setting
$$
\wh{q}_1(\eta,n) = \left\{\begin{array}{l}
                            \frac{\wh{r}_1(\eta,n)}{|\eta|^2 + n^2},\;
                            n\neq 0;\\
0,\; n=0.\end{array}\right.
$$
Because $\wh{r}_1(\eta,0)=0$ (the zero Fourier-mode has been removed),
the inverse Fourier transform $q_1$ of $\wh{q}_1$ solves
\eqref{e1.31.12.20}.  
Because $r_1$ is smooth and rapidly decreasing for $|w| \to \infty$,
its Fourier coefficients are also smooth in $\eta$ and rapidly
decreasing in $\eta$ and $n$.   The same is therefore true
of the Fourier coefficients $\wh{q}_1(\eta,n)$, so 
$q_1$ is smooth and rapidly decreasing (and its zero Fourier-mode is
equal to $0$). 

Extend $q_1$ smoothly to a neighbourhood of $\ff_{\phi}$, calling the
result $\wt{q}_1$.  Then $\Delta_h \wt{q_1} - R_1 = O(\rho_\phi)$
and the error is still rapidly decreasing at all other boundary
hypersurfaces.   We can now
iterate, solving next for the coefficient of $\rho_\phi$ in the RHS of this
equation and proceeding order by order.  Invoking Borel's Lemma, we
find a $Q_1$ satisfying \eqref{e22.21.11.20}.

Now turn to the $S^1\times S^1$-invariant part $R_0$. We may suppose that this is
smooth and supported near $\ff_{\phi}$. Taking $Q_0$ to be invariant,
the equation $\Delta_h Q_0 = R_0$ becomes
\bee
\Delta_{w} Q_0 = h(x'+w) R_0(x',w),
\eee
where $\Delta_{w}$ is the standard euclidean Laplacian in the $w$
variables. 
Let $\delta>0$ be small and let $\chi(t)$ be a smooth cut-off function
equal $1$ for $t<\delta/2$ and equal to $0$ for $t>\delta$.  Define
\begin{equation}\label{e3.18.2.20}
Q_0(w,x') = \chi(\rho)\chi(\rho')\frac{1}{4\pi}\int 
  \frac{1}{|w-w''|}h(x'+w'')R_0(w'',x')\,\rd w''.
\end{equation}
This is the standard formula for solving the Poisson equation
in $\RR^3$ apart from the cut-offs, and because $R_0$ is rapidly
decreasing for $|w| \to \infty$, $Q_0$ is smooth in $w$ and decays like
$1/|w|$, for $|w| \to \infty$.  In fact $Q_0$ has a multipole
expansion (cf.\ \eqref{e23.21.11.20}) for
large $|w|$ the coefficients of which are smooth in the parameter
$x'$.  Recalling that $\rho_{\bo} 
= 1/|w|$ near $\ff_{\phi}\cap \ff_{\bo}$, we see that $Q_0 \in
\rho_{\bo} C^\infty$, at least in a neighbourhood of $\ff_{\phi}$. 

By the change of variables $w= x-x'$ (cf.\
\eqref{e5.19.1.20}), however, we obtain 
\bee\label{e1.21.11.20}
Q_0(x-x',x') = \chi(\rho)\chi(\rho')\frac{1}{4\pi}\int 
  \frac{1}{|x-x'-w''|}h(x'+w'')R_0(w'',x')\,\rd w''
\eee
and so we have an extension of $Q_0$ to $X^2_{\phi}$,
supported near $\ff_{\phi}\cup \ff_{\bo}$. One checks that
${Q}_0 \in  \rho_{\bo}\rho\rho' C^\infty$ as claimed in the
statement of Step 2.  Moreover,
\bee
\Delta_h {Q}_0 = \chi(\rho)\chi(\rho')R_0 + E
\eee
where $E$ comes from the commutator $[\Delta_h,\chi(\rho)]$ and is
thus supported near $X\times \p X$ (away from the corner) and vanishes
to first order in $\rho'$.  This completes the proof of Step 2.

{\noindent \bf Step 3:}  Let $Q = Q_0 + Q_1$ and let $P' = P + Q$.
Then
\bee\label{e25.21.11.20}
\Delta_g P' = 1 - R'
\eee
where the error term $R'$ is as in \eqref{e22.21.11.20}.  One shows
(as in \cite{MM_FB} and \cite{HHM}) that the operator $R'$ is compact on
$\cR_{\alpha,\beta,m}$ for $\alpha \in (0,1)$.  (The fact that there
is no indicial root of the euclidean Laplacian in $(0,1)$ is used
here, cf.\ \cite{GreenBook}.)
Theorem~\ref{t11.26.3.20} follows this from for $\alpha \in  (0,1)$, 
using integration by parts to prove that $\Delta_g$ is injective and
the Fredholm alternative.   If
$\alpha >1$, $\Delta_g$ is still injective on $\cD_{\alpha,\beta,m}$,
and if $f \in \cR_{\alpha,\beta,m}$, the result for $\alpha' \in (0,1)$
gives a unique solution $u$ of the equation in $\cD_{\alpha',\beta,m}$
of $\Delta_g u = f$. Using the fact that the invariant part of $f$
decays at the faster rate $\rho^{\alpha}$, one may deduce that the
leading terms of $u$ are in $\cH_{\alpha}$.  This completes the proof
of Theorem~\ref{t11.26.3.20}. 

\end{proof}

\begin{rmk}
In \cite{MM_FB}, what we have called Step 2 is addressed using the
`normal operator' of $\Delta_g$.  This is the Laplacian
$\Delta_{\RR^3\times S^1}$, acting in the $(w,\psi)$ variables,
i.e. on the fibres of $\ff_{\phi}$.  This operator appeared in our
argument to solve away the component $R_1$ of $R$ near $\ff_{\phi}$.
The normal operator is used systematically
in \cite{MM_FB} for the Fredholm theory of 
{\em fully elliptic} $\phi$-operators.  However, the Laplacian of a
$\phi$-metric is not fully elliptic and this is why the additional
argument was needed to deal with the $S^1\times S^1$-invariant part
$R_0$ of the error.
\end{rmk}

\markboth{\refname}{\refname}
\bibliography{dnbib}
\bibliographystyle{plainurl}
\end{document}